\documentclass[11pt, DIV10,a4paper]{article}
\usepackage{color}
\usepackage{float}
\usepackage[bf,small]{caption2}
\usepackage{comment}
\usepackage{amsmath}
\usepackage{bigints}
\usepackage{rotating, booktabs}
\usepackage{amsopn}
\usepackage{amsfonts}
\usepackage{amsthm}
\usepackage{amssymb}
\usepackage{amsbsy}
\usepackage{multirow}
\usepackage{slashbox}
\usepackage{natbib}
\usepackage{tocbibind}
\usepackage[a4paper,colorlinks,breaklinks,bookmarksopen,bookmarksnumbered]{hyperref}
\usepackage[refpage]{nomencl}
\usepackage{makeidx}
\usepackage{pst-all}
\usepackage{epsfig,psfrag}
\usepackage{graphicx}
\usepackage{amsmath}
\usepackage{epstopdf}
\usepackage{tabularx}
\usepackage{subfigure}
\usepackage{setspace}
\usepackage[mathscr]{euscript}
\usepackage[margin=1in]{geometry}
\usepackage{xr-hyper}
\usepackage{amsfonts} % Mathesymbole
\usepackage{calc}

\newcommand{\ueq}[1][]{%
  \if\relax\detokenize{#1}\relax
    \sbox0{$\underbrace{=}_{}$}%
    \mathrel{\mathmakebox[\wd0]{=}}
  \else
    \mathrel{\underbrace{=}_{\mathclap{#1}}}
  \fi}

\newcommand {\ctn}{\cite}

\usepackage{url}
\newcommand {\btheta}{\mbox{$\theta$}}

\newcommand {\bTheta}{\mbox{$\Theta$}}

\newcommand{\bX}{\mathbf X}
\newcommand{\by}{\mathbf y}
\newcommand{\bY}{\mathbf Y}
\newcommand{\bU}{\mathbf U}

\newtheorem{theorem}{Theorem}

\numberwithin{equation}{section}
\numberwithin{algo}{section}
\numberwithin{table}{section}
\numberwithin{figure}{section}

\newtheorem{remark}[theorem]{Remark}

\normalsize
\begin{document}

\title{\textbf{Convergence of Pseudo-Bayes Factors in Forward and Inverse Regression Problems}}
\author{Debashis Chatterjee$^{\dag}$ and Sourabh Bhattacharya$^{\dag, +}$ }
\date{}
\maketitle
\begin{center}
%$^{\dag}$  University of Chicago \\
$^{\dag}$ Indian Statistical Institute\\
$+$ Corresponding author:  \href{mailto: bhsourabh@gmail.com}{bhsourabh@gmail.com}
%\href{mailto: kshldey@gmail.com}{kshldey@gmail.com}
%,  \href{mailto: bhsourabh@gmail.com}{bhsourabh@gmail.com}\\
\end{center}

\begin{abstract}
In the Bayesian literature on model comparison, Bayes factors play the leading role. In the classical statistical literature, model selection criteria are often devised
used cross-validation ideas. Amalgamating the ideas of Bayes factor and cross-validation \ctn{Geisser79} created the pseudo-Bayes factor.
The usage of cross-validation inculcates several theoretical advantages, computational simplicity and numerical stability in Bayes factors 
as the marginal density of the entire dataset is replaced with products of cross-validation densities of individual data points. 
	
However, the popularity of pseudo-Bayes factors is still negligible in comparison with Bayes factors, with respect to both theoretical investigations and
practical applications. In this article, we establish almost sure exponential convergence of pseudo-Bayes factors for large samples under a general setup
consisting of dependent data and model misspecifications. We particularly focus on general parametric and nonparametric regression setups in both forward and inverse contexts.
In forward regression the goal is to predict the response given some observed value of the covariate and the rest of the data, while
in inverse regression the objective is to infer about unobserved covariate values from observed responses and covariates. For the Bayesian treatment that we
consider here, a prior for the unknown covariate value is needed. 

Depending upon forward and inverse regression ideas, our asymptotic theory manifests itself in terms of 
almost sure exponential convergence %of the average of the logarithm
of the pseudo-Bayes factor in terms of the Kullback-Leibler divergence rate or its integrated version, between the competing and the true models.
Our asymptotic theory encompasses general model selection, variable selection and combinations of both.

We illustrate our theoretical results with various examples, providing explicit calculations. We also supplement our asymptotic theory with simulation experiments
in small sample situations of Poisson log regression and geometric logit and probit regression, additionally addressing the variable selection problem. We consider
both linear and nonparametric regression modeled by Gaussian processes for our purposes. Our simulation results provide quite interesting insights into the usage of
pseudo-Bayes factors in forward and inverse setups.
\\[2mm]
\end{abstract}

\section{Introduction}
\label{sec:intro}

The Bayesian statistical literature on model selection is rich in its collection of innovative methodologies. Among them 
the most principled method of comparing different competing models seems to be offered by Bayes factors, through the ratio of the posterior and
prior odds associated with the models under comparison, which reduces to the ratio of the marginal densities of the data under the two models.  
To illustrate, let us consider the problem of comparing any two models $\mathcal M_1$ and $\mathcal M_2$ 
given data $\bY_n=\{y_1,y_2,\ldots,y_n\}$, where $n$ is the sample size. Let
$\Theta_1$ and $\Theta_2$ be the parameter spaces associated with $\mathcal M_1$ and $\mathcal M_2$, respectively. 
For $j=1,2$, let the likelihoods, priors 
and the marginal densities for the two models be $L_n(\theta_j|\mathcal M_j)$. %=f(\bY_n|\theta_j,\mathcal M_j)$, 
$\pi(\theta_j|\mathcal M_j)$ and $m(\bY_n|\mathcal M_j)=\int_{\Theta_j} L_n(\theta_j|\mathcal M_j)\pi(d\theta_j|\mathcal M_j)$,
respectively. Then the Bayes factor (BF) of model $\mathcal M_1$ against $\mathcal M_2$ is given by
\begin{equation}
	BF^{(n)}(\mathcal M_1,\mathcal M_2)=\frac{m(\bY_n|\mathcal M_1)}{m(\bY_n|\mathcal M_2)}.
	\label{eq:bf}
\end{equation}
The above formula follows directly from the coherent procedure of Bayesian hypothesis testing of one model versus the other. 
In view of (\ref{eq:bf}), $ BF^{(n)}(\mathcal M_1,\mathcal M_2)$ admits the interpretation as the quantification of the evidence of $\mathcal M_1$
against $\mathcal M_2$, given data $\bY_n$. A comprehensive account of BF and its various advantages are provided
in \ctn{Kass95}. BFs  have interesting asymptotic convergence properties. Indeed, recently \ctn{Chatterjee18}
establish the almost sure convergence theory of BF in the general setup that includes even dependent data and misspecified models.
Their result depends explicitly on the average Kullback-Leibler (KL) divergence between the competing and the true models. %even in such a general setup.

BFs are known to have several limitations. First, if the prior for the model parameter $\theta_j$ is improper, then the marginal density $m(\cdot|\mathcal M_j)$
is also improper and hence $m(\bY_n|\mathcal M_j)$ does not admit any sensible interpretation. Second, BFs suffer from the Jeffreys-Lindley-Bartlett paradox
(see \ctn{Jeffreys39}, \ctn{Lindley57}, \ctn{Bartlett57}, \ctn{Robert93}, \ctn{Villa15} for details and general discussions on the paradox). 
Furthermore, a drawback of BFs in practical applications is that the marginal density of the data $\bY_n$ is usually quite challenging to compute accurately, even with 
sophisticated simulation techniques based on importance sampling, bridge sampling and path sampling (see, for example, \ctn{Meng96}, \ctn{Gelman98}; see also \ctn{Gronau17} 
for a relatively recent tutorial and many relevant references), 
particularly when the posterior is far from normal and when the dimension of the parameter space is large. 
Moreover, the marginal density is usually extremely close to zero if $n$ is even moderately large.
This causes numerical instability in computation of the BF.

The problems of BFs regarding improper prior, Jeffreys-Lindley-Bartlett paradox, 
and general computational difficulties associated with the marginal density can be simultaneously alleviated 
if the marginal density $m(\bY_n|\mathcal M_j)$ for model $\mathcal M_j$ is replaced
with the product of leave-one-out cross-validation posteriors $\prod_{i=1}^n\pi\left(y_i|\bY_{n,-i},\mathcal M_j\right)$, where
$\bY_{n,-i}=\bY_n\backslash\{y_i\}=\{y_1,\ldots,y_{i-1},y_{i+1},\ldots,y_n\}$, and
\begin{equation}
	\pi\left(y_i|\bY_{n,-i},\mathcal M_j\right)=\int_{\Theta_j}f(y_i|\theta_j,y_1,\ldots,y_{i-1},\mathcal M_j)d\pi\left(\theta_j|\bY_{n,-i},\mathcal M_j\right)
	\label{eq:xval1}
\end{equation}
is the $i$-th leave-one-out cross-validation posterior density evaluated at $y_i$. 
In the above equation (\ref{eq:xval1}), $f(y_i|\theta_j,y_1,\ldots,y_{i-1},\mathcal M_j)$ is the density of $y_i$ given model parameters $\theta_j$ and $y_1,\ldots,y_{i-1}$;
$\pi\left(\theta_j|\bY_{n,-i},\mathcal M_j\right)$ is the posterior distribution of $\theta_j$ given $\bY_{n,-i}$. 
Viewing $\prod_{i=1}^n\pi\left(y_i|\bY_{n,-i},\mathcal M_j\right)$ as the surrogate for $m(\bY_n|\mathcal M_j)$, it seems 
reasonable to replace $BF^{(n)}(\mathcal M_1,\mathcal M_2)$ with the corresponding pesudo-Bayes factor (PBF) given by
\begin{equation}
	PBF^{(n)}(\mathcal M_1,\mathcal M_2)=\frac{\prod_{i=1}^n\pi\left(y_i|\bY_{n,-i},\mathcal M_1\right)}{\prod_{i=1}^n\pi\left(y_i|\bY_{n,-i},\mathcal M_2\right)}.
	\label{eq:pbf}
\end{equation}
In the case of independent observations, the above formula and the terminology ``pseudo-Bayes factor" seem to be first proposed by \ctn{Geisser79}. 
Their motivation for PBF did not seem to arise as providing solutions to the problems of BFs, however, 
but rather the urge to exploit the concept of cross-validation in Bayesian model selection, which had been proved to be 
indispensable for constructing model selection criteria in the classical statistical paradigm. 
Below we argue how this cross-validation idea helps solve the aforementioned problems of BFs.

First note that the posterior 
$\pi\left(\theta_j|\bY_{n,-i},\mathcal M_j\right)$ is usually proper even for improper prior for $\theta_j$ is $n$ is sufficiently large.
Thus, $\pi\left(y_i|\bY_{n,-i},\mathcal M_j\right)$ given by (\ref{eq:xval1}) is usually well-defined even for improper priors, unlike $m(\bY_n|\mathcal M_j)$. 
So, even though BF is ill-defined for improper priors, PBF is usually still well-defined.

Second, a clear theoretical advantage of PBF over BF is that PBF is immune to the problem of Jeffreys-Lindley-Bartlett paradox (see \ctn{Gelfand94} for example), while
BF is certainly not. 

Finally, PBF enjoys significant computational advantages over BF.
Note that straightforward Monte Carlo averages of $f(y_i|\theta_j,y_1,\ldots,y_{i-1},\mathcal M_j)$ over realizations of $\theta$ obtained from 
$\pi\left(\theta|\bY_{n,-i},\mathcal M_j\right)$ by simulation techniques is sufficient to ensure good estimates of 
the cross-validation posterior density $\pi\left(y_i|\bY_{n,-i},\mathcal M_j\right)$.
%Th estimate is also numerically stable since $\pi\left(y_i|\bY_{n,-i},\mathcal M_j\right)$ is the density of $y_i$ only, 
Since $\pi\left(y_i|\bY_{n,-i},\mathcal M_j\right)$ is the density of $y_i$ individually, the estimate is also numerically stable compared to estimates of
$m(\bY_n|\mathcal M_j)$. Hence, the sum of logarithms of the estimates of $\pi\left(y_i|\bY_{n,-i},\mathcal M_j\right)$, for $i=1,\ldots,n$, results in
quite accurate and stable estimates of $\log\left[\prod_{i=1}^n\pi\left(y_i|\bY_{n,-i},\mathcal M_j\right)\right]$.
In other words, PBF is far simpler to compute accurately than BF and is numerically far more stable and reliable.

In spite of the advantages of PBF over BF, it seems to be largely ignored in the statistical literature, both theoretically and application-wise.  
Some asymptotic theory of PBF has been attempted by \ctn{Gelfand94} using independent observations, Laplace approximations and 
some essentially ad-hoc simplifying approximations and arguments. Application of PBF has been considered in \ctn{Bhattacharya08} for demonstrating the superiority
of his new Bayesian nonparametric Dirichlet process model over the traditional Dirichlet process mixture model. But apart from these works we are not aware
of any other significant research involving PBF.

In this article, we establish the asymptotic theory for PBF in the general setup consisting of dependent observations, model misspecifications as well as covariates;
inclusion of covariates also validates our asymptotic theory in the variable selection framework. Judiciously exploiting the posterior convergence treatise
of \ctn{Shalizi09} we prove almost sure exponential convergence of PBF in favour of the true model, the convergence explicitly depending upon 
the KL-divergence rate from the true model.
For any two models different from the true model, we prove almost sure exponential convergence of PBF in favour of the better model, where the convergence 
depends explicitly upon the difference between KL-divergence rates from the true model. Thus, our PBF convergence results agree with the BF convergence
results established in \ctn{Chatterjee18}.

An important aspect of our PBF research involves establishing its convergence properties even for ``inverse regression problems", and even if one of the two competing models
involve ``inverse regression" and the other ``forward regression". We distinguish forward and inverse regression as follows. 
In forward regression problems the goal is to predict the response from a given covariate value and the rest of the
data. On the other hand, in inverse regression unknown values of the covariates are to be predicted given the observed response and the rest of the data. 
Crucially, Bayesian inverse regression problems require priors on the covariate values to be predicted. 
In our case, the inverse regression setup has been motivated by the quantitative palaeoclimate reconstruction problem where `modern data' 
consisting of multivariate counts of species
are available along with the observed climate values. Also available are fossil assemblages of the same species, but deposited in lake sediments 
for past thousands of years. This is the fossil species data. However, the past climates corresponding to the fossil species data are unknown, and it is of interest
to predict the past climates given the modern data and the fossil species data. Roughly, the species composition are regarded as functions
of climate variables, since in general ecological terms, variations in climate drives variations in species, but not vice versa.
However, since the interest lies in prediction of climate variables,  the inverse nature of the problem is clear. The past climates, which must be regarded as random variables,
may also be interpreted as {\it unobserved covariate values}. It is thus natural to put a prior probability distribution on the unobserved covariate values.
Various other examples of inverse regression problems are provided in \ctn{Chatterjee17}.

In this article, we consider two setups of inverse regression and establish almost sure exponential convergence of PBF 
in general inverse regression for both the setups. These include situations where one of the competing models involve
forward regression and the other is associated with inverse regression. 

We illustrate our asymptotic results with various theoretical 
examples in both forward and inverse regression contexts, including forward and inverse variable selection problems.
We also follow up our theoretical investigations with simulation experiments in small samples involving Poisson and geometric forward and inverse 
regression models with relevant link functions and both linear regression and nonparametric regression, the latter modeled by Gaussian processes.
We also illustrate variable selection in the aforementioned setups with two different covariates. The results that we obtain are quite encouraging and
illuminating, providing useful insights into the behaviour of PBF for forward and inverse parametric and nonparametric regression.	

The roadmap for the rest of our paper is as follows. We begin our progress by discussing and formalizing the relevant aspects of 
forward and inverse regression problems and the associated pseudo-Bayes factors in Section \ref{sec:prelims}.
Then in Section \ref{sec:shalizi_briefing} we include a brief overview of Shalizi's approach to treatment of posterior convergence which we usefully
exploit for our treatise of PBF asymptotics; further details are provided in Appendix \ref{subsec:assumptions_shalizi}.
Convergence of PBF in the forward regression context is established in Section \ref{sec:conv_forward}, while in Sections \ref{sec:pbf_inverse_1st} and 
\ref{sec:pbf_inverse_2nd} we establish convergence of PBF in the two setups related to inverse regression.
In Sections \ref{sec:illustrations_forward} and \ref{sec:illustrations_inverse} we provide theoretical illustrations of 
PBF convergence in forward and inverse setups, respectively, with various examples including variable selection.
Details of our simulation experiments with small samples involving Poisson and geometric linear and Gaussian process regression for relevant link functions, under both
forward and inverse setups, are reported in Section \ref{sec:simstudy}, which also includes experiments on variable selection.
Finally, we summarize our contributions and provide future directions in Section \ref{sec:conclusion}.

%The rest of our paper is structured as follows. The general premise of our inverse regression model, LOO-CV and the IRD approach are described in Section \ref{sec:prelims}.  
%General consistency issues of the same are discussed in Section \ref{sec:discussion_consistency_IRD}.
%We propose an appropriate prior for $\tilde x_i$ and investigate its properties in Section \ref{sec:prior}, and in Section \ref{sec:consistency_loo_cv}
%prove consistency of the LOO-CV posteriors under reasonably mild conditions.
%Relating consistency of the LOO-CV posteriors, we prove consistency of the IRD approach in Section \ref{sec:consistency_IRD}.
%In Section \ref{sec:discussion_results} we provide a discussion on the issues and applicability of our asymptotic theory in various inverse regression contexts 
%and in Section \ref{sec:simstudy}, we illustrate our asymptotic theory with simulation studies. Finally, we make concluding remarks in Section \ref{sec:conclusion}.

\section{Preliminaries and general setup for forward and inverse regression problems}
\label{sec:prelims}

Let us first consider the forward regression setup.
\subsection{Forward regression problem}
\label{subsec:forward}
For $i=1,\ldots,n$, let observed response $y_i$ be related to observed covariate $x_i$ through
\begin{equation}
	y_1\sim f(\cdot|\theta,x_1)~\mbox{and}~y_i\sim f(\cdot|\theta,x_i,\bY^{(i-1)})~\mbox{for}~i=2,\ldots,n,
	\label{eq:forward0}
\end{equation}
where for $i=2,\ldots,n$, $\bY^{(i)}=\{y_1,\ldots,y_i\}$ and $f(\cdot|\theta,x_1)$, $f(\cdot|\theta,x_i,\bY^{(i-1)})$ are known densities depending upon 
(a set of) parameters $\theta\in\Theta$, where $\Theta$ is the parameter space, which may be infinite-dimensional. 
For the sake of generality, we shall consider $\theta=(\eta,\xi)$, where $\eta$ is a function of the covariates, which we more explicitly denote as $\eta(x)$. 
The covariate $x\in\mathcal X$, $\mathcal X$ being the space of covariates. The part $\xi$ of $\eta$ will be assumed to consist of other parameters, such as the unknown
error variance. For Bayesian forward regression problems, some prior needs to be assigned on the parameter space $\Theta$.
For notational convenience, we shall denote $f(\cdot|\theta,x_1)$ by $f(\cdot|\theta,x_1,\bY^{(0)})$, so that we can represent (\ref{eq:forward0}) more conveniently
as 
\begin{equation}
	y_i\sim f(\cdot|\theta,x_i,\bY^{(i-1)})~\mbox{for}~i=1,\ldots,n.
	\label{eq:forward1}
\end{equation}
%We assume that given $\theta$, $y_i$ are conditionally independent.

\subsubsection{Examples of the forward regression setup}
\label{subsubsec:model_examples_forward}
\begin{itemize}
%\item[(i)] $y_i=\alpha+\beta x_i+\epsilon_i$, where $\theta=(\alpha,\beta)$ and $\epsilon_i$ are $iid$ Gaussian errors.
\item[(i)] $y_{i}\sim Bernoulli(p_i)$, where $p_i=H\left(\eta(x_i)\right)$, where $H$ is some appropriate link function and $\eta$ is some 
function with known or unknown form. For known, suitably parameterized form, the model is parametric. If the form of $\eta$ is unknown, 
one may model it by a Gaussian process, assuming adequate smoothness of the function.
\item[(ii)] $y_{i}\sim Poisson(\lambda_i)$, where $\lambda_i=H\left(\eta(x_i)\right)$, where $H$ is some appropriate link function and $\eta$ is 
some function with known (parametric) or unknown (nonparametric) form. Again, in case of unknown form of $\eta$, 
the Gaussian process can be used as a suitable model under sufficient smoothness
assumptions.
\item[(iii)] $y_{i}=\eta(x_i)+\epsilon_{i}$, where $\eta$ is a parametric or nonparametric function and
$\epsilon_{i}$ are $iid$ Gaussian errors. In particular, $\eta(x_i)$ may be a linear regression function, that is, $\eta(x_i)=\beta'x_i$, where
$\beta$ is a vector of unknown parameters.
Non-linear forms of $\eta$ are also permitted. 
Also, $\eta$ may be a reasonably smooth function of unknown form, modeled by some appropriate Gaussian process.		
\end{itemize}

\subsection{Forward pseudo-Bayes factor}
\label{subsec:pbf_forward}
Letting $\bY_{n}=\left\{y_{i}:i=1,\ldots,n\right\}$, $\bX_n=\{x_i:i=1,\ldots,n\}$, $\bY_{n,-i}=\bY_n\backslash\{y_i\}$ and $\bX_{n,-i}=\bX_n\backslash\{x_i\}$,
let $\pi(y_i|\bY_{n,-i},\bX_n,\mathcal M)$ denote the posterior density at $y_i$, given data $\bY_{n,-i}$, $\bX_n$ and model $\mathcal M$.
Let the density of $y_i$ given $\theta$ and $x_i$ under model $\mathcal M$ be denoted by $f(y_i|\theta,x_i,\bY^{(i-1)}\mathcal M)$. Then
note that
\begin{equation}
	\pi(y_i|\bY_{n,-i},\bX_n,\mathcal M)=\int_{\Theta} f(y_i|\theta,x_i,\bY^{(i-1)},\mathcal M)d\pi(\theta|\bY_{n,-i},\bX_{n,-i},\mathcal M),
\label{eq:post_forward1}
\end{equation}
where 
\begin{equation}
	\pi(\theta|\bY_{n,-i},\bX_{n,-i},\mathcal M)\propto\pi(\theta)\prod_{j\neq i;j=1}^{n}f(y_j|\theta,x_j,\bY^{(j-1)},\mathcal M).
\label{eq:post_forward2}
\end{equation}
For any two models $\mathcal M_1$ and $\mathcal M_2$, the forward pseudo Bayes factor (FPBF) of $\mathcal M_1$ against $\mathcal M_2$ 
based on the cross-validation posteriors of the form (\ref{eq:post_forward1}) is defined as follows:
\begin{equation}
	FPBF^{(n)}(\mathcal M_1,\mathcal M_2)
		=\frac{\prod_{i=1}^n\pi(y_i|\bY_{n,-i},\bX_n,\mathcal M_1)}{\prod_{i=1}^n\pi(y_i|\bY_{n,-i},\bX_n,\mathcal M_2)},
		\label{eq:pbf_forward}
\end{equation}
and we are interested in studying the limit $\underset{n\rightarrow\infty}{\lim}~\frac{1}{n}\log FPBF^{(n)}(\mathcal M_1,\mathcal M_2)$ for almost all data sequences.

%Suppose that the true data-generating parameter $\theta_0$ is not contained in $\Theta$, the parameter space considered.
%This is a case of misspecification that we must incorporate in our convergence theory of PBF. Our PBF asymptotics draws on posterior convergence theory for 
%(possibly infinte-dimensional) parameters that also allows misspecification. In this regard, the approach presented in \ctn{Shalizi09} seems to be very appropriate.
%Before we proceed, we first provide a brief overview of this approach, which we conveniently exploited for our purpose.

\subsection{Inverse regression problem: first setup}
\label{subsec:inverse}

%The distribution $f_\theta$, parameter $\theta$, the parameter and the covariate space remain the same as in the forward regression setup. 
%Unlike in Bayesian forward regression problems where a prior needs to be assigned only to the unknown parameter $\theta$, a prior is also required for
%$\tilde x$, the unknown covariate observation associated with known response $\tilde y$, say. Given the entire dataset and $\tilde y$, the problem
%in inverse regression problems is to predict $\tilde x$. Hence, in the Bayesian inverse setup, a prior on $\tilde x$ is necessary.
%In (\ref{model}), $f_\theta$ is a known distribution depending upon 
%(a set of) parameters $\theta\in\Theta$, where $\Theta$ is the parameter space, which may be infinite-dimensional. 
%For the same of generality, we shall consider $\theta=(\eta,\xi)$, where $\eta$ is a function of the covariates, which we more explicitly denote as $\eta(x)$, 
%where $x\in\mathcal X$, $\mathcal X$ being the space of covariates. The part $\xi$ of $\eta$ will be assumed to consist of other parameters, such as the unknown
%error variance.

In inverse regression, the basic premise remains the same as in forward regression detailed in Section \ref{subsec:forward}. In other words,
the distribution $f(\cdot|\theta,x_i,\bY^{(i-1)})$, parameter $\theta$, the parameter and the covariate space remain the same as in the forward regression setup. 
However, unlike in Bayesian forward regression problems where a prior needs to be assigned only to the unknown parameter $\theta$, a prior is also required for
$\tilde x$, the unknown covariate observation associated with known response $\tilde y$, say. Given the entire dataset and $\tilde y$, the problem
in inverse regression is to predict $\tilde x$. Hence, in the Bayesian inverse setup, a prior on $\tilde x$ is necessary. Given model $\mathcal M$
and the corresponding parameters $\theta$, we denote such prior
by $\pi(\tilde x|\theta,\mathcal M)$. For Bayesian cross-validation in inverse problems it is pertinent to successively leave out $(y_i,x_i)$; $i=1,\ldots,n$,
and compute the posterior predictive distribution $\pi(\tilde x_i|\bY_n,\bX_{n,-i})$, from $y_i$ and the rest of the data $(\bY_{n,-i},\bX_{n,-i})$
(see \ctn{Bhatta07}).
But these posteriors are not useful for Bayes of pesudo-Bayes factors even for inverse regression setups. 
%although the prior $\pi(\tilde x_i|\theta,\mathcal M)$
%plays a significant role in inverse model comparison with Bayes of pesudo-Bayes factors. 
The reason is that the Bayes factor for inverse regression is still the ratio of posterior odds and prior odds associated with the competing models, which
as usual translates to the ratio of the marginal densities of the data under the two competing models. The marginal densities depend upon the prior 
for $(\theta,\tilde x)$, however, under the competing models. The pseudo-Bayes factor for inverse models is then the ratio of products of the cross-validation
posteriors of $y_i$, where $\theta$ and $\tilde x_i$ are marginalized out. Details of such inverse cross-validation posteriors and the definition of
pseudo-Bayes factors for inverse regression are given below.

\subsubsection{Inverse pseudo-Bayes factor in this setup}
\label{subsubsec:ibpf1}
In the inverse regression setup, first note that
\begin{align}
&\pi(\tilde x_i,\theta|\bY_{n,-i},\bX_{n,-i},\mathcal M)\notag\\
	&=\frac{\pi(\tilde x_i,\theta|\mathcal M)\prod_{j\neq i;j=1}^nf(y_j|\theta,x_j,\bY^{(j-1)},\mathcal M)}
	{\int_{\mathcal X}\int_{\Theta} d\pi(u,\psi)\prod_{j\neq i;j=1}^nf(y_j|\psi,x_j,\bY^{(j-1)},\mathcal M)}\notag\\
	&=\frac{\pi(\tilde x_i|\theta,\mathcal M)\pi(\theta|\mathcal M)\prod_{j\neq i;j=1}^nf(y_j|\theta,x_j,\bY^{(j-1)},\mathcal M)}
	{\int_{\mathcal X}\int_{\Theta}d\pi(u|\psi,\mathcal M)d\pi(\psi|\mathcal M)\prod_{j\neq i;j=1}^nf(y_j|\psi,x_j,\bY^{(j-1)},\mathcal M)}\notag\\
	&=\frac{\pi(\tilde x_i|\theta,\mathcal M)\pi(\theta|\mathcal M)\prod_{j\neq i;j=1}^nf(y_j|\theta,x_j,\bY^{(j-1)},\mathcal M)}
	{\int_{\Theta} d\pi(\psi|\mathcal M)\prod_{j\neq i;j=1}^nf(y_j|\psi,x_j,\bY^{(j-1)},\mathcal M)}
	=\pi(\tilde x_i|\theta,\mathcal M)\pi(\theta|\bY_{n,-i},\bX_{n,-i},\mathcal M).
	\label{eq:inv1}
\end{align}
Using (\ref{eq:inv1}) we obtain
\begin{align}
	\pi(y_i|\bY_{n,-i},\bX_{n,-i},\mathcal M)&=\int_{\mathcal X}\int_{\Theta} f(y_i|\theta,\tilde x_i,\bY^{(i-1)},\mathcal M)
	d\pi(\tilde x_i,\theta|\bY_{n,-i},\bX_{n,-i},\mathcal M),\notag\\
	&=\int_{\Theta}g(\bY^{(i)},\theta,\mathcal M)d\pi(\theta|\bY_{n,-i},\bX_{n,-i},\mathcal M),
%	&=\int_{\mathcal X}\int_{\Theta} f(y_i|\theta,\tilde x_i,\mathcal M)d\pi(\theta|\bY_{n,-i},\bX_{n,-i},\mathcal M)d\pi(\tilde x_i|\theta,\mathcal M),
\label{eq:post_inverse1}
\end{align}
where 
\begin{equation}
	g(\bY^{(i)},\theta,\mathcal M)=\int_{\mathcal X}f(y_i|\theta,\tilde x_i,\bY^{(i-1)},\mathcal M)d\pi(\tilde x_i|\theta,\mathcal M),
	\label{eq:g1}
\end{equation}
and $\pi(\theta|\bY_{n,-i},\bX_{n,-i},\mathcal M)$ is the same as (\ref{eq:post_forward2}).
%\begin{equation}
%	\pi(\theta|\bY_{n,-i},\bX_{n,-i},\mathcal M)\propto\pi(\theta)\prod_{j\neq i;j=1}^{n}f(y_j|\theta,x_j,\mathcal M).
%\label{eq:post_inverse2}
%\end{equation}
For any two models $\mathcal M_1$ and $\mathcal M_2$, the inverse pseudo Bayes factor (IPBF) of $\mathcal M_1$ against $\mathcal M_2$ based on cross-validation
posteriors of the form (\ref{eq:post_inverse1}) is given by
	\begin{equation}
		IPBF^{(n)}(\mathcal M_1,\mathcal M_2)
		=\frac{\prod_{i=1}^n\pi(y_i|\bY_{n,-i},\bX_{n,-i},\mathcal M_1)}{\prod_{i=1}^n\pi(y_i|\bY_{n,-i},\bX_{n,-i},\mathcal M_2)},
		\label{eq:pbf_inverse}
	\end{equation}
and our goal is to investigate $\underset{n\rightarrow\infty}{\lim}~\frac{1}{n}\log IPBF^{(n)}(\mathcal M_1,\mathcal M_2)$ for almost all data sequences.

\subsection{Inverse regression problem: second setup}
\label{subsec:inverse2}
In the inverse regression context, we consider another setup under which \ctn{Chat20} establish consistency of the inverse cross-validation posteriors 
of $\tilde x_i$. Here we consider experiments with covariate observations $x_1, x_2, \ldots, x_n $ along with responses 
$\bY_{nm}=\{y_{ij} :i=1,\ldots,n, j=1,\ldots,m\}$.
In other words, the experiment considered here will allow us to have $m$ samples of responses $\by_i=\{y_{i1}, y_{i2}, \ldots, y_{im}\}$ against each covariate observation
$x_{i}$, for $i=1, 2,\ldots, n$. Again, both $x_i$ and $y_{ij}$ are allowed to be multidimensional. Let $\bY_{nm,-i}=\bY_{nm}\backslash\{\by_i\}$.
%We shall investigate the large sample scenario where both $m, n\rightarrow \infty$. 
%We will denote the covariate space along with assumption of 
%compactness (Assumption \ref{cov1}) by  $\mathfrak{X}$  and response space  by $\mathfrak{Y} \subset \Re$.  

For $i=1,\ldots,n$ %and $j=1,\ldots,m$, 
consider the following general model setup: %as \eqref{model} for $i \in \{1, 2, \cdots, n\}, j \in \{1, 2, \cdots, m\}$.
conditionally on $\theta$, $x_i$ and $\bY^{(i-1)}_j=\{y_{1j},\ldots,y_{i-1,j}\}$,
\begin{eqnarray}\label{model}
\begin{aligned}
	& y_{ij} \sim f\left(\cdot|\theta,x_i,\bY^{(i-1)}_j\right);~j=1,\ldots,m, % f_{\theta}\left(x_i\right), 
%& X_i\sim Q.
\end{aligned}
\end{eqnarray}
%and that for every $i\geq 1$, $y_{i1},\ldots,y_{im}$ are conditionally independent, given $\theta$ and $x_i$. 
independently, where $f(\cdot|\theta,x_1,\bY^{(0)})=f(\cdot|\theta,x_1)$ as before.

%\subsubsection{Examples of the inverse regression setup}
%\label{subsubsec:model_examples_inverse}
%For the inverse regression setup, the examples in Section \ref{subsubsec:model_examples_forward} can be generalized as follows:
%\begin{itemize}
%%	\item[(i)] $y_{ij}=\alpha+\beta x_i+\epsilon_{ij}$, where $\theta=(\alpha,\beta)$ and $\epsilon_{ij}$ are $iid$ Gaussian errors.
%\item[(i)] $y_{ij}\sim Bernoulli(p_i)$, where $p_i=H\left(\eta(x_i)\right)$, where $H$ is some appropriate link function and $\eta$ is some 
%function with known or unknown form. For known, suitably parameterized form, the model is parametric. If the form of $\eta$ is unknown, 
%one may model it by a Gaussian process, assuming adequate smoothness of the function.
%\item[(ii)] $y_{ij}\sim Poisson(\lambda_i)$, where $\lambda_i=H\left(\eta(x_i)\right)$, where $H$ is some appropriate link function and $\eta$ is 
%some function with known (parametric) or unknown (nonparametric) form. Again, in case of unknown form of $\eta$, 
%the Gaussian process can be used as a suitable model under sufficient smoothness
%assumptions.
%\item[(iii)] $y_{ij}=\eta(x_i)+\epsilon_{ij}$, where $\eta$ is a parametric or nonparametric function and
%$\epsilon_{ij}$ are $iid$ Gaussian errors. In particular, $\eta(x_i)$ may be a linear regression function, that is, $\eta(x_i)=\beta'x_i$, where
%$\beta$ is a vector of unknown parameters.
%Non-linear forms of $\eta$ are also permitted. 
%Also, $\eta$ may be a reasonably smooth function of unknown form, modeled by some appropriate Gaussian process.		
%\end{itemize}

\subsubsection{Prior for $\tilde x_i$}
\label{sec:prior}
Following \ctn{Chat20}, we consider the following prior for $\tilde x_i$: given $\theta$,
\begin{equation}
	\tilde x_i\sim U\left(B_{im}(\theta)\right),
\label{eq:prior_x}
\end{equation}
the uniform distribution on 
\begin{equation}
	B_{im}(\theta)=\left(\left\{x:H\left(\eta(x)\right)\in \left[\bar y_i-\frac{cs_i}{\sqrt{m}},\bar y_i+\frac{cs_i}{\sqrt{m}}\right]\right\}\right),
\label{eq:set1}
\end{equation}
where $H$ is some suitable transformation of $\eta(x)$.
In (\ref{eq:set1}), $\bar y_i=\frac{1}{m}\sum_{j=1}^my_{ij}$ and $s^2_i=\frac{1}{m-1}\sum_{j=1}^m(y_{ij}-\bar y_i)^2$, and $c\geq 1$ is some constant.
We denote this prior by $\pi(\tilde x_i|\eta)$. \ctn{Chat20} show that the density or any probability 
associated with $\pi(\tilde x_i|\eta)$ is continuous with respect to $\eta$.

\subsubsection{Examples of the prior}
\label{subsubsec:illustrations_prior}
\begin{itemize}
	\item[(i)] $y_{ij}\sim Poisson(\theta x_i)$, where $\theta>0$ and $x_i>0$ for all $i$. Here, under the prior $\pi(\tilde x_i|\theta)$, 
		$\tilde x_i$ has uniform distribution on the set 
		$B_{im}(\theta)=\left\{x>0:\frac{\bar y_i-\frac{cs_i}{\sqrt{m}}}{\theta}\leq x\leq \frac{\bar y_i+\frac{cs_i}{\sqrt{m}}}{\theta}\right\}$.
	\item[(ii)] $y_{ij}\sim Poisson(\lambda_i)$, where $\lambda_i=\lambda(x_i)$, with $\lambda(x)=H(\eta(x))$. Here $H$ is a known, one-to-one, 
		continuously differentiable function and $\eta(\cdot)$ is an unknown function modeled by Gaussian process.
		Here, the prior for $\tilde x_i$ is the uniform distribution on $$B_{im}(\eta)=\left\{x:\eta(x)\in 
		H^{-1}\left\{\left[\bar y_i-\frac{cs_i}{\sqrt{m}},\bar y_i+\frac{cs_i}{\sqrt{m}}\right]\right\}\right\}.$$
	\item[(iii)] $y_{ij}\sim Bernoulli(p_i)$, where $p_i=\lambda(x_i)$, with $\lambda(x)=H(\eta(x))$. Here $H$ is a known, increasing,
		continuously differentiable, cumulative distribution function and $\eta(\cdot)$ is an unknown function modeled by some appropriate Gaussian process.
		Here, the prior for $\tilde x_i$ is the uniform distribution on $B_{im}(\eta)=\left\{x:\eta(x)\in 
		H^{-1}\left\{\left[\bar y_i-\frac{cs_i}{\sqrt{m}},\bar y_i+\frac{cs_i}{\sqrt{m}}\right]\right\}\right\}$.
	\item[(iv)] $y_{ij}=\eta(x_i)+\epsilon_{ij}$, where $\eta(\cdot)$ is an unknown function modeled by some appropriate Gaussian process,
		and $\epsilon_{ij}$ are $iid$ zero-mean Gaussian noise with variance $\sigma^2$. 
		Here, the prior for $\tilde x_i$ is the uniform distribution on $B_{im}(\eta)=\left\{x:\eta(x)\in 
		\left[\bar y_i-\frac{cs_i}{\sqrt{m}},\bar y_i+\frac{cs_i}{\sqrt{m}}\right]\right\}$.
		If $\eta(x_i)=\alpha+\beta x_i$, then the prior for $\tilde x_i$ is the uniform distribution on $[a,b]$, where 
		%$\left[\frac{\bar y_i-\alpha}{\beta}-\frac{cs_i}{\sqrt{m}},\frac{\bar y_i-\alpha}{\beta}+\frac{cs_i}{\sqrt{m}}\right]$.
		$a=\min\left\{\frac{\bar y_i-\frac{cs_i}{\sqrt{m}}-\alpha}{\beta},\frac{\bar y_i+\frac{cs_i}{\sqrt{m}}-\alpha}{\beta}\right\}$
		and $b=\max\left\{\frac{\bar y_i-\frac{cs_i}{\sqrt{m}}-\alpha}{\beta},\frac{\bar y_i+\frac{cs_i}{\sqrt{m}}-\alpha}{\beta}\right\}$.
\end{itemize}
Further examples of the prior in various other inverse regression models are provided in Sections \ref{sec:illustrations_inverse} and 
\ref{sec:simstudy}. %\ref{subsec:poisson_vs_geometric_vs}.

\subsubsection{Inverse pseudo-Bayes factor in this setup}
\label{subsubsec:ibpf2}
For any two models $\mathcal M_1$ and $\mathcal M_2$ we define inverse pseudo-Bayes factor for model $\mathcal M_1$ against model $\mathcal M_2$, for any $k\geq 1$, as
\begin{equation}
IPBF^{(n,m,k)}(\mathcal M_1,\mathcal M_2)=\frac{\prod_{i=1}^n\pi(y_{ik}|\bY_{nm,-i},\bX_{n,-i},\mathcal M_1)}{\prod_{i=1}^n\pi(y_{ik}|\bY_{nm,-i},\bX_{n,-i},\mathcal M_2)}
	\label{eq:pdf_inverse2}
\end{equation}
and study the limit $\underset{m\rightarrow\infty}{\lim}\underset{n\rightarrow\infty}{\lim}~\frac{1}{n}\log IPBF^{(n,m,k)}(\mathcal M_1,\mathcal M_2)$ for almost all data 
sequences.
Note that since $\left\{y_{ik};k\geq 1\right\}$ are distributed independently as $f\left(\cdot|\theta,x_i,\bY^{(i-1)}_k\right)$ given any $\theta$ and $x_i$, 
it would follow that if the limit exists, it must be the same for all $k\geq 1$.

Suppose that the true data-generating parameter $\theta_0$ is not contained in $\Theta$, the parameter space considered.
This is a case of misspecification that we must incorporate in our convergence theory of PBF. Our PBF asymptotics draws on posterior convergence theory for 
(possibly infinite-dimensional) parameters that also allows misspecification. In this regard, the approach presented in \ctn{Shalizi09} seems to be very appropriate.
Before proceeding further, we first provide a brief overview of this approach, which we conveniently exploit for our purpose.

\section{A brief overview of Shalizi's approach to posterior convergence}
\label{sec:shalizi_briefing}
Let $\bY_n=(Y_1,\ldots,Y_n)^T$, and let $f_{\theta}(\bY_n)$ and $f_{\theta_0}(\bY_n)$ denote the observed and the true likelihoods respectively, under the given value of the parameter $\theta$
and the true parameter $\theta_0$. We assume that $\theta\in\Theta$, where $\Theta$ is the (often infinite-dimensional) parameter space. However, we {\it do not} 
assume that $\theta_0\in\Theta$, thus allowing misspecification.
%note that
%\begin{align}
%f_{\theta}(\bY_n)&=\frac{1}{\left(\sigma\sqrt{2\pi}\right)^n}\exp\left\{-\frac{1}{2\sigma^2}\sum_{i=1}^n(Y_i-\eta(\bx_i))^2\right\};\label{eq:like1}\\
%f_{\theta_0}(\bY_n)&=\frac{1}{\left(\sigma_0\sqrt{2\pi}\right)^n}\exp\left\{-\frac{1}{2\sigma^2_0}\sum_{i=1}^n(Y_i-\eta_0(\bx_i))^2\right\}.\label{eq:true_like1}
%\end{align}
The key ingredient associated with Shalizi's approach to proving convergence of the posterior distribution of $\theta$ is to show that the 
asymptotic equipartition property holds.
To elucidate, let us consider the following likelihood ratio:
\begin{equation*}
R_n(\theta)=\frac{f_{\theta}(\bY_n)}{f_{\theta_0}(\bY_n)}.
%\label{eq:R_n}
\end{equation*}
Then, to say that for each $\theta\in\Theta$, the generalized or relative asymptotic equipartition property holds, we mean
\begin{equation}
\underset{n\rightarrow\infty}{\lim}~\frac{1}{n}\log R_n(\theta)=-h(\theta),
\label{eq:equipartition}
\end{equation}
almost surely, where
$h(\theta)$ is the KL-divergence rate given by
\begin{equation}
h(\theta)=\underset{n\rightarrow\infty}{\lim}~\frac{1}{n}E_{\theta_0}\left(\log\frac{f_{\theta_0}(\bY_n)}{f_{\theta}(\bY_n)}\right),
\label{eq:S3}
\end{equation}
provided that it exists (possibly being infinite), where $E_{\theta_0}$ denotes expectation with respect to the true model.
Let
\begin{align}
h\left(A\right)&=\underset{\theta\in A}{\mbox{ess~inf}}~h(\theta);\notag\\ %\label{eq:h2}\\
J(\theta)&=h(\theta)-h(\Theta);\notag\\ %\label{eq:J}\\
J(A)&=\underset{\theta\in A}{\mbox{ess~inf}}~J(\theta).\notag %\label{eq:J2}
\end{align}
Thus, $h(A)$ can be roughly interpreted as the minimum KL-divergence between the postulated and the true model over the set $A$. If $h(\Theta)>0$, this indicates
model misspecification. 
%However, as we shall show, model misspecification need not always imply that $h(\Theta)>0$. 
For $A\subset\Theta$, $h(A)>h(\Theta)$, so that $J(A)>0$.

As regards the prior, it is required to construct an appropriate sequence of sieves $\mathcal G_n$ such that $\mathcal G_n\rightarrow\Theta$ and $\pi(\mathcal G^c_n)\leq\alpha\exp(-\beta n)$,
for some $\alpha>0$. 

With the above notions, verification of (\ref{eq:equipartition}) along with several other technical conditions ensure that for any $A\subseteq\Theta$ such that $\pi(A)>0$, 
\begin{equation}
\underset{n\rightarrow\infty}{\lim}~\pi(A|\bY_n)=0,
\label{eq:post_conv1}
\end{equation}
almost surely, provided that $h(A)>h(\Theta)$.

The seven assumptions of Shalizi leading to the above result, which we denote as (S1)--(S7), are provided in Appendix \ref{subsec:assumptions_shalizi}.
In what follows, we denote almost sure convergence by ``$\stackrel{a.s.}{\longrightarrow}$", almost sure equality by ``$\stackrel{a.s.}{=}$" and weak convergence
by ``$\stackrel{w}{\longrightarrow}$".

\section{Convergence of PBF in forward problems}
\label{sec:conv_forward}
%Letting $\bY_{n}=\left\{y_{i}:i=1,\ldots,n\right\}$, $\bX_n=\{x_i:i=1,\ldots,n\}$, $\bY_{n,-i}=\bY_n\backslash\{y_i\}$ and $\bX_{n,-i}=\bX_n\backslash\{x_i\}$,
%let $\pi(y_i|\bY_{n,-i},\bX_n,\mathcal M)$ denote the posterior density at $y_i$, given data $\bY_{n,-i}$, $\bX_n$ and model $\mathcal M$.
%Let the density of $y_i$ given $\theta$ and $x_i$ under model $\mathcal M$ be denoted by $f(y_i|\theta,x_i,\mathcal M)$. Then
%note that
%\begin{equation}
%	\pi(y_i|\bY_{n,-i},\bX_n,\mathcal M)=\int_{\Theta} f(y_i|\theta,x_i,\mathcal M)d\pi(\theta|\bY_{n,-i},\bX_{n,-i},\mathcal M),
%\label{eq:post_forward1}
%\end{equation}
%where 
%\begin{equation}
%	\pi(\theta|\bY_{n,-i},\bX_{n,-i},\mathcal M)\propto\pi(\theta)\prod_{j\neq i;j=1}^{n}f(y_j|\theta,x_j,\mathcal M).
%\label{eq:post_forward2}
%\end{equation}

%\subsection{Main convergence result for FPBF in forward regression}
%\label{subsec:main_forward}
Let $\mathcal M_0$ denote the true model which is also associated with parameter $\theta\in\Theta_0$, where $\Theta_0$ is a parameter 
space containing the true parameter $\theta_0$. 
Then the following result holds.
\begin{theorem}
	\label{theorem:forward_cons2}
	Assume conditions (S1)--(S7) of Shalizi, and let the infimum of $h(\theta)$ over $\Theta$ be attained at $\tilde\theta\in\Theta$, where $\tilde\theta\neq\theta_0$.
	Also assume that $\Theta$ and $\Theta_0$ are complete separable metric spaces and that 
	for $i\geq 1$, $f(y_i|\theta,x_i,\bY^{(i-1)},\mathcal M)$ and $f(y_i|\theta,x_i,\bY^{(i-1)},\mathcal M_0)$ are bounded and continuous in $\theta$. 
	Then, 
	\begin{equation}
		\frac{1}{n}\log FPBF^{(n)}(\mathcal M,\mathcal M_0)
		=\frac{1}{n}\log\left[\frac{\prod_{i=1}^n\pi(y_i|\bY_{n,-i},\bX_n,\mathcal M)}{\prod_{i=1}^n\pi(y_i|\bY_{n,-i},\bX_n,\mathcal M_0)}\right]
		\stackrel{a.s.}{\longrightarrow} -h(\tilde\theta),%f(y_i|\tilde\theta,x_i,\mathcal M),
		~\mbox{as}~n\rightarrow\infty,
		\label{eq:forward_cons2}
	\end{equation}
		where, for any $\theta$,
		\begin{equation}
		h(\theta)=
		\underset{n\rightarrow\infty}{\lim}~\frac{1}{n}
			E_{\theta_0}\left\{\sum_{i=1}^n\log\left[\frac{f(y_i|\theta_0,x_i,\bY^{(i-1)},\mathcal M_0)}{f(y_i|\theta,x_i,\bY^{(i-1)},\mathcal M)}\right]\right\}.
			\label{eq:h_forward}
		\end{equation}
\end{theorem}
\begin{proof}
By the hypotheses, (\ref{eq:post_conv1}) holds, from which it follows that for any $\epsilon>0$, 
\begin{equation}
	\underset{n\rightarrow\infty}{\lim}~\pi(\mathbb N^c_{\epsilon}|\bY_{n,-i},\bX_{n,-i},\mathcal M)=0, 
\label{eq:cons1}
\end{equation}
where $\mathbb N_{\epsilon}=\left\{\theta:h(\theta)\leq h\left(\Theta\right)+\epsilon\right\}$.

Now, by hypothesis, the infimum of $h(\theta)$ over $\Theta$ be attained at $\tilde\theta\in\Theta$, where $\tilde\theta\neq\theta_0$. 
	Then by (\ref{eq:cons1}), the posterior of $\theta$ given $\bY_{n,-i}$ and $\bX_{n,-i}$, given by (\ref{eq:post_forward2}),
concentrates around $\tilde\theta$, the minimizer of the limiting KL-divergence rate from the true distribution. 
Formally, given any neighborhood $U$ of $\tilde\theta$, the set $\mathbb N_{\epsilon}$ is contained in $U$ for sufficiently small $\epsilon$.
It follows that for any neighborhood $U$ of $\theta_0$, $\pi(U|\bY_{n,-i},\bX_{n,-i},\mathcal M)\rightarrow 1$, almost surely, as $n\rightarrow\infty$.
Since $\Theta$ is a complete, separable metric space, it follows that (see, for example, \ctn{Ghosh03}, \ctn{Ghosal17})
	\begin{equation}	
		\pi(\cdot|\bY_{n,-i},\bX_{n,-i},\mathcal M)\stackrel{w}{\longrightarrow} \delta_{\tilde\theta}(\cdot),~\mbox{almost surely, as}~n\rightarrow\infty.
		\label{eq:weak1}
	\end{equation}
%where $``\stackrel{w}{\longrightarrow}"$ denotes weak convergence.
	Then, due to (\ref{eq:weak1}) and the Portmanteau theorem, as $f(y_i|\theta,x_i,\bY^{(i-1)},\mathcal M)$ is bounded and continuous in $\theta$, it holds 
	using (\ref{eq:post_forward1}), that
	\begin{equation}
		\pi(y_i|\bY_{n,-i},\bX_n,\mathcal M)\stackrel{a.s.}{\longrightarrow} f(y_i|\tilde\theta,x_i,\bY^{(i-1)},\mathcal M),~\mbox{as}~n\rightarrow\infty.
		\label{eq:forward_cons1}
	\end{equation}
Now, due to (\ref{eq:forward_cons1}),
	\begin{equation}
	\frac{1}{n}\sum_{i=1}^n\log\pi(y_i|\bY_{n,-i},\bX_n,\mathcal M)\stackrel{a.s.}{\longrightarrow} 
		\underset{n\rightarrow\infty}{\lim}~\frac{1}{n}\sum_{i=1}^n\log f(y_i|\tilde\theta,x_i,\bY^{(i-1)},\mathcal M),~\mbox{as}~n\rightarrow\infty.
		\label{eq:forward_cons3}
	\end{equation}
	Also, essentially the same arguments leading to (\ref{eq:forward_cons1}) yield
	\begin{equation*}
		\pi(y_i|\bY_{n,-i},\bX_n,\mathcal M_0)\stackrel{a.s.}{\longrightarrow} f(y_i|\theta_0,x_i,\bY^{(i-1)},\mathcal M_0),~\mbox{as}~n\rightarrow\infty,
	\end{equation*}
	which ensures
	\begin{equation}
	\frac{1}{n}\sum_{i=1}^n\log\pi(y_i|\bY_{n,-i},\bX_n,\mathcal M_0)\stackrel{a.s.}{\longrightarrow} 
		\underset{n\rightarrow\infty}{\lim}~\frac{1}{n}\sum_{i=1}^n\log f(y_i|\theta_0,x_i,\bY^{(i-1)},\mathcal M_0),~\mbox{as}~n\rightarrow\infty.
		\label{eq:forward_cons4}
	\end{equation}
	From (\ref{eq:forward_cons3}) and (\ref{eq:forward_cons4}) we obtain
	\begin{equation}
		\underset{n\rightarrow\infty}{\lim}~\frac{1}{n}\log FPBF^{(n)}(\mathcal M,\mathcal M_0)\stackrel{a.s.}{=}
		\underset{n\rightarrow\infty}{\lim}~\frac{1}{n}\sum_{i=1}^n\log\left[\frac{f(y_i|\tilde\theta,x_i,\bY^{(i-1)},\mathcal M)}
		{f(y_i|\tilde\theta_0,x_i,\bY^{(i-1)},\mathcal M_0)}\right]
		\stackrel{a.s.}{=}-h(\tilde\theta),
		\label{eq:forward_cons5}
	\end{equation}
	where the rightmost step of (\ref{eq:forward_cons5}), given by (\ref{eq:h_forward}), follows due to (\ref{eq:equipartition}). 
	Hence, the result is proved.
\end{proof}

For postulated model $\mathcal M_j$, let the KL-divergence rate $h$ in (\ref{eq:S3}) be denoted by $h_j$, for $j\geq 1$.
\begin{theorem}
	\label{theorem:forward_cons2_mis}
	For models $\mathcal M_0$, $\mathcal M_1$ and $\mathcal M_2$ with complete separable parameter spaces $\Theta_0$, $\Theta_1$ and $\Theta_2$, 
	assume conditions (S1)--(S7) of Shalizi, and for $j=1,2$, let the infimum of $h_j(\theta)$ over $\Theta_j$ 
	be attained at $\tilde\theta_j\in\Theta_j$, where $\tilde\theta_j\neq\theta_0$.
	Also assume that for $i\geq 1$, $f(y_i|\theta,x_i,\bY^{(i-1)},\mathcal M_j)$; $j=1,2$, and $f(y_i|\theta,x_i,\bY^{(i-1)},\mathcal M_0)$ are bounded 
	and continuous in $\theta$. 
	Then, 
	\begin{equation}
		\frac{1}{n}\log FPBF^{(n)}(\mathcal M_1,\mathcal M_2)
		=\frac{1}{n}\log\left[\frac{\prod_{i=1}^n\pi(y_i|\bY_{n,-i},\bX_n,\mathcal M_1)}{\prod_{i=1}^n\pi(y_i|\bY_{n,-i},\bX_n,\mathcal M_2)}\right]
		\stackrel{a.s.}{\longrightarrow} -\left[h(\tilde\theta_1)-h(\tilde\theta_2)\right],%f(y_i|\tilde\theta,x_i,\mathcal M),
		~\mbox{as}~n\rightarrow\infty,
		\label{eq:forward_cons2_mis}
	\end{equation}
		where, for $j=1,2$, and for any $\theta$,
		\begin{equation}
		h_j(\theta)=
		\underset{n\rightarrow\infty}{\lim}~\frac{1}{n}
			E_{\theta_0}\left\{\sum_{i=1}^n\log\left[\frac{f(y_i|\theta_0,x_i,\bY^{(i-1)},\mathcal M_0)}
			{f(y_i|\theta,x_i,\bY^{(i-1)},\mathcal M_j)}\right]\right\}.
			\label{eq:h_forward_mis}
		\end{equation}
\end{theorem}
\begin{proof}
	The proof follows by noting that $$\frac{1}{n}\log FPBF^{(n)}(\mathcal M_1,\mathcal M_2)=\frac{1}{n}\log FPBF^{(n)}(\mathcal M_1,\mathcal M_0)
	-\frac{1}{n}\log FPBF^{(n)}(\mathcal M_2,\mathcal M_0),$$ and then using (\ref{eq:forward_cons2}) for 
	$\frac{1}{n}\log FPBF^{(n)}(\mathcal M_1,\mathcal M_0)$ and $\frac{1}{n}\log FPBF^{(n)}(\mathcal M_2,\mathcal M_0)$.
\end{proof}

%\subsection{Main convergence result for FPBF in forward regression}
%\label{subsec:main_forward}
\section{Convergence results for PBF in inverse regression: first setup}
\label{sec:pbf_inverse_1st}

\begin{theorem}
	\label{theorem:inverse_cons2}
	Assume conditions (S1)--(S7) of Shalizi, and let the infimum of $h(\theta)$ over $\Theta$ be attained at $\tilde\theta\in\Theta$, where $\tilde\theta\neq\theta_0$.
	Also assume that $\Theta$ and $\Theta_0$ are complete separable metric spaces and that 
	for $i\geq 1$, $g(\bY^{(i)},\theta,\mathcal M)$ and $g(\bY^{(i)},\theta,\mathcal M_0)$ are bounded and continuous in $\theta$. 
	Then, 
	\begin{equation}
		\frac{1}{n}\log IPBF^{(n)}(\mathcal M,\mathcal M_0)
		=\frac{1}{n}\log\left[\frac{\prod_{i=1}^n\pi(y_i|\bY_{n,-i},\bX_{n,-i},\mathcal M)}{\prod_{i=1}^n\pi(y_i|\bY_{n,-i},\bX_{n,-i},\mathcal M_0)}\right]
		\stackrel{a.s.}{\longrightarrow} -h^*(\tilde\theta),%f(y_i|\tilde\theta,x_i,\mathcal M),
		~\mbox{as}~n\rightarrow\infty,
		\label{eq:inverse_cons2}
	\end{equation}
		where, for any $\theta$,
		\begin{equation*}
		h^*(\theta)=
		\underset{n\rightarrow\infty}{\lim}~\frac{1}{n}
			\sum_{i=1}^n\log\left[\frac{g(\bY^{(i)},\theta_0,\mathcal M_0)}{g(\bY^{(i)},\theta,\mathcal M)}\right],
		%E_{\theta_0}\left\{\sum_{i=1}^n\log\left[\frac{g(y_i,\theta_0,\mathcal M_0)}{g(y_i,\theta,\mathcal M)}\right]\right\}.
			%\label{eq:h_inverse}
		\end{equation*}
		provided that the limit exists.
\end{theorem}
\begin{proof}
Since $\pi(\cdot|\bY_{n,-i},\bX_{n,-i},\mathcal M)$ remains the same as in Theorem \ref{theorem:forward_cons2}, it follows as before that
%By the hypotheses, (\ref{eq:post_conv1}) holds, from which it follows that for any $\epsilon>0$, 
%\begin{equation}
%	\underset{n\rightarrow\infty}{\lim}~\pi(\mathbb N^c_{\epsilon}|\bY_{n,-i},\bX_{n,-i},\mathcal M)=0, 
%\label{eq:cons1}
%\end{equation}
%where $\mathbb N_{\epsilon}=\left\{\theta:h(\theta)\leq h\left(\Theta\right)+\epsilon\right\}$.

%Now, by hypothesis, the infimum of $h(\theta)$ over $\Theta$ be attained at $\tilde\theta\in\Theta$, where $\tilde\theta\neq\theta_0$. 
%	Then by (\ref{eq:cons1}), the posterior of $\theta$ given $\bY_{n,-i}$ and $\bX_{n,-i}$, given by (\ref{eq:post_forward2}),
%concentrates around $\tilde\theta$, the minimizer of the limiting KL-divergence rate from the true distribution. 
%Formally, given any neighborhood $U$ of $\tilde\theta$, the set $\mathbb N_{\epsilon}$ is contained in $U$ for sufficiently small $\epsilon$.
%It follows that for any neighborhood $U$ of $\theta_0$, $\pi(U|\bY_{n,-i},\bX_{n,-i},\mathcal M)\rightarrow 1$, almost surely, as $n\rightarrow\infty$.
%Since $\Theta$ is a complete, separable metric space, it follows that (see, for example, \ctn{Ghosh03}, \ctn{Ghosal17})
	\begin{equation*}	
		\pi(\cdot|\bY_{n,-i},\bX_{n,-i},\mathcal M)\stackrel{w}{\longrightarrow} \delta_{\tilde\theta}(\cdot),~\mbox{almost surely, as}~n\rightarrow\infty.
		%\label{eq:weak1}
	\end{equation*}
%where $``\stackrel{w}{\longrightarrow}"$ denotes weak convergence.
	Then, since $g(y_i,\theta,\mathcal M)$ is bounded and continuous in $\theta$, 
	the above ensures in conjunction with the Portmanteau theorem using (\ref{eq:post_inverse1}), that
	\begin{equation}
		\pi(y_i|\bY_{n,-i},\bX_{n,-i},\mathcal M)\stackrel{a.s.}{\longrightarrow} g(\bY^{(i)},\tilde\theta,\mathcal M),~\mbox{as}~n\rightarrow\infty.
		\label{eq:inverse_cons1}
	\end{equation}
Hence,
	\begin{equation}
		\frac{1}{n}\sum_{i=1}^n\log\pi(y_i|\bY_{n,-i},\bX_{n,-i},\mathcal M)\stackrel{a.s.}{\longrightarrow} 
		\underset{n\rightarrow\infty}{\lim}~\frac{1}{n}\sum_{i=1}^n\log g(\bY^{(i)},\tilde\theta,\mathcal M),~\mbox{as}~n\rightarrow\infty.
		\label{eq:inverse_cons3}
	\end{equation}
	%Also, essentially the same arguments leading to (\ref{eq:forward_cons1}) yield
	%\begin{equation*}
%		\pi(y_i|\bY_{n,-i},\bX_n,\mathcal M_0)\stackrel{a.s.}{\longrightarrow} f(y_i|\theta_0,x_i,\mathcal M_0),~\mbox{as}~n\rightarrow\infty,
%	\end{equation*}
Similarly,
	\begin{equation}
		\frac{1}{n}\sum_{i=1}^n\log\pi(y_i|\bY_{n,-i},\bX_{n,-i},\mathcal M_0)\stackrel{a.s.}{\longrightarrow} 
		\underset{n\rightarrow\infty}{\lim}~\frac{1}{n}\sum_{i=1}^n\log g(\bY^{(i)},\theta_0,\mathcal M_0),~\mbox{as}~n\rightarrow\infty.
		\label{eq:inverse_cons4}
	\end{equation}
	Combining (\ref{eq:inverse_cons3}) and (\ref{eq:inverse_cons4}) yields
	\begin{equation*}
		\underset{n\rightarrow\infty}{\lim}~\frac{1}{n}\log IPBF^{(n)}(\mathcal M,\mathcal M_0)\stackrel{a.s.}{=}
		\underset{n\rightarrow\infty}{\lim}~\frac{1}{n}\sum_{i=1}^n\log\left[\frac{g(\bY^{(i)},\tilde\theta,\mathcal M)}{g(\bY^{(i)},\theta_0,\mathcal M_0)}\right]
		%\stackrel{a.s.}{=}
		=-h^*(\tilde\theta).
		%\label{eq:forward_cons5}
	\end{equation*}
	%where $h(\theta)$ is given by (\ref{eq:h_inverse}).
	%where the rightmost step of (\ref{eq:forward_cons5}) follows due to (\ref{eq:equipartition}). 
	Hence, the result is proved.
\end{proof}
\begin{remark}
	\label{eq:h1}
	Observe that $h^*(\tilde\theta)$ in Theorem \ref{theorem:inverse_cons2} does not correspond to the KL-divergence rate given by (\ref{eq:S3}), even though
	in the forward context, Theorem \ref{theorem:forward_cons2} shows almost convergence of $\frac{1}{n}\log FPBF^{(n)}$ to $-h(\tilde\theta)$, 
	where $h(\tilde\theta)$ is the {\it bona fide} KL-divergence rate.
\end{remark}

In Theorem \ref{theorem:inverse_cons2} we have assumed that for cross-validation even in the true model $\mathcal M_0$, $x_i$ is assumed unknown, 
and that a prior has been placed on
the corresponding unknown random quantity $\tilde x_i$. If, on the other hand, $x_i$ is considered known for cross-validation in $\mathcal M_0$, 
then we we have the following theorem,
which is an appropriately modified version of Theorem \ref{theorem:inverse_cons2}.

\begin{theorem}
	\label{theorem:inverse_cons2_true}
	Assume conditions (S1)--(S7) of Shalizi for models $\mathcal M_0$ and $\mathcal M$, 
	and let the infimum of $h(\theta)$ over $\Theta$ be attained at $\tilde\theta\in\Theta$, where $\tilde\theta\neq\theta_0$.
	Also assume that $\Theta$ and $\Theta_0$ are complete separable metric spaces and that 
	for $i\geq 1$, $g(\bY^{(i)},\theta,\mathcal M)$ and $f(y_i|\theta,x_i,\bY^{(i-1)},\mathcal M_0)$ are bounded and continuous in $\theta$. 
	Then the following result holds if $x_i$ is assumed known for cross-validation with respect to $\mathcal M_0$:  
	\begin{equation}
		\frac{1}{n}\log IPBF^{(n)}(\mathcal M,\mathcal M_0)
		=\frac{1}{n}\log\left[\frac{\prod_{i=1}^n\pi(y_i|\bY_{n,-i},\bX_{n,-i},\mathcal M)}{\prod_{i=1}^n\pi(y_i|\bY_{n,-i},\bX_{n,-i},\mathcal M_0)}\right]
		\stackrel{a.s.}{\longrightarrow} -h^*(\tilde\theta),%f(y_i|\tilde\theta,x_i,\mathcal M),
		~\mbox{as}~n\rightarrow\infty,
		\label{eq:inverse_cons2_true}
	\end{equation}
		where, for any $\theta$,
		\begin{equation*}
		h^*(\theta)=
		\underset{n\rightarrow\infty}{\lim}~\frac{1}{n}
			\sum_{i=1}^n\log\left[\frac{f(y_i|\theta_0,x_i,\bY^{(i-1)},\mathcal M_0)}{g(\bY^{(i)},\theta,\mathcal M)}\right],
		%E_{\theta_0}\left\{\sum_{i=1}^n\log\left[\frac{g(y_i,\theta_0,\mathcal M_0)}{g(y_i,\theta,\mathcal M)}\right]\right\}.
			%\label{eq:h_inverse}
		\end{equation*}
		provided that the limit exists.
\end{theorem}
\begin{proof}
	In this case, for the true model $\mathcal M_0$, the cross-validation posterior $\pi(y_i|\bY_{n,-i},\bX_{n,-i},\mathcal M_0)$ is of the same form
	as (\ref{eq:post_forward1}) and hence, (\ref{eq:forward_cons4}) holds. The rest of the proof remains the same as that of Theorem \ref{theorem:inverse_cons2}.
\end{proof}

\begin{remark}
	\label{eq:h2}
	Observe that $h^*(\tilde\theta)$ in Theorem \ref{theorem:inverse_cons2_true} 
	is a genuine KL-divergence rate. However, this is not the same as $h(\tilde\theta)$ of Theorem \ref{theorem:forward_cons2}, which is
	the KL-divergence rate between $\mathcal M$ and $\mathcal M_0$ when all the $x_i$ are known.
	Since cross-validation with all $x_i$ known can occur only in the forward regression setup, convergence rates of pseudo-Bayes factors
	in inverse regression problems can never be associated with $h$, even though 
	the conditions of Theorem \ref{theorem:inverse_cons2_true} show that $\tilde\theta$ is the minimizer of $h$.
\end{remark}

\begin{theorem}
	\label{theorem:inverse_cons2_mis}
	For models $\mathcal M_0$, $\mathcal M_1$ and $\mathcal M_2$ with complete separable parameter spaces $\Theta_0$, $\Theta_1$ and $\Theta_2$, 
	assume conditions (S1)--(S7) of Shalizi, and for $j=1,2$, let the infimum of $h_j(\theta)$ over $\Theta_j$ 
	be attained at $\tilde\theta_j\in\Theta_j$, where $\tilde\theta_j\neq\theta_0$.
	Also assume that for $i\geq 1$, $g(\bY^{(i)},\theta,\mathcal M_j)$; $j=1,2$, and $f(y_i|\theta,x_i,\bY^{(i-1)},\mathcal M_0)$ are bounded and continuous in $\theta$, 
	%and $g(y_i,\theta,\mathcal M_0)$ or $f(y_i|\theta,x_i,\mathcal M_0)$ is bounded and continuous in $\theta$, dependning upon whether or not
	%$x_i$ is assumed unknown or known for cross-validation with respect to $\mathcal M_0$. 
	Then, if $x_i$ is assumed known for cross-validation with respect to $\mathcal M_0$, the following holds: 
	\begin{equation}
		\frac{1}{n}\log IPBF^{(n)}(\mathcal M_1,\mathcal M_2)
		=\frac{1}{n}\log\left[\frac{\prod_{i=1}^n\pi(y_i|\bY_{n,-i},\bX_n,\mathcal M_1)}{\prod_{i=1}^n\pi(y_i|\bY_{n,-i},\bX_n,\mathcal M_2)}\right]
		\stackrel{a.s.}{\longrightarrow} -\left[h^*_1(\tilde\theta_1)-h^*_2(\tilde\theta_2)\right],%f(y_i|\tilde\theta,x_i,\mathcal M),
		~\mbox{as}~n\rightarrow\infty,
		\label{eq:inverse_cons2_mis}
	\end{equation}
		where, for $j=1,2$, and for any $\theta$,
		\begin{equation}
		h^*_j(\theta)=
		\underset{n\rightarrow\infty}{\lim}~\frac{1}{n}
		%E_{\theta_0}\left\{\sum_{i=1}^n\log\left[\frac{f(y_i|\theta_0,x_i,\mathcal M_0)}{g(y_i,\theta,\mathcal M_j)}\right]\right\}.
			\sum_{i=1}^n\log\left[\frac{f(y_i|\theta_0,x_i,\bY^{(i-1)},\mathcal M_0)}{g(\bY^{(i)},\theta,\mathcal M_j)}\right],
			\label{eq:h_inverse_mis}
		\end{equation}
		provided the limit exists.
\end{theorem}
\begin{proof}
	The proof follows by noting that $$\frac{1}{n}\log IPBF^{(n)}(\mathcal M_1,\mathcal M_2)=\frac{1}{n}\log IPBF^{(n)}(\mathcal M_1,\mathcal M_0)
	-\frac{1}{n}\log IPBF^{(n)}(\mathcal M_2,\mathcal M_0),$$ and then using (\ref{eq:inverse_cons2_true}) for 
	$\frac{1}{n}\log IPBF^{(n)}(\mathcal M_1,\mathcal M_0)$ and $\frac{1}{n}\log IPBF^{(n)}(\mathcal M_2,\mathcal M_0)$.
\end{proof}
\begin{remark}
	\label{remark:inverse1}
	Note that the result of Theorem \ref{theorem:inverse_cons2_mis} holds without the assumption that $\Theta_0$ is complete separable and 
	$f(y_i|\theta,x_i,\bY^{(i-1)},\mathcal M_0)$ is bounded and continuous in $\theta$, irrespective of whether or not $x_i$ is treated as known
	in the case of cross-validation with respect to the true model $\mathcal M_0$. 
	Indeed, assuming the rest of the conditions of Theorem \ref{theorem:inverse_cons2_mis},
	it holds that 
	\begin{equation*}
		\frac{1}{n}\log IPBF^{(n)}(\mathcal M_1,\mathcal M_2)
		=\frac{1}{n}\log\left[\frac{\prod_{i=1}^n\pi(y_i|\bY_{n,-i},\bX_{n,-i},\mathcal M_1)}{\prod_{i=1}^n\pi(y_i|\bY_{n,-i},\bX_{n,-i},\mathcal M_2)}\right]
		\stackrel{a.s.}{\longrightarrow} -h^*(\tilde\theta_1,\tilde\theta_2),%f(y_i|\tilde\theta,x_i,\mathcal M),
		~\mbox{as}~n\rightarrow\infty,
		%\label{eq:inverse_cons2}
	\end{equation*}
		where, for any $\theta_1,\theta_2$,
		\begin{equation*}
		h^*(\theta_1,\theta_2)=
		\underset{n\rightarrow\infty}{\lim}~\frac{1}{n}
			\sum_{i=1}^n\log\left[\frac{g(\bY^{(i)},\theta_2,\mathcal M_2)}{g(\bY^{(i)},\theta_1,\mathcal M_1)}\right],
		%E_{\theta_0}\left\{\sum_{i=1}^n\log\left[\frac{g(y_i,\theta_0,\mathcal M_0)}{g(y_i,\theta,\mathcal M)}\right]\right\}.
			%\label{eq:h_inverse}
		\end{equation*}
		provided that the limit exists.
		The proof follows in the same way as in Theorem \ref{theorem:inverse_cons2} by replacing $\mathcal M$ and $\mathcal M_0$ with
		$\mathcal M_1$ and $\mathcal M_2$. 
		Note that $h^*(\tilde\theta_1,\tilde\theta_2)$ above is the same as $h^*(\tilde\theta_1)-h^*(\tilde\theta_2)$ of Theorem \ref{theorem:inverse_cons2_mis},
		but the latter is interpretable as the difference between limiting KL-divergence rates for $\mathcal M_1$ and $\mathcal M_2$, while the former does not admit
		such desirable interpretation since without the assumptions $\Theta_0$ is complete separable and $f(y_i|\theta,x_i,\bY^{(i-1)},\mathcal M_0)$ is 
		bounded and continuous in $\theta$, the convergence
	\begin{equation*}
		\frac{1}{n}\sum_{i=1}^n\log\pi(y_i|\bY_{n,-i},\bX_{n,-i},\mathcal M_0)\stackrel{a.s.}{\longrightarrow} 
		\underset{n\rightarrow\infty}{\lim}~\frac{1}{n}\sum_{i=1}^n\log f(y_i|\theta_0,x_i,\bY^{(i-1)},\mathcal M_0),~\mbox{as}~n\rightarrow\infty,
		%\label{eq:inverse_cons4}
	\end{equation*}
	need not hold, even if $x_i$ is considered known for cross-validation with respect to $\mathcal M_0$.
\end{remark}

\section{Convergence results for PBF in inverse regression: second setup}
\label{sec:pbf_inverse_2nd}
In the misspecified situation, $\theta_0\notin\Theta$, and $\tilde\theta$ is the minimizer of the limiting KL-divergence rate from $\theta_0$.
If $\theta$ is thus misspecified, then as $m\rightarrow\infty$, $B_{im}(\tilde\theta)\stackrel{a.s.}{\longrightarrow}\{x^*_i\}$ for some non-random $x^*_i~(\neq x_i)$ 
depending upon both $\tilde\theta$ and $\theta_0$. In other words, the prior distribution of $\tilde x_i$ given $\tilde\theta$ and $\by_i$ concentrates around
$x^*_i$, as $m\rightarrow\infty$. %With the following assumption
%\begin{as}
%$f(\by_i|\tilde\theta,\tilde x_i)$ is continuous in $\tilde x_i$;
%\label{as:as4}
%\end{as}
%\begin{as}
%$\eta_0$ and $\tilde\eta$ are one-to-one functions,
%\label{as:as5}
%\end{as}
%\noindent Lemma \ref{lemma:lemma2} below can be proved in the same way as the proof of Theorem 2 of \ctn{Chat20}.
%\begin{lemma}
%\label{lemma:lemma2}
%Under the prior (\ref{eq:prior_x}) and Assumptions \ref{as:as4} and \ref{as:as5}, for any neighborhood $U_i$ of $x^*_i$, for any $i\geq 1$,
%\begin{equation}
%\pi(\tilde x_i\in U^c_i|\tilde\theta,\by_i)\stackrel{a.s.}{\longrightarrow}0,~\mbox{as}~m\rightarrow\infty.
%\label{eq:conv3}
%\end{equation}
%\end{lemma}
%With these, 
We now state and prove our result on IPBF convergence with respect to the prior (\ref{eq:prior_x}).
\begin{theorem}
	\label{theorem:inverse_cons3}
	Assume conditions (S1)--(S7) of Shalizi. %and Assumption \ref{as:as5}. %conditions \ref{as:as1}--\ref{as:as5}. 
	Let the infimum of $h(\theta)$ over $\Theta$ be attained at $\tilde\theta\in\Theta$, where $\tilde\theta\neq\theta_0$.
	Assume that $\tilde\theta$ and $\theta_0$ are one-to-one functions.
	Also assume that $\Theta$ and $\Theta_0$ are complete separable metric spaces and that 
	for $i\geq 1$ and $k\geq 1$, $f(y_{ik}|\theta,\tilde x_i,\bY^{(i-1)}_k,\mathcal M)$ and $f(y_{ik}|\theta,\tilde x_i,\bY^{(i-1)}_k,\mathcal M_0)$ 
	are bounded and continuous in $(\theta,\tilde x_i)$.
	Then, for prior (\ref{eq:prior_x}) on $\tilde x_i$, the following holds for any $k\geq 1$: 
	\begin{equation}
		\underset{m\rightarrow\infty}{\lim}\underset{n\rightarrow\infty}{\lim}~\frac{1}{n}\log IPBF^{(n,m,k)}(\mathcal M,\mathcal M_0)
		=\underset{m\rightarrow\infty}{\lim}\underset{n\rightarrow\infty}{\lim}~
		\frac{1}{n}\log\left[\frac{\prod_{i=1}^n\pi(y_{ik}|\bY_{nm,-i},\bX_{n,-i},\mathcal M)}{\prod_{i=1}^n\pi(y_{ik}|\bY_{nm,-i},\bX_{n,-i},\mathcal M_0)}\right]
		\stackrel{a.s.}{=} -h^*(\tilde\theta),%f(y_i|\tilde\theta,x_i,\mathcal M),
		\label{eq:inverse_cons3_2}
	\end{equation}
		where %, for any $\theta$,
		\begin{equation*}
		h^*(\tilde\theta)=
		\underset{n\rightarrow\infty}{\lim}~\frac{1}{n}
			\sum_{i=1}^n\log\left[\frac{f(y_{ik}|\theta_0,x_i,\bY^{(i-1)}_k,\mathcal M_0)}{f(y_{ik}|\tilde\theta,x^*_i,\bY^{(i-1)}_k,\mathcal M)}\right],
		%	=\underset{n\rightarrow\infty}{\lim}~
		%	\frac{1}{n}E_{\theta_0}\left\{\sum_{i=1}^n\log\left[\frac{f(\by_i|\theta_0,x_i,\mathcal M_0)}{f(\by_i|\theta,x^*_i,\mathcal M)}\right]\right\}.
			%\label{eq:h_inverse}
		\end{equation*}
		provided that the limit exists.
\end{theorem}
\begin{proof}
	It follows from (\ref{eq:inv1}) that 
	$\pi(\tilde x_i,\theta|\bY_{nm,-i},\bX_{n,-i},\mathcal M)=\pi(\tilde x_i|\theta,\mathcal M)\pi(\theta|\bY_{nm,-i},\bX_{n,-i},\mathcal M)$.
	Hence, letting $U_i\times V$ be any neighborhood of $(x^*_i,\tilde\theta)$, we have
	\begin{equation}
		\pi(\tilde x_i\in U_i,\theta\in V|\bY_{nm,-i},\bX_{n,-i},\mathcal M)
		=\int_V\pi(\tilde x_i\in U_i|\theta,\mathcal M)d\pi(\theta|\bY_{n,-i},\bX_{n,-i},\mathcal M).
		\label{eq:eq1}
	\end{equation}
	Since $\pi(\cdot|\bY_{n,-i},\bX_{n,-i},\mathcal M)\stackrel{w}{\longrightarrow}\delta_{\tilde\theta}(\cdot)$, as $n\rightarrow\infty$, for any $m\geq 1$,
	and since $\pi(\tilde x_i\in U_i|\theta,\mathcal M)$ is bounded (since it is a probability) and continuous in $\theta$ by Lemma 4.1 of \ctn{Chat20},
	by the Portmanteau theorem it follows from (\ref{eq:eq1}) that for $m\geq 1$,
	\begin{equation}
		\pi(\tilde x_i\in U_i,\theta\in V|\bY_{nm,-i},\bX_{n,-i},\mathcal M)\stackrel{a.s.}{\longrightarrow} \pi(\tilde x_i\in U_i|\tilde\theta,\mathcal M),
		~\mbox{as}~n\rightarrow\infty.
		\label{eq:eq2}
	\end{equation}
	Now, since $B_{im}(\tilde\theta)\stackrel{a.s.}{\longrightarrow}\{x^*_i\}$ as $m\rightarrow\infty$ since $\tilde\theta$ is one-to-one, %Assumption \ref{as:as5}, 
	it follows that there exists $m_0\geq 1$ such that for $m\geq m_0$,
	$B_{im}(\tilde\theta)\subset U_i$. Hence, 
	\begin{equation}
		\pi(\tilde x_i\in U_i|\tilde\theta,\mathcal M)\stackrel{a.s.}{\longrightarrow}1,~\mbox{as}~m\rightarrow\infty.
		\label{eq:eq3}
	\end{equation}
	Combining (\ref{eq:eq2}) and (\ref{eq:eq3}) yields
	\begin{equation}
		\pi(\tilde x_i\in U_i,\theta\in V|\bY_{nm,-i},\bX_{n,-i},\mathcal M)\stackrel{a.s.}{\longrightarrow} 1,
		~\mbox{as}~m\rightarrow\infty,~n\rightarrow\infty.
		\label{eq:eq4}
	\end{equation}
	From (\ref{eq:eq4}) it follows thanks to complete separability of $\mathcal X$ and $\Theta$, that
	\begin{equation}
		\pi(\cdot|\bY_{nm,-i},\bX_{n,-i},\mathcal M)\stackrel{w}{\longrightarrow} \delta_{(x^*_i,\tilde\theta)}(\cdot),
		~\mbox{as}~m\rightarrow\infty,~n\rightarrow\infty.
		\label{eq:eq4_new}
	\end{equation}
Since 
	$\pi(y_{ik}|\bY_{nm,-i},\bX_{n,-i},\mathcal M)
	=\int_{\mathcal X}\int_{\Theta} f(y_{ik}|\theta,\tilde x_i,\bY^{(i-1)},\mathcal M)d\pi(\tilde x_i,\theta|\bY_{nm,-i},\bX_{n,-i},\mathcal M)$,
	and $f(y_{ik}|\theta,\tilde x_i,\bY^{(i-1)},\mathcal M)$ is bounded and continuous in $(\theta,\tilde x_i)$, it follows using (\ref{eq:eq4_new}) 
	and the Portmanteau theorem, that
	\begin{equation}
		\pi(y_{ik}|\bY_{nm,-i},\bX_{n,-i},\mathcal M)\stackrel{a.s.}{\longrightarrow}f(y_{ik}|\tilde\theta,x^*_i,\bY^{(i-1)}_k,\mathcal M),
		~\mbox{as}~m\rightarrow\infty,~n\rightarrow\infty.
		\label{eq:eq5}
	\end{equation}
Hence,
	\begin{equation}
		\frac{1}{n}\sum_{i=1}^n\log\pi(y_{ik}|\bY_{nm,-i},\bX_{n,-i},\mathcal M)\stackrel{a.s.}{\longrightarrow} 
		\underset{n\rightarrow\infty}{\lim}~\frac{1}{n}\sum_{i=1}^n\log f(y_{ik}|\tilde\theta,x^*_i,\bY^{(i-1)}_k,\mathcal M),
		~\mbox{as}~m\rightarrow\infty,~n\rightarrow\infty.
		\label{eq:eq6}
	\end{equation}
In the same way,
	\begin{equation}
		\frac{1}{n}\sum_{i=1}^n\log\pi(y_{ik}|\bY_{nm,-i},\bX_{n,-i},\mathcal M_0)\stackrel{a.s.}{\longrightarrow} 
		\underset{n\rightarrow\infty}{\lim}~\frac{1}{n}\sum_{i=1}^n\log f(y_{ik}|\theta_0,x_i,\bY^{(i-1)}_k,\mathcal M_0),
		~\mbox{as}~m\rightarrow\infty,~n\rightarrow\infty.
		\label{eq:eq7}
	\end{equation}
	Combining (\ref{eq:eq6}) and (\ref{eq:eq7}) yields
	\begin{equation*}
		\underset{m\rightarrow\infty}{\lim}\underset{n\rightarrow\infty}{\lim}~\frac{1}{n}\log IPBF^{(n,m,k)}(\mathcal M,\mathcal M_0)\stackrel{a.s.}{=}
		\underset{n\rightarrow\infty}{\lim}~\frac{1}{n}\sum_{i=1}^n\log
		\left[\frac{f(y_{ik}|\tilde\theta,x^*_i,\bY^{(i-1)}_k,\mathcal M)}{f(y_{ik}|\theta_0,x_i,\bY^{(i-1)}_k,\mathcal M_0)}\right]
		%\stackrel{a.s.}{=}
		=-h^*(\tilde\theta),
		%\label{eq:forward_cons5}
	\end{equation*}
	%where $h(\theta)$ is given by (\ref{eq:h_inverse}).
	%where the rightmost step of (\ref{eq:forward_cons5}) follows due to (\ref{eq:equipartition}). 
	thereby proving the result.
\end{proof}
\begin{remark}
	\label{remark:inv_cv1}
	Theorem \ref{theorem:inverse_cons3} assumes that for $\mathcal M_0$, cross-validation is carried out assuming $x_i$ is unknown. However,
	as is clear from the proof, the same result continues to hold even if $x_i$ is treated as known.
\end{remark}

\begin{theorem}
	\label{theorem:inverse_cons3_mis}
	For models $\mathcal M_0$, $\mathcal M_1$ and $\mathcal M_2$ with complete separable parameter spaces $\Theta_0$, $\Theta_1$ and $\Theta_2$, 
	assume conditions (S1)--(S7) of Shalizi and for $j=1,2$, let the infimum of $h_j(\theta)$ over $\Theta_j$ 
	be attained at $\tilde\theta_j\in\Theta_j$, where $\tilde\theta_j\neq\theta_0$. Consider the prior (\ref{eq:prior_x}) on $\tilde x_i$
	and let $B_{im}(\tilde\theta_j)\stackrel{a.s.}{\longrightarrow}\{x^*_{ij}\}$, for $j=1,2$.
	Also assume that for $i\geq 1$ and $k\geq 1$, $f(y_{ik}|\theta,\tilde x_i,\bY^{(i-1)},\mathcal M_j)$; $j=1,2$, 
	and $f(y_{ik}|\theta,\tilde x_i,\bY^{(i-1)},\mathcal M_0)$ are bounded and continuous in 
	$(\theta,\tilde x_i)$, in addition to the conditions that $\theta_0$ and $\tilde\theta_j$; $j=1,2$, are one-to-one. %conditions \ref{as:as1}--\ref{as:as5}. 
	%and $g(y_i,\theta,\mathcal M_0)$ or $f(y_i|\theta,x_i,\mathcal M_0)$ is bounded and continuous in $\theta$, dependning upon whether or not
	%$x_i$ is assumed unknown or known for cross-validation with respect to $\mathcal M_0$. 
	%Then, if $x_i$ is assumed known for cross-validation with respect to $\mathcal M_0$, the following holds: 
	Then, the following holds for any $k\geq 1$: 
	\begin{align}
		&\underset{m\rightarrow\infty}{\lim}\underset{n\rightarrow\infty}{\lim}~\frac{1}{n}\log IPBF^{(n,m,k)}(\mathcal M_1,\mathcal M_2)\notag\\
		&=\underset{m\rightarrow\infty}{\lim}\underset{n\rightarrow\infty}{\lim}~
		\frac{1}{n}\log\left[\frac{\prod_{i=1}^n\pi(y_{ik}|\bY_{nm,-i},\bX_{n,-i},\mathcal M_1)}{\prod_{i=1}^n\pi(y_{ik}|\bY_{nm,-i},\bX_{n,-i},\mathcal M_2)}\right]
		\stackrel{a.s.}{=} -\left[h^*_1(\tilde\theta_1)-h^*_2(\tilde\theta_2)\right],%f(y_i|\tilde\theta,x_i,\mathcal M),
		%&\qquad\mbox{as}~m\rightarrow\infty;~n\rightarrow\infty,
		\label{eq:inverse_cons2_mis2}
	\end{align}
		where, for $j=1,2$, and for any $\theta$,
		\begin{equation}
		h^*_j(\theta)=
		\underset{n\rightarrow\infty}{\lim}~\frac{1}{n}
		%E_{\theta_0}\left\{\sum_{i=1}^n\log\left[\frac{f(y_i|\theta_0,x_i,\mathcal M_0)}{g(y_i,\theta,\mathcal M_j)}\right]\right\}.
			\sum_{i=1}^n\log\left[\frac{f(y_{ik}|\theta_0,x_i,\bY^{(i-1)}_k,\mathcal M_0)}{f(y_{ik}|\theta,x^*_{ij},\bY^{(i-1)}_k,\mathcal M_j)}\right],
			\label{eq:h_inverse3_mis}
		\end{equation}
		provided the limit exists.
\end{theorem}
\begin{proof}
	The proof follows by noting that $$\frac{1}{n}\log IPBF^{(n,m,k)}(\mathcal M_1,\mathcal M_2)=\frac{1}{n}\log IPBF^{(n,m,k)}(\mathcal M_1,\mathcal M_0)
	-\frac{1}{n}\log IPBF^{(n,m,k)}(\mathcal M_2,\mathcal M_0),$$ and then using (\ref{eq:inverse_cons3_2}) for 
	$\frac{1}{n}\log IPBF^{(n,m,k)}(\mathcal M_1,\mathcal M_0)$ and $\frac{1}{n}\log IPBF^{(n,m,k)}(\mathcal M_2,\mathcal M_0)$.
\end{proof}
\begin{remark}
	\label{remark:inverse2}
	As in Remark \ref{remark:inverse1} note that the result of Theorem \ref{theorem:inverse_cons3_mis} holds without the assumption that 
	$\Theta_0$ is complete separable and $f(y_{ik}|\theta,\tilde x_i,\bY^{(i-1)},\mathcal M_0)$ is bounded and continuous in $(\theta,\tilde x_i)$ for $k\geq 1$, 
	irrespective of whether or not $x_i$ is treated as known for cross-validation with respect to $\mathcal M_0$. 
	In this case, assuming the rest of the conditions of Theorem \ref{theorem:inverse_cons3_mis},
	it holds for any $k\geq 1$, that 
	\begin{align*}
		&\underset{m\rightarrow\infty}{\lim}\underset{n\rightarrow\infty}{\lim}~\frac{1}{n}\log IPBF^{(n,m,k)}(\mathcal M_1,\mathcal M_2)\\
		&=\underset{m\rightarrow\infty}{\lim}\underset{n\rightarrow\infty}{\lim}~
		\frac{1}{n}\log\left[\frac{\prod_{i=1}^n\pi(y_{ik}|\bY_{nm,-i},\bX_{nm,-i},\mathcal M_1)}{\prod_{i=1}^n\pi(y_{ik}|\bY_{nm,-i},\bX_{nm,-i},\mathcal M_2)}\right]
		\stackrel{a.s.}{=} -h^*(\tilde\theta_1,\tilde\theta_2),%f(y_i|\tilde\theta,x_i,\mathcal M),
		%~\mbox{as}~n\rightarrow\infty,
		%\label{eq:inverse_cons2}
	\end{align*}
		where, for any $\theta_1,\theta_2$,
		\begin{equation*}
		h^*(\theta_1,\theta_2)=
		\underset{n\rightarrow\infty}{\lim}~\frac{1}{n}
			\sum_{i=1}^n\log\left[\frac{f(y_{ik}|\theta_2,x^*_{i2},\bY^{(i-1)}_k,\mathcal M_2)}{f(y_{ik}|\theta_1,x^*_{i1},\bY^{(i-1)}_k,\mathcal M_1)}\right],
		%E_{\theta_0}\left\{\sum_{i=1}^n\log\left[\frac{g(y_i,\theta_0,\mathcal M_0)}{g(y_i,\theta,\mathcal M)}\right]\right\}.
			%\label{eq:h_inverse}
		\end{equation*}
		provided that the limit exists.
		%The proof follows in the same way as in Theorem \ref{theorem:inverse_cons2} by replacing $\mathcal M$ and $\mathcal M_0$ with
		%$\mathcal M_1$ and $\mathcal M_2$. 
%
		As in Remark \ref{remark:inverse1}, again $h^*(\tilde\theta_1,\tilde\theta_2)$ above is the same as 
		$h^*(\tilde\theta_1)-h^*(\tilde\theta_2)$ of Theorem \ref{theorem:inverse_cons3_mis}, although, unlike the latter,
		the former meed not be interpretable as the difference between limiting KL-divergence rates for $\mathcal M_1$ and $\mathcal M_2$.
		%while the former does not admit
		%such desirable interpretation since without the assumptions $\Theta_0$ is complete separable and $f(y_i|\theta,x_i,\mathcal M_0)$ is 
		%bounded and continuous in $\theta$, the convergence
	%\begin{equation*}
	%	\frac{1}{n}\sum_{i=1}^n\log\pi(y_i|\bY_{n,-i},\bX_{n,-i},\mathcal M_0)\stackrel{a.s.}{\longrightarrow} 
	%\underset{n\rightarrow\infty}{\lim}~\frac{1}{n}\sum_{i=1}^n\log f(y_i|\theta_0,x_i\mathcal M_0),~\mbox{as}~n\rightarrow\infty,
	%	%\label{eq:inverse_cons4}
	%\end{equation*}
	%need not hold, even if $x_i$ is considered known for cross-validation with respect to $\mathcal M_0$.
\end{remark}

\section{Illustrations of PBF convergence in forward regression problems}
\label{sec:illustrations_forward}

\subsection{Forward linear regression model}
\label{subsec:forward1}

Let 
\begin{equation}
\mathcal M_1: y_i=\alpha+\beta x_i+\epsilon_i;~i=1,\ldots,n,
	\label{eq:M_forward1}
\end{equation}
where $\epsilon_i\sim N\left(0,\sigma^2_{\epsilon}\right)$ independently, for $i=1,\ldots,n$. 
Here $\theta=(\alpha,\beta,\sigma^2_{\epsilon})$ is the unknown set of parameters.
Let the parameter space be $\Theta=\mathbb R\times\mathbb R\times\mathbb R^+$. Clearly, $\Theta$ is complete and separable.

Also let
\begin{equation}
	\mathcal M_0: y_i=\eta_0(x_i)+\epsilon_i;~i=1,\ldots,n,
	\label{eq:M_forward2}
\end{equation}
where $\eta_0(x)$ is the true, non-linear function of $x$, which is also continuous, and $\epsilon_i\sim N\left(0,\sigma^2_{0}\right)$ independently, for $i=1,\ldots,n$.
In this 

Let us assume that $\mathcal X$, the covariate space, is compact, under both $\mathcal M_1$ and $\mathcal M_0$.

\subsubsection{Verification of the assumptions}
\label{subsubsec:verification_forward1}
From (\ref{eq:M_forward1}) it is clear that $f(y_i|\theta,x_i,\bY^{(i-1)},\mathcal M_1)=f(y_i|\theta,x_i,\mathcal M_1)$ is bounded and continuous in $\theta$, 
and the true model $f(y_i|x_i,\bY^{(i-1)},\mathcal M_0)=f(y_i|x_i,\mathcal M_0)$ is devoid
of any parameters. Consequently, in this case, $\pi(y_i|\bY_{n,-i},\bX_n,\mathcal M_0)\equiv f(y_i|x_i,\mathcal M_0)$.

We are now left to verify the seven assumptions of Shalizi.
First note from the forms of (\ref{eq:M_forward1}) and (\ref{eq:M_forward2}) 
that measurability of $R_n(\theta)$ clearly holds, so that the first assumption of Shalizi, namely, (S1) is satisfied.

Now,
\begin{align}
	&\frac{1}{n}\log \prod_{i=1}^nf(y_i|\theta,x_i,\mathcal M_1)=-\frac{1}{2}\log 2\pi\sigma^2_{\epsilon}-\frac{1}{2\sigma^2_{\epsilon}n}\sum_{i=1}^n(y_i-\eta_0(x_i))^2
	-\frac{1}{2\sigma^2_{\epsilon}n}\sum_{i=1}^n(\eta_0(x_i)-\alpha-\beta x_i)^2\notag\\
	&\qquad\qquad-\frac{1}{\sigma^2_{\epsilon}n}\sum_{i=1}^n(y_i-\eta_0(x_i))(\eta_0(x_i)-\alpha-\beta x_i).
	\label{eq:log1}
\end{align}
In (\ref{eq:log1}), 
\begin{equation}
	\frac{1}{n}\sum_{i=1}^n(y_i-\eta_0(x_i))^2\stackrel{a.s.}{\longrightarrow}\sigma^2_0,~\mbox{as}~n\rightarrow\infty,
	\label{eq:log2}
\end{equation}
and letting $|\mathcal X|$ denote the Lebesgue measure of the compact space $\mathcal X$,
\begin{equation}
	\frac{1}{n}\sum_{i=1}^n(\eta_0(x_i)-\alpha-\beta x_i)^2\rightarrow|\mathcal X|^{-1}\int_{\mathcal X}(\eta_0(x)-\alpha-\beta x)^2dx,~\mbox{as}~n\rightarrow\infty,
	\label{eq:log3}
\end{equation}
since the former is a Riemann sum. Also, letting $E_0$ and $V_0$ denote the mean and variance under model $\mathcal M_0$, we see that for all $n\geq 1$, 
\begin{equation}
	E_0\left[\frac{1}{n}\sum_{i=1}^n(y_i-\eta_0(x_i))(\eta_0(x_i)-\alpha-\beta x_i)\right]=0,
	\label{eq:E0}
\end{equation}
and
\begin{equation}
	\sum_{i=1}^{\infty}\frac{V_0\left[(y_i-\eta_0(x_i))(\eta_0(x_i)-\alpha-\beta x_i)\right]}{i^2}
	\leq\sigma^2_0~\underset{x\in\mathcal X}{\sup}~\left(\eta_0(x)-\alpha-\beta x\right)^2\sum_{i=1}^{\infty}\frac{1}{i^2}<\infty.
	\label{eq:V0}
\end{equation}
From (\ref{eq:E0}) and (\ref{eq:V0}), it follows from Kolmogorov's strong law of large numbers for independent but non-identical random variables,
\begin{equation}
	\frac{1}{n}\sum_{i=1}^n(y_i-\eta_0(x_i))(\eta_0(x_i)-\alpha-\beta x_i)\stackrel{a.s.}{\longrightarrow} 0,~\mbox{as}~n\rightarrow\infty.
	\label{eq:slln1}
\end{equation}
Applying (\ref{eq:log2}), (\ref{eq:log3}) and (\ref{eq:slln1}) to (\ref{eq:log1}) yields
\begin{equation}
	\frac{1}{n}\log \prod_{i=1}^nf(y_i|\theta,x_i,\mathcal M_1)\stackrel{a.s.}{\longrightarrow}-\frac{1}{2}\log 2\pi\sigma^2_{\epsilon}
	-\frac{|\mathcal X|^{-1}}{2\sigma^2_{\epsilon}}\int_{\mathcal X}(\eta_0(x)-\alpha-\beta x)^2dx,~\mbox{as}~n\rightarrow\infty.
	\label{eq:M_limit}
\end{equation}
Now observe that for the true model $\mathcal M_0$,
\begin{equation}
	\frac{1}{n}\log \prod_{i=1}^nf(y_i|x_i,\mathcal M_0)=-\frac{1}{2}\log 2\pi\sigma^2_0-\frac{1}{2\sigma^2_0n}\sum_{i=1}^n(y_i-\eta_0(x_i))^2
	\stackrel{a.s.}{\longrightarrow}-\frac{1}{2}\log 2\pi\sigma^2_0-\frac{1}{2},~\mbox{as}~n\rightarrow\infty.
	\label{eq:M0_limit}
\end{equation}
From (\ref{eq:M_limit}) and (\ref{eq:M0_limit}) we have, for $\theta\in\Theta$,
\begin{equation*}
	\frac{1}{n}\log R_n(\theta)\stackrel{a.s.}{\longrightarrow}-h_1(\theta), 
\end{equation*}
where
\begin{equation}
	h_1(\theta)=\frac{1}{2}\log\left(\frac{\sigma^2_{\epsilon}}{\sigma^2_0}\right)+\frac{\sigma^2_0}{2\sigma^2_{\epsilon}}
	+\frac{|\mathcal X|^{-1}}{2\sigma^2_{\epsilon}}\int_{\mathcal X}(\eta_0(x)-\alpha-\beta x)^2dx-\frac{1}{2}.
	\label{eq:h_linear_reg1}
\end{equation}
Hence, (S3) of Shalizi holds.

It is easy to see by taking the limits of the expectations of $\frac{1}{n}\log \prod_{i=1}^nf(y_i|\theta,x_i,\mathcal M_1)$ and
$\frac{1}{n}\log \prod_{i=1}^nf(y_i|x_i,\mathcal M_0)$, that the following also holds:
\begin{equation*}
\underset{n\rightarrow\infty}{\lim}~\frac{1}{n}E_0\left[\log R_n(\theta)\right]=-h_1(\theta).
\end{equation*}
In other words, (S2) holds.

Note that $h_1(\theta)<\infty$ almost surely if under the priors for $\alpha,\beta,\sigma^2_{\epsilon}$, $|\alpha|<\infty$, $|\beta|<\infty$ and
$0<\sigma^2_{\epsilon}<\infty$, almost surely.
Hence, (S4) holds.

Let 
\begin{equation}
	\mathcal G_n=\left\{\theta\in\Theta:|\alpha|\leq \exp\left(\gamma n\right),|\beta|\leq \exp\left(\gamma n\right),
	\sigma^{-2}_{\epsilon}\leq \exp\left(\gamma n\right)\right\},
	\label{eq:Gn}
\end{equation}
where $\gamma>2h\left(\Theta\right)$. Then $\mathcal G_n\uparrow\Theta$, as $n\rightarrow\infty$. 

Let us assume that the prior for $\left(\alpha,\beta,\sigma^{-2}_{\epsilon}\right)$ is such that the prior expectations 
$E(|\alpha|)$, $E(|\beta|)$ and $E(\sigma^{-2}_{\epsilon})$ are finite.
Then under such priors, using Markov's inequality, the probabilities $P\left(|\alpha|>\exp\left(\gamma n\right)\right)$, $P\left(|\beta|>\exp\left(\gamma n\right)\right)$ and 
$P\left(\sigma^{-2}_{\epsilon}>\exp\left(\gamma n\right)\right)$ are bounded above as follows:
\begin{align}
	&P\left(|\alpha|>\exp\left(\gamma n\right)\right)<E\left(|\alpha|\right)\exp\left(-\gamma n\right);
	\label{eq:alpha_prob}\\
	&P\left(|\beta|>\exp\left(\gamma n\right)\right)<E\left(|\beta|\right)\exp\left(-\gamma n\right);
	\label{eq:beta_prob}\\
	&P\left(\sigma^{-2}_{\epsilon}>\exp\left(\gamma n\right)\right)<E\left(\sigma^{-2}_{\epsilon}\right)\exp\left(-\gamma n\right).
	\label{eq:sigma_prob}
\end{align}
From (\ref{eq:Gn}) and the inequalities (\ref{eq:alpha_prob}), (\ref{eq:beta_prob}) and (\ref{eq:sigma_prob}) it follows that
\begin{align}
	\pi(\mathcal G_n)&\geq1-\left(P(|\alpha|>\exp\left(\gamma n\right))+P(|\beta|>\exp\left(\gamma n\right))
	+P(\sigma^{-2}_{\epsilon}>\exp\left(\gamma n\right))\right)\notag\\
	&\geq 1-\left(E\left(|\alpha|\right)+E\left(|\beta|\right)+E\left(\sigma^{-2}_{\epsilon}\right)\right)\exp\left(-\gamma n\right).
	\label{eq:prob_Gn}
\end{align}
Thus, (S5)(1) holds.

The differential of $\frac{1}{n}\log R_n(\theta)$ is continuous in $\theta$, and since $\mathcal X$ is compact, it is easy to see that
the differential is almost surely bounded on any compact subset $G$ of $\Theta$, as $n\rightarrow\infty$. That is, 
$\frac{1}{n}\log R_n(\theta)$ is almost surely Lipschitz, hence, equicontinuous on $G$. Since $\frac{1}{n}\log R_n(\theta)$ almost surely converges to $-h_1(\theta)$
pointwise, as $n\rightarrow\infty$, it holds due to the stochastic Ascoli lemma that
	\begin{equation}
		\underset{n\rightarrow\infty}{\lim}\underset{\theta\in G}{\sup}~\left|\frac{1}{n}\log R_n(\theta)+h_1(\theta)\right|=0,~\mbox{almost surely}.
		\label{eq:sa1}
	\end{equation}
Since for any $n\geq 1$, $\mathcal G_n$ is compact, (S5)(2) holds.

Since $h_1(\theta)$ is continuous in $\theta$, $\mathcal G_n$ is compact and $h\left(\mathcal G_n\right)$ is non-increasing in $n$, (S5)(3) holds.
Also, for any set $A$ such that $\pi(A)>0$, since $\mathcal G_n\cap A$ increases to $A$, it follows due to continuity of $h_1(\theta)$ that 
$h\left(\mathcal G_n\cap A\right)$ decreases to $h_1(A)$, so that (S7) holds.

Regarding verification of (S6), observe that the aim of assumption (S6) is to ensure that (see the proof of Lemma 7 of \ctn{Shalizi09}) 
for every $\varepsilon>0$ and for all $n$ sufficiently large,
\begin{equation*}
\frac{1}{n}\log\int_{\mathcal G_n}R_n(\btheta)d\pi(\btheta)\leq -h\left(\mathcal G_n\right)+\varepsilon,~\mbox{almost surely}.
\end{equation*}
Since $h\left(\mathcal G_n\right)\rightarrow h\left(\bTheta\right)$ as $n\rightarrow\infty$, 
it is enough to verify that 
for every $\varepsilon>0$ and for all $n$ sufficiently large,
\begin{equation*}
\frac{1}{n}\log\int_{\mathcal G_n}R_n(\btheta)d\pi(\btheta)\leq -h\left(\bTheta\right)+\varepsilon,~\mbox{almost surely}.
%\label{eq:s6_1}
\end{equation*}
In other words, it is sufficient to verify that
	\begin{equation}
		\underset{n\rightarrow\infty}{\limsup}~\frac{1}{n}\log\int_{\mathcal G_n}R_n(\theta)\pi(\theta)d\theta\leq -h\left(\Theta\right),~\mbox{almost surely}.
		\label{eq:shalizi_Rn}
	\end{equation}
Theorem \ref{theorem:shalizi_Rn} stated and proved in Appendix \ref{sec:s6} provides sufficient conditions for (\ref{eq:shalizi_Rn}) to hold in general with 
proper priors on the parameters.
We now make use of Theorem \ref{theorem:shalizi_Rn} of Appendix \ref{sec:s6} to validate (S6) of Shalizi. 
For any function $g(x)$ on $\mathcal X$, let us consider the notation 
\begin{equation}
	E_X\left[g(X)\right]=|\mathcal X|^{-1}\int_{\mathcal X}g(x)dx.
	\label{eq:EX}
\end{equation}
Note that (\ref{eq:EX}) is indeed the expectation of $g(X)$ with respect to the uniform distribution on the compact set $\mathcal X$.

Now observe that $h_1(\theta)$ is uniquely minimized by
\begin{align}
	&\tilde\beta = \frac{E_X\left[(X-E_X(X))(\eta_0(X)-E(\eta_0(X)))\right]}{E_X(X-E_X(X))^2};\label{eq:beta1}\\
	&\tilde\alpha = E_X(\eta_0(X))-\tilde\beta E_X(X);\label{eq:alpha1}\\
	&\tilde\sigma^2_{\epsilon}=\sigma^2_0+E_X\left(\eta_0(X)-\tilde\alpha-\tilde\beta X\right)^2.\label{eq:sigma1}
\end{align}
Now, letting $\bar x_n=\frac{\sum_{i=1}^nx_i}{n}$, $\bar y_n=\frac{\sum_{i=1}^ny_i}{n}$ and 
$\bar\eta_{0n}=\frac{\sum_{i=1}^n\eta_0(x_i)}{n}$, we see that $\frac{1}{n}\log R_n(\theta)$ is minimized at
\begin{align}
	\tilde \beta^*_n &=%\frac{\sum_{i=1}^n(\eta_0(x_i)-\bar\eta_{0n})(x_i-\bar x_n)}{2\sum_{i=1}^n(x_i-\bar x_n)^2}
	\frac{\sum_{i=1}^n(y_i-\bar y_n)(x_i-\bar x_n)}{\sum_{i=1}^n(x_i-\bar x_n)^2};\label{eq:beta2}\\
	\tilde\alpha^*_n&=\bar y_n-\tilde\beta^*_n\bar x_n;\label{eq:alpha2}\\
	\tilde {\sigma^*}^2_n&=\frac{1}{n}\left[\sum_{i=1}^n(y_i-\eta_0(x_i))^2+\sum_{i=1}^n\left(\eta_0(x_i)-\tilde\alpha^*_n-\tilde\beta^*_nx_i\right)^2\right.\notag\\
	&\left.\qquad+2\sum_{i=1}^n(y_i-\eta_0(x_i))\left(\eta_0(x_i)-\tilde\alpha^*_n-\tilde\beta^*_n x_i\right)\right].
	\label{eq:sigma2}
\end{align}
%Since $\bar x_n\rightarrow E_X(X)$ and $\bar\eta_{0n}\rightarrow E_X(\eta_0(X))$ due to Riemann sum convergence and 
%$\bar y_n\rightarrow
%Note that the first term on the right hand side of (\ref{eq:beta2}) converges to $\frac{E_X(\eta_0(X)-E_X(\eta_0(X)))(X-E_X(X))}{2E_X(X-E_X(X))^2}$ thanks
%to Riemann sum convergence
Using Kolmogorov's strong law of large numbers and Riemann sum convergence, we see that 
%the second term can be seen to converge almost surely to $\frac{E_X(\eta_0(X)-E_X(\eta_0(X)))(X-E_X(X))}{2E_X(X-E_X(X))^2}$.
%In other words, 
\begin{equation}
\tilde\beta^*_n\stackrel{a.s.}{\longrightarrow}\tilde\beta, 
	\label{eq:beta3}
\end{equation}
where $\tilde\beta$ is given by (\ref{eq:beta1}). 

By (\ref{eq:beta3}), and since $\bar y_n\stackrel{a.s.}{\longrightarrow}E_X(\eta_0(X))$, $\bar x_n\rightarrow E_X(X)$, it follows that
\begin{equation}
\tilde\alpha^*_n\stackrel{a.s.}{\longrightarrow}\tilde\alpha, 
	\label{eq:alpha3}
\end{equation}
where $\tilde\alpha$ is given by (\ref{eq:alpha1}). 

For the convergence of $\tilde {\sigma^*}^2_n$ given by (\ref{eq:sigma2}), first observe that the first term on the right hand side of (\ref{eq:sigma2})
converges almost surely to $\sigma^2_0$. The $i$-th term of the second term on the right hand side converges to $(\eta_0(x_i)-\tilde\alpha-\tilde\beta x_i)^2$ 
almost surely, so that
the second term converges to $E_X(\eta_0(X)-\tilde\alpha-\tilde\beta X)^2$. The $i$-th term of the third term on the right hand side converges almost surely to
$2(y_i-\eta_0(x_i))(\eta_0(x_i)-\tilde\alpha-\tilde\beta x_i)$, so that the third term converges to zero almost surely due to (\ref{eq:slln1}). It follows that
\begin{equation}
	\tilde {\sigma^*}^2_n\stackrel{a.s.}{\longrightarrow}\tilde\sigma^2_{\epsilon}, 
	\label{eq:sigma3}
\end{equation}
where $\tilde\sigma^2_{\epsilon}$ is given by (\ref{eq:sigma1}). 
Combining (\ref{eq:beta3}), (\ref{eq:alpha3}) and (\ref{eq:sigma3}) yields
\begin{equation}
	\tilde\theta^*_n=\left(\tilde\alpha^*_n,\tilde\beta^*_n,\tilde {\sigma^*}^2_n\right)\stackrel{a.s.}{\longrightarrow}
	\left(\tilde\alpha,\tilde\beta,\tilde\sigma^2_{\epsilon}\right)=\tilde\theta,~\mbox{as}~n\rightarrow\infty.
	\label{eq:theta_conv1}
\end{equation}

In other words, we have shown that conditions (i) and (ii) of Theorem \ref{theorem:shalizi_Rn} hold. Since we have already shown 
pointwise almost sure convergence of $\frac{1}{n}\log R_n(\theta)$ to $-h_1(\theta)$ in the context of verifying (S3) and
stochastic equicontinuity of
$\frac{1}{n}\log R_n(\theta)$ on compact subsets of $\Theta$ in the context of verifying (S5)(2), all the conditions of Theorem \ref{theorem:shalizi_Rn}
go through with proper prior for $\theta$.
Hence (\ref{eq:shalizi_Rn}), and consequently, (S6), holds.

With these, it is seen that the conditions of Theorem \ref{theorem:forward_cons2} are satisfied, which leads to the following specialized version
of the theorem:
\begin{theorem}
	\label{theorem:linreg}
	Consider the linear regression model $\mathcal M_1$ given by (\ref{eq:M_forward1}) and the true, non-linear model $\mathcal M_0$ given by
	(\ref{eq:M_forward2}). Assume the parameter space $\Theta$ associated with model $\mathcal M_1$ be $\mathbb R\times\mathbb R\times\mathbb R^+$, and
	let the covariate space $\mathcal X$ be compact. Then (\ref{eq:forward_cons2}) holds for $\frac{1}{n}\log FPBF^{(n)}(\mathcal M_1,\mathcal M_0)$, 
	where for $\theta\in\Theta$, $h(\theta)=h_1(\theta)$ is given by
	(\ref{eq:h_linear_reg1}), and $\tilde\theta=\left(\tilde\alpha,\tilde\beta,\tilde\sigma^2_{\epsilon}\right)$, where $\tilde\alpha$, $\tilde\beta$
	and $\tilde\sigma^2_{\epsilon}$ are given by (\ref{eq:alpha1}), (\ref{eq:beta1}) and (\ref{eq:sigma1}), respectively.
\end{theorem}

\subsection{Forward quadratic regression model}
\label{subsec:quadreg}
Now consider the following model on quadratic regression which may be regarded as a competitor to linear regression:
\begin{equation}
\mathcal M_2: y_i=\alpha+\beta_1 x_i+\beta_2 x^2_i+\epsilon_i;~i=1,\ldots,n,
	\label{eq:M_forward3}
\end{equation}
where $\epsilon_i\sim N\left(0,\sigma^2_{\epsilon}\right)$ independently, for $i=1,\ldots,n$. 
Here $\theta=(\alpha,\beta_1,\beta_2,\sigma^2_{\epsilon})$ is the unknown set of parameters, and the parameter space is 
$\Theta=\mathbb R\times\mathbb R\times\mathbb R\times\mathbb R^+$. %Clearly, $\Theta$ is complete and separable.

In this case,
\begin{equation*}
	\frac{1}{n}\log R_n(\theta)\stackrel{a.s.}{\longrightarrow}-h_2(\theta), 
\end{equation*}
where
\begin{equation}
	h_2(\theta)=\frac{1}{2}\log\left(\frac{\sigma^2_{\epsilon}}{\sigma^2_0}\right)+\frac{\sigma^2_0}{2\sigma^2_{\epsilon}}
	+\frac{|\mathcal X|^{-1}}{2\sigma^2_{\epsilon}}\int_{\mathcal X}(\eta_0(x)-\alpha-\beta_1 x-\beta_2 x^2)^2dx-\frac{1}{2}.
	\label{eq:h_quad_reg1}
\end{equation}
It is easy to see that $h_2(\theta)$ is uniquely minimized at $\tilde\vartheta=(\tilde\alpha,\tilde\beta_1,\tilde\beta_2)$, given by
\begin{equation}
\tilde\vartheta=A^{-1}b, 
	\label{eq:ls_quad_true1}
\end{equation}
	where
\begin{equation}
	A=\left(\begin{array}{ccc}1 & E_X(X) & E_X(X^2)\\
		E_X(X) & E_X(X^2) & E_X(X^3)\\
		E_X(X^2) & E_X(X^3) & E_X(X^4)\end{array}\right)~\mbox{and}~b=\left(\begin{array}{c}E_X(\eta_0(X))\\ E_X(X\eta_0(X))\\ E_X(X^2\eta_0(X))\end{array}\right),
			\label{eq:A_b}
\end{equation}
and
	\begin{equation}
		\tilde\sigma^2_{\epsilon}=\sigma^2_0+E_X\left(\eta_0(X)-\tilde\alpha-\tilde\beta_1 X-\tilde\beta_1 X^2\right)^2.
		\label{eq:sigma_quad}
	\end{equation}
	That $A$ in (\ref{eq:A_b}) is invertible, will be shown shortly.

	The maximizer of $\frac{1}{n}\log R_n(\theta)$ here is given by the least squares estimators 
	$\tilde\vartheta^*_n=(\tilde\alpha^*_n,\tilde\beta^*_{1n},\tilde\beta^*_{2n})$
	given by 
	\begin{equation}
	\tilde\vartheta^*_n=A^{-1}_nb_n, 
		\label{eq:ls_quad1}
	\end{equation}
		where 
\begin{equation}
	A_n=n^{-1}\left(\begin{array}{ccc}n & \sum_{i=1}^nx_i & \sum_{i=1}^nx^2_i\\
		\sum_{i=1}^nx_i & \sum_{i=1}^nx^2_i & \sum_{i=1}^nx^3_i\\
	\sum_{i=1}^nx^2_i & \sum_{i=1}^nx^3_i & \sum_{i=1}^nx^4_i\end{array}\right)
		~\mbox{and}~b_n=n^{-1}\left(\begin{array}{c}\sum_{i=1}^n\eta_0(x_i)\\ \sum_{i=1}^nx_i\eta_0(x_i)\\ 
		\sum_{i=1}^nx^2_i\eta_0(x_i)\end{array}\right),
			\label{eq:A_b_n}
\end{equation}
and
	\begin{align}
		&\tilde {\sigma^*}^2_n=\frac{1}{n}\left[\sum_{i=1}^n(y_i-\eta_0(x_i))^2+\sum_{i=1}^n\left(\eta_0(x_i)-\tilde\alpha^*_n
		-\tilde\beta^*_{1n}x_i-\tilde\beta^*_{2n} x^2_i\right)^2\right.\notag\\
		&\qquad\left.\qquad+2\sum_{i=1}^n(y_i-\eta_0(x_i))\left(\eta_0(x_i)-\tilde\alpha^*_n-\tilde\beta^*_{1n} x_i-\tilde\beta^*_{2n} x^2_i\right)\right].
	\label{eq:sigma2_quad}
	\end{align}
	Now note that $A_n$ in (\ref{eq:A_b_n}) corresponds to the so-called Vandermonde design matrix (see, for example, \ctn{Macon58}) 
	associated with the least squares quadratic regression.
	The design matrix if of full rank if all the $x_i$ are distinct, which we assume. Hence, for all $n\geq 3$, $A_n$ is invertible, which makes
	the least squares estimators $\tilde\vartheta^*_n$, given by (\ref{eq:ls_quad1}), well-defined, for all $n\geq 3$. Now observe that by Riemann sum convergence,
	\begin{align}
		&A_n\stackrel{a.s.}{\longrightarrow}A,~\mbox{as}~n\rightarrow\infty,~\mbox{and}\label{eq:A_conv1}\\
		&b_n\stackrel{a.s.}{\longrightarrow}b,~\mbox{as}~n\rightarrow\infty.\label{eq:b_conv1}
	\end{align}
	Since $A_n$ is invertible for every $n\geq 3$, $A$ must also be invertible, since (\ref{eq:A_conv1}) holds. Hence, $\tilde\vartheta$ given by
	(\ref{eq:ls_quad_true1}), is well-defined.

	Now, thanks to (\ref{eq:A_conv1}) and (\ref{eq:b_conv1}), we have
	\begin{equation*}
		\tilde\vartheta^*_n\stackrel{a.s.}{\longrightarrow}\tilde\vartheta,~\mbox{as}~n\rightarrow\infty,
	\end{equation*}
	and also in the same way as for model $\mathcal M_1$, here also,
	\begin{equation*}
		\tilde {\sigma^*}^2_n\stackrel{a.s.}{\longrightarrow}\tilde\sigma^2_{\epsilon},~\mbox{as}~n\rightarrow\infty.
	\end{equation*}
	In other words,
	\begin{equation*}
		\tilde\theta^*_n\stackrel{a.s.}{\longrightarrow}\tilde\theta,~\mbox{as}~n\rightarrow\infty,
	\end{equation*}
	even for model $\mathcal M_2$.
	
	For this quadratic regression model, let 
\begin{equation*}
	\mathcal G_n=\left\{\theta\in\Theta:|\alpha|\leq \exp\left(\gamma n\right),|\beta_1|\leq \exp\left(\gamma n\right),|\beta_2|\leq \exp\left(\gamma n\right),
	\sigma^{-2}_{\epsilon}\leq \exp\left(\gamma n\right)\right\},
	%\label{eq:Gn_quad}
\end{equation*}
where $\gamma>2h\left(\Theta\right)$. Then $\mathcal G_n\uparrow\Theta$, as $n\rightarrow\infty$, and the rest of the assumptions of Shalizi are easily seen to be
satisfied. The condition of boundedness and continuity of $f(y_i|\theta,x_i,\mathcal M_2)$ are also clearly satisfied.

	%Denoting $\tilde\theta$ associated with models $\mathcal M_1$ and $\mathcal M_2$ by $\tilde\theta_1$ and $\tilde\theta_2$ respectively, we 
	We summarize our results on FPBF consistency in favour of $\mathcal M_0$ %and asymptotic comparison between $\mathcal M_1$ and 
	when the data is modeled by $\mathcal M_2$ as follows.

\begin{theorem}
	\label{theorem:quadreg1}
	Consider the quadratic regression model $\mathcal M_2$ given by (\ref{eq:M_forward3}) and the true, non-linear model $\mathcal M_0$ given by
	(\ref{eq:M_forward2}). Assume the parameter space $\Theta$ associated with model $\mathcal M_2$ be $\mathbb R\times\mathbb R\times\mathbb R\times\mathbb R^+$, and
	let the covariate space $\mathcal X$ be compact. Also assume that $x_i;i\geq 1$ are all distinct. 
	Then (\ref{eq:forward_cons2}) holds for $\frac{1}{n}\log FPBF^{(n)}(\mathcal M_2,\mathcal M_0)$, 
	where for $\theta\in\Theta$, $h(\theta)=h_2(\theta)$ is given by
	(\ref{eq:h_quad_reg1}), and $\tilde\theta=\left(\tilde\alpha,\tilde\beta_1,\tilde\beta_2,\tilde\sigma^2_{\epsilon}\right)$, where $\tilde\alpha$, $\tilde\beta_1$,
	$\tilde\beta_2$ and $\tilde\sigma^2_{\epsilon}$ are given by (\ref{eq:ls_quad_true1}) and (\ref{eq:sigma_quad}).
\end{theorem}

\subsection{Asymptotic comparison of forward linear and quadratic models with FPBF}
\label{subsec:compare_linear_quadratic}
Theorems \ref{theorem:linreg} and \ref{theorem:quadreg1} show almost sure exponential convergence of FPBF in favour of the true model $\mathcal M_0$ 
given by (\ref{eq:M_forward2}) when the postulated models are either the forward linear or quadratic regression model.
Now, if the goal is to make asymptotic comparison between the linear and quadratic regression models, then the aforementioned theorems %(see also Theorem)
ensure the following result:
	%Denoting $\tilde\theta$ associated with models $\mathcal M_1$ and $\mathcal M_2$ by $\tilde\theta_1$ and $\tilde\theta_2$ respectively, we 
\begin{theorem}
	\label{theorem:linear_vs_quadreg1}
	Let the true model be given by $\mathcal M_0$ formulated in (\ref{eq:M_forward2}). Assuming that the covariate observations $x_i$; $i\geq 1$ are all distinct
	and that the covariate space $\mathcal X$ is compact, consider comparison of the linear and quadratic regression models $\mathcal M_1$ and $\mathcal M_2$
	given by (\ref{eq:M_forward1}) and (\ref{eq:M_forward3}), respectively. Let $\tilde\theta_1$ and $\tilde\theta_2$ be the unique minimizers of $h_1$ and $h_2$.
	%given by (\ref{eq:ls_quad_true1}) and (\ref{eq:sigma_quad})
	Then,
	\begin{equation*}
	\frac{1}{n}\log FPBF^{(n)}(\mathcal M_1,\mathcal M_2)\stackrel{a.s.}{\longrightarrow}-\left(h_1(\tilde\theta_1)-h_2(\tilde\theta_2)\right),~\mbox{as}~n\rightarrow\infty.
	\end{equation*}
\end{theorem}

\subsection{FPBF asymptotics for variable selection in autoregressive time series regression}
\label{subsec:ar1_forward}
Let us consider the following first order autoregressive (AR(1)) time series linear regression as model $\mathcal M_1$:
\begin{equation}
	y_t=\rho_1 y_{t-1}+\beta_1 x_t+\epsilon_{1t};~t=1,\ldots,n,
	\label{eq:ar1_x}
\end{equation}
where $y_0\equiv 0$ $x_t;t=1,\ldots,n$ are covariate observations associated with variable $x$ and 
$\epsilon_{1t}\stackrel{iid}{\sim}N\left(0,\sigma^2_{1}\right)$. Here $\theta_1=\left(\rho_1,\beta_1,\sigma^2_{1}\right)$ is the 
set of unknown parameters and 
$\Theta_1=\mathbb R\times\mathbb R\times\mathbb R^+$ is the parameter space. We might wish to compare this model with another AR(1) regression model
with covariate $z$ different from $x$. This model, which we refer to as $\mathcal M_2$, is given as follows: 
\begin{equation}
	y_t=\rho_2 y_{t-1}+\beta_2 z_t+\epsilon_{2t};~t=1,\ldots,n,
	\label{eq:ar1_z}
\end{equation}
where $y_0\equiv 0$ $z_t;t=1,\ldots,n$ are observations associated with covariate $z$ different from $x$ 
and $\theta_2=\left(\rho_2,\beta_2,\sigma^2_{2}\right)$ is the set of parameters and the parameter space 
$\Theta_2=\mathbb R\times\mathbb R\times\mathbb R^+$ remains the same as $\Theta_1$. Here, for $t=1,\ldots,n$, 
$\epsilon_{2t}\stackrel{iid}{\sim}N\left(0,\sigma^2_{2}\right)$. 
Let the true model $\mathcal M_0$ be given by
\begin{equation}
	y_t=\rho_0 y_{t-1}+\beta_0 (x_t+z_t)+\epsilon_{0t};~t=1,\ldots,n,
	\label{eq:ar1_both}
\end{equation}
where $|\rho_0|<1$ and $\epsilon_{0t}\stackrel{iid}{\sim}N\left(0,\sigma^2_{0}\right)$, for $t=1,\ldots,n$.
%Let $\mathcal X$ denote the common compact covariate space for both $x_t$ and $z_t$, $t\geq 1$.

Our goal in this example is to compare models $\mathcal M_1$ and $\mathcal M_2$ using FPBF. Note that if we use the same priors for $\theta_1$ and $\theta_2$,
this boils down to selection of either covariate $x$ or $z$ in the AR(1) regression. Hence, variable selection constitutes an important ingredient in this
FPBF convergence example. Note that both the models $\mathcal M_1$ and $\mathcal M_2$ are wrong with respect to the true model $\mathcal M_0$ which consists of both
$x$ and $z$. The purpose of variable selection here is then to select the more important variable among $x$ and $z$ when none of the available models considers
both $x$ and $z$.

We make the following assumptions that are analogous to the AR(1) regression example considered in \ctn{Chandra20}: 
\begin{enumerate}%[label={(C\arabic*)}]	
	\item[(A1)] \label{ar1} %$\left\{\bz_t;t=1,2,\ldots\right\}$ is a sample path of some asymptotically stationary stochastic process with 
	\begin{align}
		&\frac{1}{n}\sum_{t=1}^n x_t \rightarrow 0,~\frac{1}{n}\sum_{t=1}^n z_t \rightarrow 0;\notag\\ %~\mbox{and}~
		&\frac{1}{n}\sum_{t=1}^n x_tz_t \rightarrow 0;~\frac{1}{n}\sum_{t=1}^n x_{t+k}z_t \rightarrow 0;%\nonumber\\%\label{eq:ass1}\\
		~\frac{1}{n}\sum_{t=1}^n x_tz_{t+k} \rightarrow 0~\mbox{for any}~k\geq 1;\nonumber\\%\label{eq:ass1}\\
		&\frac{1}{n}\sum_{t=1}^n x_{t+k}x_t \rightarrow 0~\mbox{and}~\frac{1}{n}\sum_{t=1}^n z_{t+k}z_t \rightarrow 0~\mbox{for any}~k\geq 1;\nonumber\\%\label{eq:ass2}\\
		&\frac{1}{n}\sum_{t=1}^nx^2_t\rightarrow \sigma^2_x~\mbox{and}~\frac{1}{n}\sum_{t=1}^nz^2_t\rightarrow \sigma^2_z,\nonumber%\label{eq:ass3},
	\end{align}
	as $n\rightarrow\infty$. In the above, $\sigma^2_x$ and $\sigma^2_z$ are positive quantities.
	
	%$iid$ sequence of $m$-variate Gaussian random variables with mean $\bzero$ and positive definite covariance matrix $\bSigma_z$.
\item[(A2)] \label{ar2} $\underset{t\geq 1}{\sup}~|x_t\beta_0|<C$ and $\underset{t\geq 1}{\sup}~|z_t\beta_0|<C$, for some $C>0$.  
	%$|\bz^\prime_t\bbeta_0|<C$ for some $C>0$, for all $t=1,2,\ldots$. 
	%Here the norm $\|\cdot\|$ can be either the Euclidean norm or the sup-norm given by the maximum of the absolute values of the vector components.
	
%\item[(3)] \label{ar3} $\theta_0$ is an interior point of $\Theta$.
\end{enumerate}
Let $\frac{1}{n}\log R^{(1)}_n(\theta)$ and $\frac{1}{n}\log R^{(2)}_n(\theta)$ stand for $\frac{1}{n}\log R_n(\theta)$ for models $\mathcal M_1$ and $\mathcal M_2$, 
respectively.  
Also let $\sigma^2_{x+z}=\sigma^2_x+\sigma^2_z$. Then proceeding in the same way as in \ctn{Chandra20} it can be shown that
\begin{align}
	&\underset{n\rightarrow\infty}{\lim}~\frac{1}{n}\log R^{(1)}_n(\theta)\stackrel{a.s.}{=}-h_1(\theta),~\mbox{for all}~\theta\in\Theta_1;\label{eq:pointconv1}\\
	&\underset{n\rightarrow\infty}{\lim}~\frac{1}{n}\log R^{(2)}_n(\theta)\stackrel{a.s.}{=}-h_2(\theta),~\mbox{for all}~\theta\in\Theta_2,\label{eq:pointconv2}
\end{align}
and the above convergences are uniform on compact subsets of $\Theta_1$ and $\Theta_2$, respectively. In the above,
	\begin{multline}
	h_1(\theta)=\log\left(\frac{\sigma}{\sigma_0}\right)+\left(\frac{1}{2\sigma^2}-\frac{1}{2\sigma^2_0}\right)\left(\frac{\sigma^2_0}{1-\rho^2_0}
		+\frac{\beta^2_0\sigma^2_{x+z}}{1-\rho^2_0}\right)
		+\left(\frac{\rho^2}{2\sigma^2}-\frac{\rho^2_0}{2\sigma^2_0}\right)\left(\frac{\sigma^2_0}{1-\rho^2_0}+\frac{\beta^2_0\sigma^2_{x+z}}{1-\rho^2_0}\right)\\
		+\frac{1}{2\sigma^2}\beta^2\sigma^2_{x+z}-
		\frac{1}{2\sigma^2_0}\beta^2_0\sigma^2_{x+z}
		-\left(\frac{\rho}{\sigma^2}-\frac{\rho_0}{\sigma^2_0}\right)\left(\frac{\rho_0\sigma^2_0}{1-\rho^2_0}+\frac{\rho_0\beta^2_0\sigma^2_{x+z}}{1-\rho^2_0}\right)
		-\left(\frac{\beta}{\sigma^2}-\frac{\beta_0}{\sigma^2_0}\right)\sigma^2_{x+z}\beta_0+\frac{\sigma^2_z\beta(2\beta_0-\beta)}{2\sigma^2}.
		\label{eq:h1_ar1}
	\end{multline}
and
	\begin{multline}
	h_2(\theta)=\log\left(\frac{\sigma}{\sigma_0}\right)+\left(\frac{1}{2\sigma^2}-\frac{1}{2\sigma^2_0}\right)\left(\frac{\sigma^2_0}{1-\rho^2_0}
		+\frac{\beta^2_0\sigma^2_{x+z}}{1-\rho^2_0}\right)
		+\left(\frac{\rho^2}{2\sigma^2}-\frac{\rho^2_0}{2\sigma^2_0}\right)\left(\frac{\sigma^2_0}{1-\rho^2_0}+\frac{\beta^2_0\sigma^2_{x+z}}{1-\rho^2_0}\right)\\
		+\frac{1}{2\sigma^2}\beta^2\sigma^2_{x+z}-
		\frac{1}{2\sigma^2_0}\beta^2_0\sigma^2_{x+z}
		-\left(\frac{\rho}{\sigma^2}-\frac{\rho_0}{\sigma^2_0}\right)\left(\frac{\rho_0\sigma^2_0}{1-\rho^2_0}+\frac{\rho_0\beta^2_0\sigma^2_{x+z}}{1-\rho^2_0}\right)
		-\left(\frac{\beta}{\sigma^2}-\frac{\beta_0}{\sigma^2_0}\right)\sigma^2_{x+z}\beta_0+\frac{\sigma^2_x\beta(2\beta_0-\beta)}{2\sigma^2}.
		\label{eq:h2_ar1}
	\end{multline}
For $i=1,2$, for model $\mathcal M_i$, let
\begin{equation}
	\mathcal G^{(i)}_n=\left\{\theta\in\Theta_i:|\rho|\leq \exp\left(\gamma_i n\right),|\beta|\leq \exp\left(\gamma_i n\right),
	\sigma^{-2}_{\epsilon}\leq \exp\left(\gamma_i n\right)\right\},
	\label{eq:Gn_ar1}
\end{equation}
where $\gamma_i>2h_i\left(\Theta_i\right)$. Then $\mathcal G^{(i)}_n\uparrow\Theta_i$, as $n\rightarrow\infty$. 
Let us assume that under both $\mathcal M_1$ and $\mathcal M_2$, the prior for $\left(\rho,\beta,\sigma^{-2}_{\epsilon}\right)$ is such that the prior expectations 
$E(|\rho|)$, $E(|\beta|)$ and $E(\sigma^{-2}_{\epsilon})$ are finite. 
%Then in the same way as (\ref{eq:prob_Gn}), 
%Then under such priors, using Markov's inequality, the probabilities $P\left(|\alpha|>\exp\left(\gamma n\right)\right)$, $P\left(|\beta|>\exp\left(\gamma n\right)\right)$ and 
%$P\left(\sigma^{-2}_{\epsilon}>\exp\left(\gamma n\right)\right)$ are bounded above as follows:
%\begin{align}
%	&P\left(|\alpha|>\exp\left(\gamma n\right)\right)<E\left(|\alpha|\right)\exp\left(-\gamma n\right);
%	\label{eq:alpha_prob}\\
%	&P\left(|\beta|>\exp\left(\gamma n\right)\right)<E\left(|\beta|\right)\exp\left(-\gamma n\right);
%	\label{eq:beta_prob}\\
%	&P\left(\sigma^{-2}_{\epsilon}>\exp\left(\gamma n\right)\right)<E\left(\sigma^{-2}_{\epsilon}\right)\exp\left(-\gamma n\right).
%	\label{eq:sigma_prob}
%\end{align}
%From (\ref{eq:Gn}) and the inequalities (\ref{eq:alpha_prob}), (\ref{eq:beta_prob}) and (\ref{eq:sigma_prob}) it follows that
%\begin{align}
%	\pi(\mathcal G_n)&\geq1-\left(P(|\alpha|>\exp\left(\gamma n\right))+P(|\beta|>\exp\left(\gamma n\right))
%	+P(\sigma^{-2}_{\epsilon}>\exp\left(\gamma n\right))\right)\notag\\
%	&\geq 1-\left(E\left(|\alpha|\right)+E\left(|\beta|\right)+E\left(\sigma^{-2}_{\epsilon}\right)\right)\exp\left(-\gamma n\right).
%	\label{eq:prob_Gn}
%\end{align}
%Thus, (S5)(1) holds.

With these, conditions (S1)--(S5) and (S7) of Shalizi hold for $\mathcal M_1$ and $\mathcal M_2$ in the same way as the AR(1) regression example of \ctn{Chandra20}.
Thus verification of (S6) only remains, for which we begin with the following result. 

\begin{theorem}
	\label{th:concave}
	The functions $\frac{1}{n}\log R^{(1)}_n(\theta)$ and $\frac{1}{n}\log R^{(2)}_n(\theta)$ are asymptotically concave in $\theta$.
\end{theorem}
%\subsection{Proof of Theorem \ref{th:concave}}
\begin{proof}
	The proof follows in the same line as that of Theorem 17 of \ctn{Chandra20}.
%
%	Note that 
%	\begin{equation*}
%	\sup_{\theta\in\Theta} \frac{1}{n}\log R_n(\theta) = \sup_{\rho,\beta} \sup_{\sigma^2}\frac{1}{n}\log R_n(\theta)= -\inf_{\rho,\beta}~\log\left[ \frac{1}{n} \sum_{t=1}^{n} \left(x_t -\rho x_{t-1}-\beta z_t \right)^2\right]-\frac{1}{2}.
%	\end{equation*}
%	Since $\log$ is a monotonic function, minimizing $\log\left[\frac{1}{n} \sum_{t=1}^{n} \left(x_t -\rho x_{t-1}-\beta z_t \right)^2\right]$ is equivalent to minimizing $\frac{1}{n} \sum_{t=1}^{n} \left(x_t -\rho x_{t-1}-\beta z_t \right)^2=g_n(\rho,\beta)$, say. Now the Jacobian matrix $J$ of $g_n(\rho,\beta)$ is given by 
%	\begin{equation*}
%	J= \begin{bmatrix}
%	\frac{1}{n}\sum x^2_{t-1} & \frac{1}{n}\sum x_{t-1}z_t\\
%	\frac{1}{n}\sum x_{t-1}z_t & \frac{1}{n}\sum z^2_t
%	\end{bmatrix}.
%	\end{equation*}
%	\eqref{eq:xt_sq_average2}, \eqref{eq:z_tx_t_average3} together with the model assumptions \ref{ar1}-\ref{ar2} clearly shows that for large enough $n$, 
%	$J$ is positive-definite. Hence $g_n(\rho,\beta)$ is convex implying that $\frac{1}{n}\log R_n(\theta)$ is a concave function for large $n$.
\end{proof}

It is also easy to see that both $h_1(\theta)$ and $h_2(\theta)$ given by (\ref{eq:h1_ar1}) and (\ref{eq:h2_ar1}) are convex in $\theta$. Hence, there exist unique
minimizers $\tilde\theta_1$ and $\tilde\theta_2$, respectively, of $h_1$ and $h_2$.
Theorem \ref{th:uniroot} shows consistency of the unique roots of $\frac{1}{n}\log R^{(1)}_n(\theta)$ and $\frac{1}{n}\log R^{(2)}_n(\theta)$ for
$\tilde\theta_1$ and $\tilde\theta_2$, respectively.

%The above theorem ensures that for large enough $n$, the likelihood equation have unique mle. Rest we need to ensure the strong consistency of the mle for this dependent setup.

\begin{theorem}
	\label{th:uniroot}
	Given any $\eta>0$, $\frac{1}{n}\log R^{(1)}_n(\theta)$ and $\frac{1}{n}\log R^{(2)}_n(\theta)$ have their unique roots in the 
	$\eta$-neighbourhood of $\tilde\theta_1$ and $\tilde\theta_2$, respectively, almost surely, for large $n$.
\end{theorem}
\begin{proof}
	See Appendix \ref{sec:proof1}.
\end{proof}

%Theorem \ref{th:uniroot} essentially entails the strong consistency of the mle. This also leads to the verification of \ref{s6} required for posterior consistency. We formally state it in the following lemma.

For $i=1,2$, let $\tilde\theta^{(i)}_n$ stand for the unique maximizer of $\frac{1}{n}\log R^{(i)}_n(\theta)$. By Theorem \ref{th:uniroot}
\[\tilde\theta^{(i)}_n\stackrel{a.s.}{\longrightarrow}\tilde\theta^{(i)},~\mbox{for}~i=1,2,\]
which, in turn implies thanks to Theorem \ref{theorem:shalizi_Rn}, that (\ref{eq:shalizi_Rn}), and hence (S6) of Shalizi, holds for both $\mathcal M_1$ and $\mathcal M_2$.

In other words, models $\mathcal M_1$ and $\mathcal M_2$ satisfy conditions (S1)--(S7) of Shalizi.
We summarize below our results on variable selection in forward AR(1) regression framework.

\begin{theorem}[FPBF consistency for $\mathcal M_1$ versus $\mathcal M_0$]
	\label{theorem:ar1_forward1}
	Consider the AR(1) regression models $\mathcal M_1$ and $\mathcal M_0$ given by (\ref{eq:ar1_x}) and (\ref{eq:ar1_both}). 
	%and the true model $\mathcal M_0$ given by
	%(\ref{eq:ar1_both}). Assume the parameter spaces $\Theta_1$ and $\Theta_2$ for $\mathcal M_1$ and 
	%$\mathcal M_2$ be $\mathbb R\times\mathbb R\times\mathbb R^+$.
	%let the covariate space $\mathcal X$ be compact. %Also assume that $x_i;i\geq 1$ are all distinct. 
	Then under assumptions (A1) and (A2), 
	$$\underset{n\rightarrow\infty}{\lim}~\frac{1}{n}\log FPBF^{(n)}(\mathcal M_1,\mathcal M_0)\stackrel{a.s.}{=}-h_1(\tilde\theta_1),$$ 
	where $h_1$ is given by (\ref{eq:h1_ar1}) and $\tilde\theta_1$ is its unique minimizer.
\end{theorem}

\begin{theorem}[FPBF consistency for $\mathcal M_2$ versus $\mathcal M_0$]
	\label{theorem:ar1_forward2}
	Consider the AR(1) regression models $\mathcal M_2$ and $\mathcal M_0$ given by (\ref{eq:ar1_z}) and (\ref{eq:ar1_both}). 
	%and the true model $\mathcal M_0$ given by
	%(\ref{eq:ar1_both}). Assume the parameter spaces $\Theta_1$ and $\Theta_2$ for $\mathcal M_1$ and 
	%$\mathcal M_2$ be $\mathbb R\times\mathbb R\times\mathbb R^+$.
	%let the covariate space $\mathcal X$ be compact. %Also assume that $x_i;i\geq 1$ are all distinct. 
	Then under assumptions (A1) and (A2), 
	$$\underset{n\rightarrow\infty}{\lim}~\frac{1}{n}\log FPBF^{(n)}(\mathcal M_2,\mathcal M_0)\stackrel{a.s.}{=}-h_2(\tilde\theta_2),$$ 
	where $h_1$ is given by (\ref{eq:h2_ar1}) and $\tilde\theta_2$ is its unique minimizer.
\end{theorem}

\begin{theorem}[FPBF convergence for $\mathcal M_1$ versus $\mathcal M_2$]
	\label{theorem:ar1_forward3}
	Consider the AR(1) regression models $\mathcal M_1$ and $\mathcal M_2$ given by (\ref{eq:ar1_x}) and (\ref{eq:ar1_z}) 
	and the true model $\mathcal M_0$ given by
	(\ref{eq:ar1_both}). 
	%Assume the parameter spaces $\Theta_1$ and $\Theta_2$ for $\mathcal M_1$ and 
	%$\mathcal M_2$ be $\mathbb R\times\mathbb R\times\mathbb R^+$.
	%let the covariate space $\mathcal X$ be compact. %Also assume that $x_i;i\geq 1$ are all distinct. 
	Then under assumptions (A1) and (A2), 
	$$\underset{n\rightarrow\infty}{\lim}~\frac{1}{n}\log FPBF^{(n)}(\mathcal M_1,\mathcal M_2)\stackrel{a.s.}{=}-\left(h_1(\tilde\theta_1)-h_2(\tilde\theta_2)\right),$$ 
	where $h_1$ and $h_2$ are given by (\ref{eq:h1_ar1}) and (\ref{eq:h2_ar1}) and $\tilde\theta_1$ and $\tilde\theta_2$ are their respective unique minimizers.
\end{theorem}

\section{Illustrations of PBF convergence in inverse regression problems}
\label{sec:illustrations_inverse}
First note that if $f(y_i|\theta,\tilde x_i,\bY^{(i-1)},\mathcal M)$ is bounded and continuous in $(\theta, \tilde x_i)$, then in inverse regression
setups, $g(y_i,\theta, \mathcal M)$ is bounded and continuous in $\theta$ if $\pi(\tilde x_i|\theta,\mathcal M)$ is bounded and continuous in
$(\theta,\tilde x_i)$. Here continuity of $g(\bY^{(i)},\theta\mathcal M)$ follows by the dominated convergence theorem. Thus,
whenever $f(y_i|\theta,x_i,\bY^{(i-1)},\mathcal M_0)$ are also bounded and continuous in $\theta$ and conditions (S1)--(S7) of
Shalizi are verified, almost sure exponential convergence of IPBF also hold, provided that $h^*(\tilde\theta)$ exists.
But existence of $h^*(\tilde\theta)$ requires existence of the limit of $n^{-1}\sum_{i=1}^ng(\bY^{(i)},,\tilde\theta,\mathcal M)$.
Although this is expected to exist, it is not straightforward to guarantee this rigorously for
general regression problems.

However, in practice, simple approximations may be used. For example, if $\mathcal M$ stands for simple linear regression, then let us consider a uniform prior 
for $\tilde x_i$ on $\mathcal X = [-a, a]$, for some $a > 0$. Then
\begin{equation*}
g(\bY^{(i)},\theta,\mathcal M)=\int_{-a}^a\frac{1}{\sigma_{\epsilon}\sqrt{2\pi}}\exp\left\{-\frac{1}{2\sigma^2_{\epsilon}}\left(y_i-\alpha-\beta\tilde x_i\right)^2\right\}
	d\tilde x_i\stackrel{a.s.}{\longrightarrow}|\beta|^{-1},~\mbox{as}~a\rightarrow\infty.
\end{equation*}
Thus for sufficiently large $a$, $g(\bY^{(i)},\tilde\theta,\mathcal M)$ can be approximated
by $|\tilde\beta|^{-1}$, which is independent of $i$.
Thus, for large enough $a$, the limit of $n^{-1}\sum_{i=1}^n\log g(\bY^{(i)},\tilde\theta,\mathcal M)$ can be approximated by $|\beta|^{-1}$.
But in general non-linear regression, such simple approximations are not available.

The setup where $\by_i = \left\{y_{i1},\ldots,y_{im}\right\}$, is far more flexible in this regard. Let us illustrate this
with respect to the models $\mathcal M_0$, $\mathcal M_1$ and $\mathcal M_2$ considered in Section \ref{sec:illustrations_forward}. Assuming invertibility of
$\eta_0$ in addition to continuity, we assume the prior
\begin{equation}
	\pi(\tilde x_i|\eta_0,\mathcal M_0)\equiv U\left(B^{(0)}_{im}(\eta_0)\right)
	\label{eq:prior0}
\end{equation}
under model $\mathcal M_0$, where
\begin{equation}
	B^{(0)}_{im}(\eta_0)=\left\{x:\eta_0(x)\in\left[\bar y_i-\frac{cs_i}{\sqrt{m}},\bar y_i+\frac{cs_i}{\sqrt{m}}\right]\right\}.
	\label{eq:B0}
\end{equation}
In the case of the linear regression model $\mathcal M_1$ , we set
\begin{equation}
	\pi(\tilde x_i|\theta,\mathcal M_1)\equiv U\left(B^{(1)}_{im}(\theta)\right)
	\label{eq:prior1}
\end{equation}
 where
\begin{equation}
	B^{(1)}_{im}(\theta)=\left[\frac{\bar y_i-\alpha}{\beta}-\frac{cs_i}{|\beta|\sqrt{m}},\frac{\bar y_i-\alpha}{\beta}+\frac{cs_i}{|\beta|\sqrt{m}}\right].
	\label{eq:B1}
\end{equation}
For the quadratic model $\mathcal M_2$, note that even if the true model is quadratic, then it is not one-to-one. Hence the general form of the prior considered in 
Section \ref{sec:prior} is not applicable here. In this case, we propose the following prior for $\tilde x_i$ under the quadratic model $\mathcal M_2$:
\begin{equation}
	\pi(\tilde x_i|\theta,\mathcal M_2)\equiv U\left(B^{(2)}_{im}(\theta)\right)
	\label{eq:prior2}
\end{equation}
 where
\begin{equation}
	B^{(2)}_{im}(\theta)=\left[\frac{\bar y_i-\alpha-\beta_2x^2_i}{\beta_1}-\frac{cs_i}{|\beta_1|\sqrt{m}},
	\frac{\bar y_i-\alpha-\beta_2x^2_i}{\beta_1}+\frac{cs_i}{|\beta_1|\sqrt{m}}\right].
	\label{eq:B2}
\end{equation}
Note that the prior depends upon $x_i$ itself, which is the truth in this case. It is unusual in Bayesian inference to make the prior depend upon the truth. 
Indeed, the true parameter is always unknown; had it been known, then one would give full prior probability to the true
parameter. In our case $x_i$ is actually known but a prior is needed for $\tilde x_i$ for the sake of cross-validation. Moreover, the prior does not consider 
$x_i$ to be known as long as the sample sizes $n$ and $m$ remain finite and $\theta$ is unknown or takes false values. The prior has substantial variance
in these cases. Hence, although unusual, such a prior on $\tilde x_i$ is not untenable for inverse cross-validation.

Now observe that $\tilde\theta_1$ and $\tilde\theta_2$ associated with models $\mathcal M_1$ and $\mathcal M_2$ remain the same as those
in Section \ref{sec:illustrations_forward}. Also note that when the true model is $\mathcal M_0$ and when $\tilde\theta_1$ is associated with $\mathcal M_1$,
then $$B^{(1)}_{im}(\tilde\theta_1)\stackrel{a.s.}{\longrightarrow}\left\{x^*_{i1}\right\},~\mbox{as}~m\rightarrow\infty,$$
where
\begin{equation}
	x^*_{i1}=\frac{\eta_0(x_i)-\tilde\alpha}{\tilde\beta}.
	\label{eq:x_star_i1}
\end{equation}
Similarly, when the true model is $\mathcal M_0$ and when $\tilde\theta_2$ is associated with $\mathcal M_2$, then
$$B^{(2)}_{im}(\tilde\theta_2)\stackrel{a.s.}{\longrightarrow}\left\{x^*_{i2}\right\},~\mbox{as}~m\rightarrow\infty,$$
where
\begin{equation}
	x^*_{i2}=\frac{\eta_0(x_i)-\tilde\alpha-\tilde\beta_2x^2_i}{\tilde\beta_1}.
	\label{eq:x_star_i2}
\end{equation}
Since $x^*_{i1}$ and $x^*_{i2}$ given by (\ref{eq:x_star_i1}) and (\ref{eq:x_star_i2})
are both continuous in $x_i$, the asymptotic calculations of $\frac{1}{n}\log\prod_{i=1}^nf(y_{ik}|\tilde\theta_1,x^*_{i1},\mathcal M_1)$
and $\frac{1}{n}\log\prod_{i=1}^nf(y_{ik}|\tilde\theta_1,x^*_{i2},\mathcal M_2)$ remain the same as $\frac{1}{n}\log\prod_{i=1}^nf(y_i|\tilde\theta_1,x_{i},\mathcal M_1)$
and $\frac{1}{n}\log\prod_{i=1}^nf(y_i|\tilde\theta_1,x_{i},\mathcal M_2)$, respectively, detailed in Section \ref{sec:illustrations_forward}. 
Hence, the final asymptotic results for IPBF remain the same for FPBF with respect to the models $\mathcal M_0$, $\mathcal M_1$ and $\mathcal M_2$.
Also note that here the cross-validation posterior for $\mathcal M_0$ is given by
\[\pi\left(y_{ik}|\bY_{nm,-i},\bX_{n,-i}\right)=\int_{\mathcal X}f(y_{ik}|\tilde x_i,\mathcal M_0)d\pi(\tilde x_i|\mathcal M_0)
\stackrel{a.s.}{\longrightarrow}f(y_{ik}|x_i),~\mbox{as}~m\rightarrow\infty,
\]
since $B^{(0)}_{im}(\eta_0)\stackrel{a.s.}{\longrightarrow}\left\{x_i\right\}$, as $m\rightarrow\infty$.
Hence, the final asymptotic results do not depend upon whether or not $x_i$ is considered known or the prior $\pi(\tilde x_i|\eta_0,\mathcal M_0)$ is used treating it as unknown,
when cross-validating for $\mathcal M_0$. Appealing to Theorem \ref{theorem:inverse_cons3}, Remark \ref{remark:inv_cv1} and Theorem \ref{theorem:inverse_cons3_mis} 
we thus summarize our results for
IPBF concerning $\mathcal M_0$, $\mathcal M_1$ and $\mathcal M_2$ as follows.

\begin{theorem}[{\bf IPBF convergence for linear regression}] 
	\label{theorem:linear_inverse}
Assume the setup where data $\left\{y_{ij}; i=1,\ldots,n; j=1,\ldots,m\right\}$ are available. In this setup consider the linear regression model
	$\mathcal M_1$ given by (\ref{eq:M_forward1}) and the true, non-linear model $\mathcal M_0$ given by (\ref{eq:M_forward2}). Let the parameter
space $\Theta$ associated with model $\mathcal M_1$ be $\mathbb R\times\mathbb R\times\mathbb R^+$, and let the covariate space $\mathcal X$ be compact.
	Assume the priors (\ref{eq:prior0}) and (\ref{eq:prior1}) on $\tilde x_i$ under the models $\mathcal M_0$ and $\mathcal M_1$, respectively. Then
	\begin{equation*}
	\underset{m\rightarrow\infty}{\lim}\underset{n\rightarrow\infty}{\lim}~\frac{1}{n}\log IPBF^{(n,m,k)}(\mathcal M_1,\mathcal M_0)\stackrel{a.s.}{=}-h_1(\tilde\theta_1),
	\end{equation*}
	where for $\theta\in\Theta$, $h_1(\theta)$ is given by (\ref{eq:h_linear_reg1}), and $\tilde\theta_1=\left(\tilde\alpha,\tilde\beta,{\tilde\sigma}^2_{\epsilon}\right)$, 
	where $\tilde\alpha$, $\tilde\beta$ and $\tilde\sigma^2_{\epsilon}$ are given by
	(\ref{eq:alpha2}), (\ref{eq:beta2}) and (\ref{eq:sigma2}), respectively. The result remains unchanged if $x_i$ is treated as known
for cross-validation with respect to $\mathcal M_0$.
\end{theorem}

\begin{theorem}[{\bf IPBF convergence for quadratic regression}] 
	\label{theorem:quadratic_inverse}
Assume the setup where data $\left\{y_{ij}; i=1,\ldots,n; j=1,\ldots,m\right\}$ are available. In this setup consider the quadratic regression model
	$\mathcal M_2$ given by (\ref{eq:M_forward3}) and the true, non-linear model $\mathcal M_0$ given by (\ref{eq:M_forward2}). Let the parameter
space $\Theta$ associated with model $\mathcal M_2$ be $\mathbb R\times\mathbb R\times\mathbb R\times\mathbb R^+$, and let the covariate space $\mathcal X$ be compact.
Also assume that $x_i;i\geq 1$ are all distinct.
	Assume the priors (\ref{eq:prior0}) and (\ref{eq:prior2}) on $\tilde x_i$ under the models $\mathcal M_0$ and $\mathcal M_2$, respectively. Then
	\begin{equation*}
	\underset{m\rightarrow\infty}{\lim}\underset{n\rightarrow\infty}{\lim}~\frac{1}{n}\log IPBF^{(n,m,k)}(\mathcal M_2,\mathcal M_0)\stackrel{a.s.}{=}-h_2(\tilde\theta_2),
	\end{equation*}
	where for $\theta\in\Theta$, $h_2(\theta)$ is given by (\ref{eq:h_quad_reg1}), and 
	$\tilde\theta_2=\left(\tilde\alpha,\tilde\beta_1,\tilde\beta_2,{\tilde\sigma}^2_{\epsilon}\right)$, 
	where $\tilde\alpha$, $\tilde\beta_1$, $\tilde\beta_2$ and $\tilde\sigma^2_{\epsilon}$ are given by
	(\ref{eq:ls_quad_true1}), (\ref{eq:A_b}) and (\ref{eq:sigma_quad}). The result remains unchanged if $x_i$ is treated as known
for cross-validation with respect to $\mathcal M_0$.
\end{theorem}

\begin{theorem}[{\bf Comparison between linear and quadratic regressions}] 
	\label{theorem:linear_vs_quadratic_inverse}
Assume the setup where data $\left\{y_{ij}; i=1,\ldots,n; j=1,\ldots,m\right\}$ are available. 
	Let the true model be given by $\mathcal M_0$ formulated in (\ref{eq:M_forward2}). Assuming that the covariate observations $x_i;i\geq 1$ are all distinct
	and that the covariate space $\mathcal X$ is compact, consider comparison of the linear and quadratic
	regression models $\mathcal M_1$ and $\mathcal M_2$ given by (\ref{eq:M_forward1}) and (\ref{eq:M_forward3}), respectively, using IPBF. Assume the
	priors (\ref{eq:prior0}), (\ref{eq:prior1}) and (\ref{eq:prior2}) on $\tilde x_i$ under the models $\mathcal M_0$, $\mathcal M_1$ and $\mathcal M_2$, respectively. Then,
	\begin{equation*}
		\underset{m\rightarrow\infty}{\lim}\underset{n\rightarrow\infty}{\lim}~\frac{1}{n}\log IPBF^{(n,m,k)}(\mathcal M_1,\mathcal M_2)\stackrel{a.s.}{=}
		-\left(h_1(\tilde\theta_1)-h_2(\tilde\theta_2)\right)
	\end{equation*}
	where $h_1(\tilde\theta_1)$ and $h_2(\tilde\theta_2)$ are the same as in Theorems \ref{theorem:linear_inverse} and \ref{theorem:quadratic_inverse}, respectively. 
	The result remains unchanged if $\tilde x_i$ is treated as known for cross-validation with respect to $\mathcal M_0$.
\end{theorem}

\subsection{IPBF asymptotics for variable selection in AR(1)}
\label{subsec:ar1_inverse}
Now let us reconsider the AR(1) regression setup described by the competing models $\mathcal M_1$ (\ref{eq:ar1_x}), $\mathcal M_2$ (\ref{eq:ar1_z})
and the true model $\mathcal M_0$ (\ref{eq:ar1_both}), along with assumptions (A1) and (A2).
But now we reformulate the models as follows to suit the second setup of inverse regression.
\begin{equation}
	y_{tj}=\rho_1 y_{t-1,j}+\beta_1 x_t+\epsilon_{1t,j};~t=1,\ldots,n;~j=1,\ldots,m,
	\label{eq:ar1_x_inv}
\end{equation}
where $y_{0j}\equiv 0$ for $j=1,\ldots,m$ and 
$\epsilon_{1t,j}\stackrel{iid}{\sim}N\left(0,\sigma^2_{1}\right)$, for $t=1,\ldots,n$ and $j=1,\ldots,m$. 
%Here $\theta_1=\left(\rho_1,\beta_1,\sigma^2_{1}\right)$ is the 
%set of unknown parameters and 
%$\Theta_1=\mathbb R\times\mathbb R\times\mathbb R^+$ is the parameter space. We might wish to compare this model with the following AR(1) quadratic regression model
%which we refer to as $\mathcal M_2$:
Similarly, $\mathcal M_2$ is now given by
\begin{equation}
	y_{tj}=\rho_2 y_{t-1,j}+\beta_2 z_t+\epsilon_{2t,j};~t=1,\ldots,n;~j=1,\ldots,m,
	\label{eq:ar1_z_inv}
\end{equation}
where %$y_{0j}\equiv 0$ for $j=1,\ldots,m$ and 
$\epsilon_{2t,j}\stackrel{iid}{\sim}N\left(0,\sigma^2_{2}\right)$, for $t=1,\ldots,n$ and $j=1,\ldots,m$.  
%$y_0\equiv 0$ $z_t;t=1,\ldots,n$ are covariate observations and $\theta_2=\left(\rho_2,\beta_2,\sigma^2_{2}\right)$ is the set of parameters and the parameter space 
%$\Theta_2=\mathbb R\times\mathbb R\times\mathbb R^+$ remains the same as $\Theta_1$. Here, for $t=1,\ldots,n$, 
%$\epsilon_{2t}\stackrel{iid}{\sim}N\left(0,\sigma^2_{2}\right)$ 

The true model $\mathcal M_0$ be given by
\begin{equation}
	y_{tj}=\rho_0 y_{t-1,j}+\beta_0 (x_t+z_t)+\epsilon_{0t,j};~t=1,\ldots,n;~j=1,\ldots,m,
	\label{eq:ar1_both_inv}
\end{equation}
where $|\rho_0|<1$ and $\epsilon_{0t,j}\stackrel{iid}{\sim}N\left(0,\sigma^2_{0}\right)$, for $t=1,\ldots,n$ and $j=1,\ldots,m$.

For $t=1,\ldots,n$, let $\bar y_t=\frac{\sum_{j=1}^my_{tj}}{m}$ and $s^2_t(\rho)=\frac{1}{m}\left[(y_{tj}-\bar y_t)-\rho (y_{t-1,j}-\bar y_{t-1})\right]^2$. 
We consider the following priors for $\tilde x_t$ and $\tilde z_t$ 
associated with $\mathcal M_1$ and $\mathcal M_2$: 
\begin{align}
	&\pi(\tilde x_t|\theta_1,\mathcal M_1)\equiv U\left(B^{(1)}_{tm}(\theta_1)\right);
	\label{eq:ar1_prior1}\\
	&\pi(\tilde z_t|\theta_2,\mathcal M_2)\equiv U\left(B^{(2)}_{tm}(\theta_1)\right),
	\label{eq:ar1_prior2}
\end{align}
where
\begin{align}
	&B^{(1)}_{tm}(\theta_1)=\left[\frac{\bar y_t-\rho_1\bar y_{t-1}}{\beta_1}-\frac{cs_t(\rho_1)}{|\beta_1|\sqrt{m}},
	\frac{\bar y_t-\rho_1\bar y_{t-1}}{\beta_1}+\frac{cs_t(\rho_1)}{|\beta_1|\sqrt{m}}\right];
	\label{eq:B1_ar1_inv}\\
	&B^{(2)}_{tm}(\theta_2)=\left[\frac{\bar y_t-\rho_2\bar y_{t-1}}{\beta_2}-\frac{cs_t(\rho_2)}{|\beta_2|\sqrt{m}},
	\frac{\bar y_t-\rho_2\bar y_{t-1}}{\beta_2}+\frac{cs_t(\rho_2)}{|\beta_2|\sqrt{m}}\right].
	\label{eq:B2_ar1_inv}
\end{align}

Note that
\begin{align}
	&B^{(1)}_{tm}(\tilde\theta_1)\stackrel{a.s.}{\longrightarrow}\{x^*_t\}~\mbox{as}~m\rightarrow\infty;\label{eq:xt_star}\\
	&B^{(2)}_{tm}(\tilde\theta_2)\stackrel{a.s.}{\longrightarrow}\{z^*_t\}~\mbox{as}~m\rightarrow\infty,\label{eq:zt_star}
\end{align}
where
\begin{align}
	&x^*_t=\frac{\beta_0\sum_{k=1}^t\rho^{t-k}x_k-\beta_0\tilde\rho_1\sum_{k=1}^{t-1}\rho^{t-k}_0x_k}{\tilde\beta_1};\label{eq:xt_star2}\\
	&z^*_t=\frac{\beta_0\sum_{k=1}^t\rho^{t-k}z_k-\beta_0\tilde\rho_2\sum_{k=1}^{t-1}\rho^{t-k}_0z_k}{\tilde\beta_2}.\label{eq:zt_star2}
\end{align}
Direct calculations reveal that
\begin{align}
	&\frac{1}{n}\sum_{t=1}^nx^*_t\rightarrow 0;~\frac{1}{n}\sum_{t=1}^n{x^*_t}^2\rightarrow \sigma^2_{x^*}
	=\sigma^2_x\frac{\beta^2_0(1-\tilde\rho_1)^2}{\tilde\beta^2_1(1-\rho^2_0)},~\mbox{as}~n\rightarrow\infty;\label{eq:moments_x}\\
	&\frac{1}{n}\sum_{t=1}^nz^*_t\rightarrow 0;~\frac{1}{n}\sum_{t=1}^n{z^*_t}^2\rightarrow \sigma^2_{z^*}
	=\sigma^2_z\frac{\beta^2_0(1-\tilde\rho_2)^2}{\tilde\beta^2_2(1-\rho^2_0)},~\mbox{as}~n\rightarrow\infty.\label{eq:moments_z}
\end{align}
Hence, for the final IPBF calculations associated with $h_1$ and $h_2$ for this example, we need to replace $x_t$, $z_t$, $\sigma^2_x$ and $\sigma^2_z$ in (A1)
with $x^*_t$, $z^*_t$, $\sigma^2_{x^*}$ and $\sigma^2_{z^*}$, respectively, for models $\mathcal M_1$ and $\mathcal M_2$.
In this regard, let
	\begin{multline}
	h^*_1(\theta)=\log\left(\frac{\sigma}{\sigma_0}\right)+\left(\frac{1}{2\sigma^2}-\frac{1}{2\sigma^2_0}\right)\left(\frac{\sigma^2_0}{1-\rho^2_0}
		+\frac{\beta^2_0\sigma^2_{x+z}}{1-\rho^2_0}\right)
		+\left(\frac{\rho^2}{2\sigma^2}-\frac{\rho^2_0}{2\sigma^2_0}\right)\left(\frac{\sigma^2_0}{1-\rho^2_0}+\frac{\beta^2_0\sigma^2_{x+z}}{1-\rho^2_0}\right)\\
		+\frac{1}{2\sigma^2}\beta^2\sigma^2_{x+z}-
		\frac{1}{2\sigma^2_0}\beta^2_0\sigma^2_{x+z}
		-\left(\frac{\rho}{\sigma^2}-\frac{\rho_0}{\sigma^2_0}\right)\left(\frac{\rho_0\sigma^2_0}{1-\rho^2_0}+\frac{\rho_0\beta^2_0\sigma^2_{x+z}}{1-\rho^2_0}\right)
		-\left(\frac{\beta}{\sigma^2}-\frac{\beta_0}{\sigma^2_0}\right)\sigma^2_{x+z}\beta_0\\
		+\frac{\sigma^2_z\beta(\beta_0-\beta)}{\sigma^2}+\frac{\beta^2}{2\sigma^2}\left(\sigma^2_{x+z}+\sigma^2_{x^*}-\frac{2\beta_0\sigma^2_x}{\tilde\beta_1}\right).
		\label{eq:h1_star_ar1}
	\end{multline}
and
	\begin{multline}
	h^*_2(\theta)=\log\left(\frac{\sigma}{\sigma_0}\right)+\left(\frac{1}{2\sigma^2}-\frac{1}{2\sigma^2_0}\right)\left(\frac{\sigma^2_0}{1-\rho^2_0}
		+\frac{\beta^2_0\sigma^2_{x+z}}{1-\rho^2_0}\right)
		+\left(\frac{\rho^2}{2\sigma^2}-\frac{\rho^2_0}{2\sigma^2_0}\right)\left(\frac{\sigma^2_0}{1-\rho^2_0}+\frac{\beta^2_0\sigma^2_{x+z}}{1-\rho^2_0}\right)\\
		+\frac{1}{2\sigma^2}\beta^2\sigma^2_{x+z}-
		\frac{1}{2\sigma^2_0}\beta^2_0\sigma^2_{x+z}
		-\left(\frac{\rho}{\sigma^2}-\frac{\rho_0}{\sigma^2_0}\right)\left(\frac{\rho_0\sigma^2_0}{1-\rho^2_0}+\frac{\rho_0\beta^2_0\sigma^2_{x+z}}{1-\rho^2_0}\right)
		-\left(\frac{\beta}{\sigma^2}-\frac{\beta_0}{\sigma^2_0}\right)\sigma^2_{x+z}\beta_0\\
		+\frac{\sigma^2_x\beta(\beta_0-\beta)}{\sigma^2}+\frac{\beta^2}{2\sigma^2}\left(\sigma^2_{x+z}+\sigma^2_{z^*}-\frac{2\beta_0\sigma^2_z}{\tilde\beta_2}\right).
		\label{eq:h2_star_ar1}
	\end{multline}

If cross-validation is considered with respect to the true model $\mathcal M_0$ with a prior on the covariates, then since $x_t$ and $z_t$ are not separately 
identifiable in $\mathcal M_0$, let $u_t=x_t+z_t$ and consider a prior on $\tilde u_t$ as follows:
\begin{equation}
	\pi(\tilde u_t|\theta_0,\mathcal M_0)\equiv U\left(B^{(0)}_{tm}(\theta_0)\right),
	\label{eq:ar1_prior0}
\end{equation}
where
\begin{equation}
	B^{(0)}_{tm}(\theta_0)=\left[\frac{\bar y_t-\rho_0\bar y_{t-1}}{\beta_0}-\frac{cs_t(\rho_0)}{|\beta_0|\sqrt{m}},
	\frac{\bar y_t-\rho_1\bar y_{t-1}}{\beta_0}+\frac{cs_t(\rho_0)}{|\beta_0|\sqrt{m}}\right].
	\label{eq:B0_ar1_inv}
\end{equation}
Note that $B^{(0)}_{tm}(\theta_0)\stackrel{a.s.}{\longrightarrow}\{u_t\}$, as $m\rightarrow\infty$. Let $\bU_{n,-t}=\{u_1,\ldots,u_n\}\backslash\{u_t\}$. 
As before, it follows that $\pi\left(y_{tk}|\bY_{nm,-t},\bU_{n,-t}\right)\stackrel{a.s.}{\longrightarrow}f(y_{tk}|u_t,y_{t-1,k})$, as $m\rightarrow\infty$.
Hence, the final asymptotic results do not depend upon whether or not $u_t$ is considered known or the prior (\ref{eq:ar1_prior0}) is used for $\tilde u_t$ 
treating $u_t$ it as unknown, when cross-validating for the true model $\mathcal M_0$.

We summarize our results on variable selection in the inverse AR(1) regression framework as follows.
\begin{theorem}[{\bf IPBF consistency for $\mathcal M_1$ versus $\mathcal M_0$}] 
	\label{theorem:ar1_1}
	Consider comparing model $\mathcal M_1$ (\ref{eq:ar1_x_inv}) against the true model $\mathcal M_0$ (\ref{eq:ar1_both_inv}).
	Assume the priors (\ref{eq:ar1_prior1}) and (\ref{eq:ar1_prior0}) on $\tilde x_t$ and $\tilde u_t$ 
	under the models $\mathcal M_1$ and $\mathcal M_0$, respectively. Then
	\begin{equation*}
	\underset{m\rightarrow\infty}{\lim}\underset{n\rightarrow\infty}{\lim}~\frac{1}{n}\log IPBF^{(n,m,k)}(\mathcal M_1,\mathcal M_0)\stackrel{a.s.}{=}-h^*_1(\tilde\theta_1),
	\end{equation*}
	where for $\theta\in\Theta_1$, $h^*_1(\theta)$ is given by (\ref{eq:h1_star_ar1}), and 
	$\tilde\theta_1$ is the unique minimizer of $h_1$ given by (\ref{eq:h1_ar1}). 
	The result remains unchanged if $u_t$ is treated as known for cross-validation with respect to $\mathcal M_0$.
\end{theorem}

\begin{theorem}[{\bf IPBF consistency for $\mathcal M_2$ versus $\mathcal M_0$}] 
	\label{theorem:ar1_2}
	Consider comparing model $\mathcal M_2$ (\ref{eq:ar1_z_inv}) against the true model $\mathcal M_0$ (\ref{eq:ar1_both_inv}).
	Assume the priors (\ref{eq:ar1_prior2}) and (\ref{eq:ar1_prior0}) on $\tilde z_t$ and $\tilde u_t$ 
	under the models $\mathcal M_2$ and $\mathcal M_0$, respectively. Then
	\begin{equation*}
	\underset{m\rightarrow\infty}{\lim}\underset{n\rightarrow\infty}{\lim}~\frac{1}{n}\log IPBF^{(n,m,k)}(\mathcal M_1,\mathcal M_0)\stackrel{a.s.}{=}-h^*_2(\tilde\theta_2),
	\end{equation*}
	where for $\theta\in\Theta_2$, $h^*_2(\theta)$ is given by (\ref{eq:h2_star_ar1}), and 
	$\tilde\theta_2$ is the unique minimizer of $h_2$ given by (\ref{eq:h2_ar1}). 
	The result remains unchanged if $u_t$ is treated as known for cross-validation with respect to $\mathcal M_0$.
\end{theorem}

\begin{theorem}[{\bf IPBF convergence for $\mathcal M_1$ versus $\mathcal M_2$}] 
	\label{theorem:ar1_3}
	Consider comparing models $\mathcal M_1$ (\ref{eq:ar1_x_inv}) against model $\mathcal M_2$ (\ref{eq:ar1_z_inv}).
	Assume the priors (\ref{eq:ar1_prior1}) and (\ref{eq:ar1_prior2}) on $\tilde x_t$ and $\tilde z_t$ 
	under the models $\mathcal M_1$ and $\mathcal M_2$, respectively. Then
	\begin{equation*}
		\underset{m\rightarrow\infty}{\lim}\underset{n\rightarrow\infty}{\lim}~\frac{1}{n}\log IPBF^{(n,m,k)}(\mathcal M_1,\mathcal M_2)
		\stackrel{a.s.}{=}-\left(h^*_1(\tilde\theta_1)-h^*_2(\tilde\theta_2)\right),
	\end{equation*}
	where $h^*_1$ and $h^*_2$ are given by (\ref{eq:h1_star_ar1}) and (\ref{eq:h2_star_ar1}). In the above, 
	$\tilde\theta_1$ and $\tilde\theta_2$ are the unique minimizers of $h_1$ of $h_2$ given by (\ref{eq:h1_ar1}) and (\ref{eq:h2_ar1}), respectively. 
	The result remains unchanged if $u_t$ is treated as known for cross-validation with respect to $\mathcal M_0$.
\end{theorem}

\subsection{Discussion of FPBF and IPBF convergence for nonparametric regression models}
\label{subsec:nonpara1}
\ctn{Chatterjee1} investigate posterior convergence for Gaussian and general stochastic process regression under suitable assumptions while
posterior convergence for binary and Poisson nonparametric regression based on Gaussian process modeling of the regression function are addressed in \ctn{Chatterjee2}. 
In all these nonparametric setups, the authors verified assumptions (S1)--(S7) of Shalizi. Here it is important to point out that Theorem \ref{theorem:shalizi_Rn} used to verify
assumption (S6) of Shalizi in our parametric setups, is not valid in infinite-dimensional nonparametric models since without further assumptions on model
sparsity, $\tilde\theta^*_n$ can not converge to $\tilde\theta$. That is, condition (ii) of Theorem \ref{theorem:shalizi_Rn} does not hold in general for nonparametric models. 
Moreover, enforcing sparsity conditions to general stochastic processes, such as Gaussian processes, need not be desirable. \ctn{Chatterjee1} and \ctn{Chatterjee2}
propose a general sufficient condition for verification of (S6) of Shalizi, which is appropriate for nonparametric models, and use that condition for their purposes.

The point of the above discussion is that assumptions (S1)--(S7) are already verified by \ctn{Chatterjee1} and \ctn{Chatterjee2} 
for nonparametric Bayesian regression models, and since boundedness and continuity of $f(y_i|\theta,\mathcal M)$ also hold for such models $\mathcal M$,
our asymptotic results on almost sure exponential convergence of FPBF and IPBF are directly applicable to such models. 
For IPBF convergence in nonparametric situations, the priors for $\tilde x_i$ proposed in Section \ref{subsubsec:illustrations_prior} for nonparametric cases (ii)--(iv)
are appropriate.

Note that parametric and nonparametric models
can also be compared asymptotically using our FPBF and IPBF theory.

\section{Simulation experiments}
\label{sec:simstudy}
So far we have investigated large sample properties of FPBF and IPBF. However, for all practical purposes it is important to provide insights into small sample
behaviours of such versions of pseudo-Bayes factor. In this section we undertake such small sample study with the help of simulation experiments.
Specifically, we set $n=m=10$ and generate data from relevant Poisson distribution with the log-linear link function  and consider modeling the data with
Poisson and geometric distributions with log, logit and probit links for linear models as well as nonparametric regression modeled by Gaussian process having linear mean function
and squared exponential covariance. We also consider variable selection in these setups with respect to two different covariates. We report both FPBF and IPBF results
for the experiments. Details follow.

\subsection{Poisson versus geometric linear and nonparametric regresison models when the true model is Poisson linear regression}
\label{subsec:poisson_vs_geometric}

\subsubsection{True distribution}
Let us first consider the case 
%where $y_{ij}\sim Poisson(\lambda(x_i))$, where $\lambda(x)=H(\eta(x))$, as briefed in Section \ref{subsec:illustrations_prior} (ii).
%In particular, we let $H(\cdot)=\exp(\cdot)$ and $\eta(\cdot)$ be a Gaussian process with mean function $\mu(x)=\alpha+\beta x$ and covariance 
%$Cov\left(\eta(x_1),\eta(x_2)\right)=\sigma^2\exp\left\{-(x_1-x_2)^2\right\}$, where $\sigma$ is unknown. We assume that 
where the true data-generating distribution
is $y_{ij}\sim Poisson(\lambda(x_i))$, with $\lambda(x)=\exp(\alpha_0+\beta_0 x)$. We generate the data by simulating $\alpha_0\sim U(-1,1)$, $\beta_0\sim U(-1,1)$
and $x_i\sim U(-1,1)$; $i=1,\ldots,n$, and then finally simulating $y_{ij}\sim Poisson(\lambda(x_i))$; $j=1,\ldots,m$, $i=1,\ldots,n$.

To model the data generated from the true distribution, we consider both Poisson and geometric distributions and both linear and Gaussian process based nonparametric
regression for such models. Let us begin with the Poisson setup. 

\subsection{Competing forward and inverse Poisson regression models}
\label{subsec:poisson_models}

\subsubsection{Forward Poisson linear regression model}
\label{subsubsec:forward_poisson_linear}
In this setup we model the data as follows: $y_{ij}\sim Poisson(\lambda(x_i))$, with $\lambda(x)=\exp(\alpha+\beta x)$, and set the prior
$\pi\left(\alpha,\beta\right)=1$, for $-\infty<\alpha,\beta<\infty$. 
For the forward setup, this completes the model and prior specifications. Denoting this by model $\mathcal M$, we compute the forward cross-validation posterior
of the form %(\ref{eq:post_forward1}) 
\begin{equation}
	\pi(y_{i1}|\bY_{n,-i},\bX_n,\mathcal M)=\int_{\Theta} f(y_{i1}|\theta,x_i,\bY^{(i-1)}_1,\mathcal M)d\pi(\theta|\bY_{n,-i},\bX_{n,-i},\mathcal M),
\label{eq:post_forward_mult}
\end{equation}
by taking Monte Carlo averages of $f(y_{i1}|\theta,x_i,\bY^{(i-1)},\mathcal M)$ over realizations of $\theta$ from $\pi(\theta|\bY_{n,-i},\bX_{n,-i},\mathcal M)$. 
In our case this is the Monte Carlo average of 
the relevant Poisson probability of $y_{i1}$ given $x_i$ over realizations of $\theta=(\alpha,\beta)$. Samples of $\theta$ are obtained approximately 
from the posterior distribution
of $\pi(\theta|\bY_{nm,-i},\bX_{n,-i})$ by first generating realizations from the ``importance sampling density" 
$\pi(\theta|\bY_{nm},\bX_{n})$ using transformation based Markov chain Monte Carlo (TMCMC)
(\ctn{Dutta13}) and then re-using the realizations with importance weights to obtain the desired Monte Carlo averages. 
The rationale behind the choice of the full posterior $\pi(\theta|\bY_{nm},\bX_{n})$ associated with the full data set as the importance sampling density is that it is
not significantly different from the posterior $\pi(\theta|\bY_{nm,-i},\bX_{n,-i})$ associated with leaving out a single data point.
%is unlikely to significantly affect the posterior of $\theta$ given the full data set. 
This choice is also quite popular in the literature; see, for example,
\ctn{Gelfand96}. In our examples, we generate $30,000$ TMCMC samples from $\pi(\theta|\bY_{nm},\bX_{n})$
of which we discard the first $10,000$ as burn-in, and re-sample $1000$ $\theta$-realizations without replacement from the remaining $20,000$ realizations. 
We re-use each re-sampled $\theta$-value $100$ times and compute the Monte Carlo average over such $1000\times 100=100,000$ realizations. 
The re-use of each re-sampled $\theta$-value corresponds to importance re-sampling MCMC (IRMCMC) of \ctn{Bhatta07}. Although IRMCMC is meant for cross-validation
in inverse problems, the idea carries over to forward problems as well. We finally compute $\frac{1}{n}\sum_{i=1}^n\log\pi(y_{i1}|\bY_{nm,-i},\bX_n,\mathcal M)$
for model $\mathcal M$.

\subsubsection{Inverse Poisson linear regression model}
\label{subsubsec:inverse_poisson_linear}

With the same Poisson linear regression model as in the forward case, we now put a prior on $\tilde x_i$ corresponding to $x_i$. 
In our case, it follows from Section \ref{subsubsec:illustrations_prior} that $\pi(\tilde x_i|\alpha,\beta)\equiv U(a,b)$, where 
\begin{equation}
a=\min\left\{\beta^{-1}\left(\log\left(\bar y_i-\frac{c_1s_i}{\sqrt{m}}\right)-\alpha\right),
\beta^{-1}\left(\log\left(\bar y_i+\frac{c_2s_i}{\sqrt{m}}\right)-\alpha\right)\right\}
	\label{eq:a1}
\end{equation}
and 
\begin{equation}
b=\max\left\{\beta^{-1}\left(\log\left(\bar y_i-\frac{c_1s_i}{\sqrt{m}}\right)-\alpha\right),
\beta^{-1}\left(\log\left(\bar y_i+\frac{c_2s_i}{\sqrt{m}}\right)-\alpha\right)\right\}.
	\label{eq:b1}
\end{equation}
We set $c_1=1$ and $c_2=100$, for ensuring positive value of $\bar y_i-\frac{c_1s_i}{\sqrt{m}}$ (so that logarithm of this quantity is well-defined) 
and a reasonably large support of the prior for $\tilde x_i$. We then compute 
$$
\pi(y_{i1}|\bY_{nm,-i},\bX_{n,-i},\mathcal M)=\int_{\mathcal X}\int_{\Theta} f(y_{i1}|\theta,\tilde x_i,\bY^{(i-1)}_1,\mathcal M)
	d\pi(\tilde x_i,\theta|\bY_{nm,-i},\bX_{n,-i},\mathcal M)
$$
by Monte Carlo averaging of the relevant Poisson probability of $y_{i1}$ over realizations of $(\tilde x_i,\theta)=(\tilde x_i,\alpha,\beta)$
generated from $\pi(\tilde x_i,\theta|\bY_{nm,-i},\bX_{n,-i},\mathcal M)$. Since it follows from (\ref{eq:inv1}) that
$\pi(\tilde x_i,\theta|\bY_{nm,-i},\bX_{n,-i},\mathcal M)$ $=$ $\pi(\tilde x_i|\theta,\mathcal M)\pi(\theta|\bY_{nm,-i},\bX_{n,-i},\mathcal M)$, and 
since realizations of $\theta$ from $\pi(\theta|\bY_{nm,-i},\bX_{n,-i},\mathcal M)$ are already available in the forward context, we simply generate
$\tilde x_i$ given $\theta$ from the prior for $\tilde x_i$ to obtain realizations from $\pi(\tilde x_i,\theta|\bY_{nm,-i},\bX_{n,-i},\mathcal M)$.
Note that for different $i$, only sub-samples of $\theta$ of size $1000$ from the original sample of size $20,000$ from the full posterior of $\theta$ are available,
and each $\theta$ is repeated $100$ times. However, realizations of $\tilde x_i$ are all distinct in spite of repetitions of $\theta$-values.

Once for each $i=1,\ldots,n$, the Monte Carlo estimates of $\pi(y_{i1}|\bY_{nm,-i},\bX_{n,-i},\mathcal M)$ are available, we finally obtain the estimate of
$\frac{1}{n}\sum_{i=1}^n\log\pi(y_{i1}|\bY_{nm,-i},\bX_{n,-i},\mathcal M)$ using the individual Monte Carlo estimates.

\subsubsection{Forward Poisson nonparametric regression model}
\label{subsubsec:forward_poisson_nonpara}

We now consider the case where $y_{ij}\sim Poisson(\lambda(x_i))$, where $\lambda(x)=\exp(\eta(x))$, where
$\eta(\cdot)$ is a Gaussian process with mean function $\mu(x)=\alpha+\beta x$ and covariance 
$Cov\left(\eta(x_1),\eta(x_2)\right)=\sigma^2\exp\left\{-(x_1-x_2)^2\right\}$, where $\sigma$ is unknown. 
%We assume that the true data-generating distribution
%is $y_{ij}\sim Poisson(\lambda(x_i))$, with $\lambda(x)=\exp(\alpha_0+\beta_0(x))$. We generate the data by simulating $\alpha_0\sim Uniform(-1,1)$, $\beta_0\sim Uniform(-1,1)$
%and $x_i\sim Uniform(-1,1)$; $i=1,\ldots,n$, and then finally simulating $y_{ij}\sim Poisson(\lambda(x_i))$; $j=1,\ldots,m$, $i=1,\ldots,n$.
%
For our convenience, we reparameterize $\sigma^2$ as $\exp(\omega)$, where $-\infty<\omega<\infty$. 
For the prior on the parameters, we set $\pi\left(\alpha,\beta,\omega\right)=1$, for $-\infty<\alpha,\beta,\omega<\infty$. 

In the inverse case, for the reason of prior specification, we linearize $\eta(\tilde x_i)$ as $\alpha+\beta\tilde x_i$; see Section \ref{subsubsec:inverse_poisson_nonpara}.
Hence, for comparability with the inverse counterpart, we set $\eta(x_i)=\alpha+\beta x_i$. Thus, in the forward case, 
$\theta=(\alpha,\beta,\eta(x_1),\ldots,\eta(x_{i-1}),\eta(x_{i+1}),\ldots,\eta(x_n),\omega)$. 
We obtain $\frac{1}{n}\sum_{i=1}^n\log\pi(y_{i1}|\bY_{nm,-i},\bX_n,\mathcal M)$ 
using the same method of Monte Carlo averaging described in Section \ref{subsubsec:forward_poisson_linear}, where $\theta$ is again first generated using TMCMC
from the full posterior of $\theta$ by discarding the first $10,000$ iterations and retaining the next $20,000$ for inference, which are re-used to approximate
the desired posteriors $\pi(\theta|\bY_{nm,-i},\bX_{n,-i},\mathcal M)$. As before, we obtain Monte Carlo averages over $100,000$ realizations of $\theta$.

\subsubsection{Inverse Poisson nonparametric regression model}
\label{subsubsec:inverse_poisson_nonpara}
The model in this case remains the same as that in Section \ref{subsubsec:forward_poisson_nonpara}, but now a prior on $\tilde x_i$ is needed.
However, note that the prior for $\tilde x_i$, which is uniform on 
$B_{im}(\eta)=\left\{x:\eta(x)\in \log\left\{\left[\bar y_i-\frac{c_1s_i}{\sqrt{m}},\bar y_i+\frac{c_2s_i}{\sqrt{m}}\right]\right\}\right\}$, does not have a closed form,
since the form of $\eta(x)$ is unknown. However, if $m$ is large, the interval 
$\log\left\{\left[\bar y_i-\frac{c_1s_i}{\sqrt{m}},\bar y_i+\frac{c_2s_i}{\sqrt{m}}\right]\right\}$ is small, and $\eta(x)$ falling in this small interval
can be reasonably well-approximated by a straight line. Hence, we set $\eta(x)=\mu(x)=\alpha+\beta x$, for $\eta(x)$ falling in this interval.
Thus it follows that $\pi(\tilde x_i|\eta)\equiv U(a,b)$, where $a$ and $b$ are given by (\ref{eq:a1}) and (\ref{eq:b1}), respectively.
Hence, we obtain the same prior for $\tilde x_i$ as in the case of linear Poisson regression described in Section \ref{subsubsec:inverse_poisson_linear}.
As before we set $c_1=1$ and $c_2=100$.

The method for obtaining $\frac{1}{n}\sum_{i=1}^n\log\pi(y_{i1}|\bY_{nm,-i},\bX_{n,-i},\mathcal M)$ remains the same as discussed in Section 
\ref{subsubsec:inverse_poisson_linear}.

\subsection{Competing forward and inverse geometric regression models}
\label{subsec:geometric_models}
We also report results of our simulation experiments where data generated from Poisson linear regression is modeled by geometric regression models of the form
\begin{equation}
	f(y_{ij}|\theta,x_i)=(1-p(x_i))^{y_{ij}}p(x_i),
	\label{eq:geo1}
\end{equation}
where $p(x_i)$ is modeled as logit or probit linear or nonparametric regression. In other words, we consider the following possibilities of modeling $p(x)$:
\begin{align}
	&\log\left(\frac{p(x)}{1-p(x)}\right)=\alpha+\beta x;~\log\left(\frac{p(x)}{1-p(x)}\right)=\eta(x);\notag\\
	&p(x)=\Phi\left(\alpha+\beta x\right);~p(x)=\Phi\left(\eta(x)\right),\notag
\end{align}
where $\Phi$ is the cumulative distribution function of the standard normal distribution. In the above, $\eta$ is again modeled by a Gaussian process with
mean function $\mu(x)=\alpha+\beta x$ and covariance function given by $Cov(\eta(x_1),\eta(x_2))=\sigma^2\exp\left\{-(x_1-x_2)^2\right\}$.
We again set $\sigma^2=\exp(\omega)$, where $-\infty<\omega<\infty$, and consider the prior $\pi(\alpha,\beta,\omega)=1$ for $-\infty<\alpha,\beta,\omega<\infty$.

In the inverse setup we assign prior on $\tilde x_i$ such that the mean of the geometric distribution, namely, $\frac{1-p(x)}{p(x)}$, lies in 
$\left[\bar y_i-\frac{c_1s_i}{\sqrt{m}},\bar y_i+\frac{c_2s_i}{\sqrt{m}}\right]$. Using the same principles as before it follows that for the logit link, either
for linear or Gaussian process regression, the prior for $\tilde x_i$ is $U(a_1,b_1)$, where  
\begin{equation}
a_1=\min\left\{-\beta^{-1}\left(\log\left(\bar y_i-\frac{c_1s_i}{\sqrt{m}}\right)+\alpha\right),
-\beta^{-1}\left(\log\left(\bar y_i+\frac{c_2s_i}{\sqrt{m}}\right)+\alpha\right)\right\}
	\label{eq:a2}
\end{equation}
and 
\begin{equation}
b_1=\max\left\{-\beta^{-1}\left(\log\left(\bar y_i-\frac{c_1s_i}{\sqrt{m}}\right)+\alpha\right),
-\beta^{-1}\left(\log\left(\bar y_i+\frac{c_2s_i}{\sqrt{m}}\right)+\alpha\right)\right\}.
	\label{eq:b2}
\end{equation}
We set $c_1=1$ and $c_2=100$, as before.

For geometric probit regression, first let $\ell_{im}=\bar y_i-\frac{c_1s_i}{\sqrt{m}}$ and $u_{im}=\bar y_i+\frac{c_2s_i}{\sqrt{m}}$. Let
\begin{align}
	a_2&=\min\left\{\frac{\Phi^{-1}\left(\frac{1}{u_{im}+1}\right)-\alpha}{\beta},\frac{\Phi^{-1}\left(\frac{1}{\ell_{im}+1}\right)-\alpha}{\beta}\right\};\label{eq:a3}\\
	b_2&=\max\left\{\frac{\Phi^{-1}\left(\frac{1}{u_{im}+1}\right)-\alpha}{\beta},\frac{\Phi^{-1}\left(\frac{1}{\ell_{im}+1}\right)-\alpha}{\beta}\right\};\label{eq:b3}
\end{align}
Then the prior for $\tilde x_i$ is $U(a_2,b_2)$, for both linear and Gaussian process based geometric probit regression.

The rest of the methodology for computing FPBF and IPBF for geometric regression remains the same as for Poisson regression described in Section \ref{subsec:poisson_models}.

\subsubsection{Results of the simulation experiment for model selection}
\label{subsubsec:results_model}

For $n=m=10$, when the true model is Poisson with log-linear regression, the last two columns of
Table \ref{table:bayesian} provide the forward and inverse estimates of $\frac{1}{n}\sum_{i=1}^n\log\pi(y_{i1}|\bY_{nm,-i},\bX_n,\mathcal M)$ and 
$\frac{1}{n}\sum_{i=1}^n\log\pi(y_{i1}|\bY_{nm,-i},\bX_{n,-i},\mathcal M)$, respectively, 
for Poisson and geometric linear and Gaussian process regression with different link functions, 
using which the models can be easily compared with respect to both forward and inverse
perspectives using FPBF and IPBF. 
Note that forward and inverse perspectives can also be compared.

Observe that the forward Poisson log-linear regression turns out to be the best model as expected, since this corresponds to the true, data-generating distribution.
The Gaussian process based Poisson inverse regression model is the next best, followed closely by the Poisson log-linear inverse regression model, and then comes
the Gaussian process based Poisson forward regression model. This order of model selection can be explained as follows. First, the inverse cases involve more
uncertainties than the corresponding forward models, since these cases treat $x_i$ as unknown. Hence, expectedly the Poisson log-linear forward regression model 
outperforms the inverse counterpart. But the inverse Gaussian process regression performs marginally better than the inverse linear model and more
significantly better than the forward Gaussian process
model. This merits an interesting explanation. Recall that in the inverse Gaussian process model $\eta(\tilde x_i)$ has been linearized for constructing the prior
for $\tilde x_i$, so that this part is equivalent to the linear model, which explains why the difference between the inverse linear and Gaussian process models
is not significant. However, 
the linear part of the Gaussian process model is of course influenced by the additional Gaussian process part associated with the other data points,
unlike the linear regression models. The posterior dependence structure, in conjunction with the posterior distribution of $\tilde x_i$, 
can yield better regression estimates $\eta(\tilde x_i)$ for the $i$-th data point in a substantial number of Monte Carlo iterations.  
Since the Gaussian process model includes the linear model as a special case (that is, it is not a case of misspecification), 
this explains why the inverse Gaussian process regression performs marginally better than the inverse linear model. In the forward Gaussian process regression,
even though we have linearized $\eta(x_i)$ for comparability with the inverse model, $x_i$ is fixed. Thus, when the $i$-th regression part is not well-estimated
in the Monte Carlo simulations, there is no further scope for improvement in this part. However, in the inverse Gaussian process regression, $x_i$ is replaced with
the random $\tilde x_i$, which, though its posterior simulations, can improve upon the $i$-th regression part with positive probability, even if the regression  
coefficients are not well-estimated. Thus, the inverse Gaussian process regression model can significantly outperform
the forward counterpart, as we observe here.

The geometric logit and probit linear and Gaussian process regressions are examples of model misspecifications since the true, data-generating model is the Poisson
log-linear regression model. Accordingly, both the forward and inverse setups perform worse than the Poisson regression setups. 
Among the forward and inverse cases for geometric regression, the probit linear model performs the best, followed closely by the logit linear model, then
by the forward logit Gaussian process and then by the forward probit Gaussian process -- all the inverse regression models perform worse than the worse
of the forward regression models. This is not surprising since all these models are cases of misspecifications and given the data generated from the true model, 
the inverse models here only increase
the uncertainty regarding $x_i$ compared to the forward models without any positive effect. However, note that the inverse logit Gaussian process model significantly
outperforms the inverse logit linear model thanks to its better flexibility and similar prior structure for $\tilde x_i$ as in the case of the true log-linear Poisson
regression whose positive effects carry over to this case from the first two rows of the last column of Table \ref{table:bayesian}. But the same phenomenon of superiority
of the inverse probit Gaussian process over inverse probit linear model is not at all visible since the prior structure of $\tilde x_i$ in this misspecified case is completely
different from that of the true Poisson log-linear model, and indeed, inconsistent.

\begin{table}[h]
\centering
	\caption{Results of our simulation study for model selection using FPBF and IPBF. The last two columns show the estimates of 
	$\frac{1}{n}\sum_{i=1}^n\log\pi(y_{i1}|\bY_{nm,-i},\bX_n,\mathcal M)$ and $\frac{1}{n}\sum_{i=1}^n\log\pi(y_{i1}|\bY_{nm,-i},\bX_{n,-i},\mathcal M)$,
	respectively, for forward and inverse setups.}
\label{table:bayesian}
	\begin{tabular}{|c|c|c|c|c|}
\hline
		Model & Link function & Regression form & Forward & Inverse\\ 
\hline
		$Poisson(\lambda(x_i))$  &  log  & linear &  $-7.913$  &  $-8.440$\\
		$Poisson(\lambda(x_i))$  &  log  & Gaussian process &  $-8.503$  &  $-8.409$\\
		$Geometric(p(x_i))$  &  logit  & linear & $-9.176$  &  $-18.247$\\
		$Geometric(p(x_i))$  &  logit  & Gaussian process & $-9.529$  &  $-14.766$\\
		$Geometric(p(x_i))$  &  probit  & linear &  $-9.348$  &  $-14.434$\\
		$Geometric(p(x_i))$  &  probit  & Gaussian process &  $-10.915$  &  $-23.733$\\
\hline
\end{tabular}
\end{table}

\subsection{Variable selection in Poisson and  geometric linear and nonparametric regression models when true model is Poisson linear regression}
\label{subsec:poisson_vs_geometric_vs}
Rather than a single covariate $x$ in the previous examples, let us now consider covariates $x$ and $z$, where the
true data-generating distribution
is $y_{ij}\sim Poisson(\lambda(x_i,z_i))$, with $\lambda(x,z)=\exp(\alpha_0+\beta_0 x+\gamma_0 z)$. 
We generate the data by simulating $\alpha_0,\beta_0,\gamma_0\sim U(-1,1)$, independently;
and $x_i\sim U(-1,1)$, $z_i\sim U(0,2)$; $i=1,\ldots,n$, and then finally simulating $y_{ij}\sim Poisson(\lambda(x_i,z_i))$; $j=1,\ldots,m$, $i=1,\ldots,n$.

We model the data $y_{ij}$; $i=1,\ldots,n$; $j=1,\ldots,m$ with both Poisson and geometric models as before with the regression part consisting of
either $x$ or $z$, or both. We denote the linear regression coefficients of the intercept, $x$ and $z$ as $\alpha$, $\beta$ and $\gamma$, respectively,
and give the improper prior density $1$ to $(\alpha,\beta)$, $(\alpha,\gamma)$ and $(\alpha,\beta,\gamma)$ when the models consist of these combinations of parameters. 
For Gaussian process regression with both $x$ and $z$, we let $\eta(x,z)$ be the regression function modeled by a Gaussian process with mean 
$\mu(x,z)=\alpha+\beta x+\gamma z$ and covariance function $Cov\left(\eta(x_1,z_1),\eta(x_2,z_2)\right)=\exp\left(\omega\right)
\exp\left[-\left\{(x_1-x_2)^2+(z_1-z_2)^2\right\}\right]$, and we assign prior mass $1$ to 
$(\alpha,\beta,\omega)$, $(\alpha,\gamma,\omega)$ and $(\alpha,\beta,\gamma,\omega)$ when the models consist of the covariates $x$, $z$ or both.
Using FPBF and IPBF we then compare the different models, along with the covariates associated with them. 
In the inverse cases, where the model consists of the single covariate $x$ or $z$, then the priors for $\tilde x_i$ and $\tilde z_i$ remain the same as
in the previous cases. 

But wherever the models consist of both the covariates $x$ and $z$, we need to assign priors for both $\tilde x_i$ and $\tilde z_i$,
in addition to requiring that $E(y_{ij}|\theta,x_i,z_i)$ under the postulated model fall in $\left[\bar y_i-\frac{c_1s_i}{\sqrt{m}},\bar y_i+\frac{c_2s_i}{\sqrt{m}}\right]$.
The same priors for $\tilde x_i$ and $\tilde z_i$ as the previous situations where the models consisted of single covariates, will not be consistent in these situations.
For consistent priors we adopt the following strategy.
Letting $\alpha$ be the intercept, $\beta$ and $\gamma$ the coefficients of $x_i$ and $z_i$ respectively in the regression forms, we envisage the following 
priors for $\tilde x_i$ and $\tilde z_i$.

\subsubsection{Prior for $\tilde x_i$ and $\tilde z_i$ for Poisson regression}
\label{subsubsec:x_z_prior_poisson}
For the Poisson linear or Gaussian process regression model with log link consisting of both the covariates $x$ and $z$, we set 
$\tilde x_i\sim U\left(a^{(1)}_x,b^{(1)}_x\right)$ and  $\tilde z_i\sim U\left(a^{(1)}_z,b^{(1)}_z\right)$, where
\begin{equation*}
	a^{(1)}_x=\min\left\{\beta^{-1}\left(\log\left(\bar y_i-\frac{c_1s_i}{\sqrt{m}}\right)-\alpha-\gamma z_i\right),
\beta^{-1}\left(\log\left(\bar y_i+\frac{c_2s_i}{\sqrt{m}}\right)-\alpha-\gamma z_i\right)\right\},
	%\label{eq:ax1}
\end{equation*}
 
\begin{equation*}
	b^{(1)}_x=\max\left\{\beta^{-1}\left(\log\left(\bar y_i-\frac{c_1s_i}{\sqrt{m}}\right)-\alpha-\gamma z_i\right),
\beta^{-1}\left(\log\left(\bar y_i+\frac{c_2s_i}{\sqrt{m}}\right)-\alpha-\gamma z_i\right)\right\},
	%\label{eq:bx1}
\end{equation*}

\begin{equation*}
	a^{(1)}_z=\min\left\{\gamma^{-1}\left(\log\left(\bar y_i-\frac{c_1s_i}{\sqrt{m}}\right)-\alpha-\beta x_i\right),
\gamma^{-1}\left(\log\left(\bar y_i+\frac{c_2s_i}{\sqrt{m}}\right)-\alpha-\beta x_i\right)\right\}
	%\label{eq:az1}
\end{equation*}
and 
\begin{equation*}
	b^{(1)}_z=\max\left\{\gamma^{-1}\left(\log\left(\bar y_i-\frac{c_1s_i}{\sqrt{m}}\right)-\alpha-\beta x_i\right),
\gamma^{-1}\left(\log\left(\bar y_i+\frac{c_2s_i}{\sqrt{m}}\right)-\alpha-\beta x_i\right)\right\}.
	%\label{eq:bz1}
\end{equation*}
Note that the priors for $\tilde x_i$ and $\tilde z_i$ depend upon $z_i$ and $x_i$ respectively.
This is somewhat in keeping with (\ref{eq:B2}) where the prior for $\tilde x_i$ depends upon $x_i$ itself. 
The discussion following (\ref{eq:B2}) is enough to justify that the priors for $\tilde x_i$ and $\tilde z_i$ in the current situation do make sense, apart from
ensuring consistency. 
%It is unusual in Bayesian inference to make the prior depend upon the truth. 
%Indeed, the true parameter is always unknown; had it been known, then one would give full prior probability to the true
%parameter. In our case $x_i$ is actually known but a prior is needed for $\tilde x_i$ for the sake of cross-validation. Moreover, the prior does not consider 
%$x_i$ to be known as long as the sample sizes $n$ and $m$ remain finite and $\theta$ is unknown or takes false values. The prior has substantial variance
%in these cases. Hence, although unusual, such a prior on $\tilde x_i$ is not untenable for inverse cross-validation.

\subsubsection{Prior for $\tilde x_i$ and $\tilde z_i$ for geometric regression with logit link}
\label{subsubsec:x_z_prior_geo_logit}
For the geometric linear or Gaussian process regression model with logit link consisting of both the covariates $x$ and $z$, we set 
$\tilde x_i\sim U\left(a^{(2)}_x,b^{(2}_x\right)$ and  $\tilde z_i\sim U\left(a^{(2)}_z,b^{(2)}_z\right)$, where
\begin{equation*}
	a^{(2)}_x=\min\left\{-\beta^{-1}\left(\log\left(\bar y_i-\frac{c_1s_i}{\sqrt{m}}\right)+\alpha+\gamma z_i\right),
-\beta^{-1}\left(\log\left(\bar y_i+\frac{c_2s_i}{\sqrt{m}}\right)+\alpha+\gamma z_i\right)\right\},
	%\label{eq:ax2}
\end{equation*}

\begin{equation*}
	b^{(2)}_x=\max\left\{-\beta^{-1}\left(\log\left(\bar y_i-\frac{c_1s_i}{\sqrt{m}}\right)+\alpha+\gamma z_i\right),
-\beta^{-1}\left(\log\left(\bar y_i+\frac{c_2s_i}{\sqrt{m}}\right)+\alpha+\gamma z_i\right)\right\},
	%\label{eq:bx2}
\end{equation*}

\begin{equation*}
	a^{(2)}_z=\min\left\{-\gamma^{-1}\left(\log\left(\bar y_i-\frac{c_1s_i}{\sqrt{m}}\right)+\alpha+\beta x_i\right),
-\gamma^{-1}\left(\log\left(\bar y_i+\frac{c_2s_i}{\sqrt{m}}\right)+\alpha+\beta x_i\right)\right\}
	%\label{eq:az2}
\end{equation*}
and 
\begin{equation*}
	b^{(2)}_z=\max\left\{-\gamma^{-1}\left(\log\left(\bar y_i-\frac{c_1s_i}{\sqrt{m}}\right)+\alpha+\beta x_i\right),
-\gamma^{-1}\left(\log\left(\bar y_i+\frac{c_2s_i}{\sqrt{m}}\right)+\alpha+\beta x_i\right)\right\}.
	%\label{eq:bz2}
\end{equation*}

\subsubsection{Prior for $\tilde x_i$ and $\tilde z_i$ for geometric regression with probit link}
\label{subsubsec:x_z_prior_geo_probit}
For the geometric linear or Gaussian process regression model with probit link consisting of both the covariates $x$ and $z$, we set 
$\tilde x_i\sim U\left(a^{(3)}_x,b^{(3}_x\right)$ and  $\tilde z_i\sim U\left(a^{(3)}_z,b^{(3)}_z\right)$, where
%For geometric probit regression, first let $\ell_{im}=\bar y_i-\frac{c_1s_i}{\sqrt{m}}$ and $u_{im}=\bar y_i+\frac{c_2s_i}{\sqrt{m}}$. Let
\begin{equation*}
a^{(3)}_x=\min\left\{\frac{\Phi^{-1}\left(\frac{1}{u_{im}+1}\right)-\alpha-\gamma z_i}{\beta},
	\frac{\Phi^{-1}\left(\frac{1}{\ell_{im}+1}\right)-\alpha-\gamma z_i}{\beta}\right\},
	%\label{eq:ax3}
\end{equation*}

\begin{equation*}
b^{(3)}_x=\max\left\{\frac{\Phi^{-1}\left(\frac{1}{u_{im}+1}\right)-\alpha-\gamma z_i}{\beta},
	\frac{\Phi^{-1}\left(\frac{1}{\ell_{im}+1}\right)-\alpha-\gamma z_i}{\beta}\right\},
	%\label{eq:bx3}
\end{equation*}

\begin{equation*}
a^{(3)}_z=\min\left\{\frac{\Phi^{-1}\left(\frac{1}{u_{im}+1}\right)-\alpha-\beta x_i}{\gamma},
	\frac{\Phi^{-1}\left(\frac{1}{\ell_{im}+1}\right)-\alpha-\beta x_i}{\gamma}\right\}
	%\label{eq:az3}
\end{equation*}
and
\begin{equation*}
b^{(3)}_z=\max\left\{\frac{\Phi^{-1}\left(\frac{1}{u_{im}+1}\right)-\alpha-\beta x_i}{\gamma},
	\frac{\Phi^{-1}\left(\frac{1}{\ell_{im}+1}\right)-\alpha-\beta x_i}{\gamma}\right\}.
	%\label{eq:bz3}
\end{equation*}

\subsubsection{Results of the simulation experiment for model and variable selection}
\label{subsubsec:results_model_var}

For $n=m=10$, when the true model is Poisson with log-linear regression on both the covariates $x$ and $z$, the last two columns of
Table \ref{table:bayesian2} provide the estimates of $\frac{1}{n}\sum_{i=1}^n\log\pi(y_{i1}|\bY_{nm,-i},\bX_n,\mathcal M)$ and 
$\frac{1}{n}\sum_{i=1}^n\log\pi(y_{i1}|\bY_{nm,-i},\bX_{n,-i},\mathcal M)$ for Poisson and geometric linear and Gaussian process regression on either $x_i$ or $z_i$ or both, 
with different link functions. Thus, the models, along with the associated covariates can be compared with respect to both forward and inverse
perspectives. %using FPBF and IPBF. 

Table \ref{table:bayesian2} shows that the correct Poisson log-linear model with both the covariates $x$ and $z$ has turned out to be the third best, after
the inverse Poisson log-linear model with covariate $x$ and the forward Poisson log-linear model with covariate $z$. However, the difference between
the latter and the correct model is not substantial and may perhaps be attributed to Monte Carlo sampling fluctuations. So, considering only the forward setup,
it is difficult to rule out the possibility of the correct Poisson log-linear model with both the covariates $x$ and $z$ from being the best.

That the inverse Poisson log-linear model with covariate $x$ seems to perform so well can be attributed to significant variability of the prior for $\tilde x_i$
which goes on to account for the missing $z_i$ as well in the additive model. Since the additive model is not identifiable when both $x_i$ and $z_i$ are unknown,
the significant prior variability of $\tilde x_i$ compensates for non-inclusion of $z_i$ in the model, given the data that has arisen from the true model consisting
of both $x$ and $z$. The same argument is valid for good performance of the inverse Poisson log-linear model with covariate $z$, where the prior variance for
$\tilde z_i$ compensates for non-inclusion of $x_i$. However, note that the performance of the inverse Poisson log-linear model deteriorates
significantly when the regression consists of both $x$ and $z$. This is of course the consequence of the priors for both $\tilde x_i$ and $\tilde z_i$, whose
variances get added up in the linear model. For small $n$ and $m$ as in our examples, the true values $x_i$ and $z_i$ fail to get enough posterior weight, an issue  that
gets reflected in the Monte Carlo simulations where the true regression is not represented in sufficiently large proportion.

For Poisson Gaussian process regression, the inverse models outperform their forward counterparts by large margins. 
This admits similar explanation provided in Section \ref{subsubsec:results_model} for the superiority of the inverse Poisson Gaussian process model 
compared to its forward counterpart as visible in Table \ref{table:bayesian}.  

For geometric linear regression, the forward models emerge the winners in all the cases, as opposed to the inverse counterparts and also outperform the Gaussian
geometric process regression models. Among the geometric models, the probit linear model with both the covariates $x$ and $z$, turns out to be the best.
That the corresponding inverse counterparts perform worse can be explained as in Section \ref{subsubsec:results_model} that these are instances of model
misspecification, and here the inverse models only increase uncertainty by treating $x_i$ and $z_i$ as unknown, without any beneficial effect. 

In geometric Gaussian process regression, the inverse models perform better than the corresponding forward ones in most cases. In these cases, given the
data generated from the true model, the Gaussian process
dependence combined with the prior variability render the inverse models somewhat less misspecified than the forward models with no prior associated with the covariates.

Also observe that given either forward or inverse setups, the linear models perform better than the corresponding Gaussian process models, for both Poisson and 
geometric cases. Since the true regression is linear, this seems to provide an internal consistency.
However, this phenomenon is somewhat different from that observed in Table \ref{table:bayesian} where the Gaussian process model performed better than the linear regression model
for Poisson and geometric logit models. The reason for this is inconsistency of the prior for $\tilde x_i$ when covariate $z$ is ignored and that of the 
prior for $\tilde z_i$ when covariate $x$ is ignored in the postulated model. Indeed, Table \ref{table:bayesian2} shows that in these cases, the inverse linear models
outperform the Gaussian process models by considerably large margins. In these cases the Gaussian process priors only increase uncertainties without adding any value,
since the priors for $\tilde x_i$ and $\tilde z_i$ are inconsistent. On the other hand, note that when both $x$ and $z$ are incorporated in the inverse models,
the linear models perform only marginally better than the Gaussian process models in the cases of inverse Poisson and inverse geometric logit models. This is 
because the priors of $\tilde x_i$ and $\tilde z_i$ are consistent in such cases, and moreover, the prior structures of $\tilde x_i$ and $\tilde z_i$ are similar
for Poisson and geometric logit regressions. For geometric probit regression, the prior structures are entirely different from those of the correct
Poisson model and in fact inconsistent, and as in Table \ref{table:bayesian}, here also inverse geometric probit Gaussian process regression performs much worse than 
inverse geometric probit linear regression. 

\begin{table}[h]
\centering
	\caption{Results of our simulation study for model and variable selection using FPBF and IPBF. The last two columns show the estimates of 
	$\frac{1}{n}\sum_{i=1}^n\log\pi(y_{i1}|\bY_{nm,-i},\bX_n,\mathcal M)$ and $\frac{1}{n}\sum_{i=1}^n\log\pi(y_{i1}|\bY_{nm,-i},\bX_{n,-i},\mathcal M)$,
	respectively, for forward and inverse setups.}
\label{table:bayesian2}
	\begin{tabular}{|c|c|c|c|c|c|}
\hline
		Covariates & Model & Link function & Regression form & Forward & Inverse\\ 
\hline
		$x_i$	&  $Poisson(\lambda(x_i))$  &  log  & linear &  $-8.618$  &  $-8.388$\\
		$z_i$	&  $Poisson(\lambda(z_i))$  &  log  & linear &  $-8.834$  &  $-8.739$\\
		$(x_i,z_i)$	&  $Poisson(\lambda(x_i,z_i))$  &  log  & linear &  $-8.686$  &  $-13.257$\\
		$x_i$ & $Poisson(\lambda(x_i))$  &  log  & Gaussian process &  $-31.831$  &  $-9.136$\\
		$z_i$ & $Poisson(\lambda(z_i))$  &  log  & Gaussian process &  $-31.213$  &  $-10.052$\\
		$(x_i,z_i)$ & $Poisson(\lambda((x_i,z_i))$  &  log  & Gaussian process &  $-17.712$  &  $-13.363$\\
		$x_i$ & $Geometric(p(x_i))$  &  logit  & linear & $-9.810$  &  $-10.526$\\
		$z_i$ & $Geometric(p(z_i))$  &  logit  & linear & $-9.673$  &  $-12.629$\\
		$(x_i,z_i)$ & $Geometric(p(x_i,z_i))$  &  logit  & linear & $-11.806$  &  $-15.478$\\
		$x_i$ & $Geometric(p(x_i))$  &  logit  & Gaussian process & $-26.232$  &  $-21.161$\\
		$z_i$ & $Geometric(p(z_i))$  &  logit  & Gaussian process & $-19.391$  &  $-29.388$\\
		$(x_i,z_i)$ & $Geometric(p(x_i,z_i))$  &  logit  & Gaussian process & $-17.128$  &  $-15.686$\\
		$x_i$ & $Geometric(p(x_i))$  &  probit  & linear &  $-9.543$  &  $-11.671$\\
		$z_i$ & $Geometric(p(z_i))$  &  probit  & linear &  $-9.401$  &  $-16.183$\\
		$(x_i,z_i)$ & $Geometric(p(x_i,z_i))$  &  probit  & linear &  $-9.060$  &  $-13.839$\\
		$x_i$ & $Geometric(p(x_i))$  &  probit  & Gaussian process &  $-23.538$  &  $-16.460$\\
		$z_i$ & $Geometric(p(z_i))$  &  probit  & Gaussian process &  $-20.522$  &  $-17.099$\\
		$(x_i,z_i)$ & $Geometric(p(x_i,z_i))$  &  probit  & Gaussian process &  $-20.102$  &  $-20.501$\\
\hline
\end{tabular}
\end{table}

\section{Summary and future direction}
\label{sec:conclusion}

The importance of PBF in Bayesian model and variable selection seems to have been overlooked in the statistical literature. In this article we have pointed out the 
theoretical and computational advantages of PBF over BF, and investigated the asymptotic convergence properties of PBF in general forward and inverse regression setups. 
Since the inverse regression problem requires a prior on the covariate value to be predicted, this makes the treatise of PBF distinct from the forward regression problems.
Specifically, we considered two setups for inverse regression. One setup is the same as that of forward regression except a prior for the relevant covariate value $\tilde x_i$.
Although the priors in this case can not guarantee consistency of the posterior for $\tilde x_i$, we show that the corresponding PBF still converges exponentially
and almost surely in favour of the better model, in the same way as for forward regression. However, for the inverse case, the convergence depends upon an integrated version
of the KL-divergence, rather than KL-divergence as in the forward case. In another inverse regression setup, we consider $m$ responses corresponding to each covariate value,
and assign the general prior for $\tilde x_i$ constructed by \ctn{Chat20}. This prior guarantees consistency for the posterior of $\tilde x_i$ when $m$ tends to
infinity, along with the sample size. For this inverse setup, PBF has convergence results similar to that of forward regression which is also
applicable to this setup, except that no prior is associated with the covariates. 

Our results on PBF for forward regression are in agreement with the general BF convergence theory established in \ctn{Chatterjee18}, as both are the same almost sure
exponential convergence depending upon the KL-divergence from the true model. Now there might arise the question if PBF and BF convergence agree even for inverse
regression setups. To clarify, first recall that BF is the ratio of the marginal densities of the data. Now for forward regression, the marginal density of the data $\bY_n$	
depends upon the observed covariates $\bX_n$. For model $\mathcal M_j$; $j=1,2$, let us denote this marginal by $m(\bY_n|\bX_n,\mathcal M_j)$. In the inverse setup,
we need to treat $\bX_n$ as unknown, and replace this with $\tilde\bX_n=(\tilde x_1,\tilde x_2,\ldots,\tilde x_n)$ having some relevant prior, which may even follow from
some stochastic process specification for $\tilde\bX_{\infty}=(\tilde x_1,\tilde x_2,\ldots)$. If $L(\theta_j|\bY_n,\bX_n,\mathcal M_j)$ denotes the 
likelihood of $\theta_j$ for fully observed data, then the marginal density of $\bY_n$ in the inverse situation is given by
\begin{align*}
	\tilde m(\bY_n|\mathcal M_j)&=\int_{\Theta_j}\int_{\mathcal X^n}L(\theta_j|\bY_n,\tilde\bX_n,\mathcal M_j)
	d\pi(\tilde\bX_n|\theta_j,\mathcal M_j)d\pi(\theta_j|\mathcal M_j)\notag\\
	&=\int_{\Theta_j}\tilde L(\theta_j|\bY_n,\mathcal M_j)d\pi(\theta_j|\mathcal M_j),
	%\label{eq:marg_inv}
\end{align*}
where
\begin{equation*}
	\tilde L(\theta_j|\bY_n,\mathcal M_j)=\int_{\mathcal X^n}L(\theta_j|\bY_n,\tilde\bX_n,\mathcal M_j)d\pi(\tilde\bX_n|\theta_j,\mathcal M_j).
	%\label{eq:like_inv}
\end{equation*}
Letting 
\begin{equation*}
	\tilde\pi(\theta_j|\bY_n,\mathcal M_j)=\frac{\tilde L(\theta_j|\bY_n,\mathcal M_j)\pi(\theta_j|\mathcal M_j)}{\tilde m(\bY_n|\mathcal M_j)}
\end{equation*}
we have for all $\theta_j\in\Theta_j$,
\begin{equation*}
	\log \tilde m(\bY_n|\mathcal M_j) = \log \tilde L(\theta_j|\bY_n,\mathcal M_j)+\log \pi(\theta_j|\mathcal M_j)
	-\log \tilde\pi(\theta_j|\bY_n,\mathcal M_j),
\end{equation*}
which reduces the inverse marginal to the same form as that used by \ctn{Chatterjee18} for establishing their almost sure exponential BF convergence result
which depends explicitly on the KL-divergence rate between the postulated and the true models.
Hence, even in both the inverse setups that we consider, our PBF and BF convergence results agree.

We have illustrated our general asymptotic results for PBF with several theoretical examples, including linear, quadratic, AR(1) regression and 
variable selection, providing the explicit theoretical calculations for both forward and inverse setups. 
Our AR(1) regression results validate our general PBF convergence theory in a dependent data setup.

We also conducted extensive simulation experiments with small simulated datasets comparing Poisson log regression and geometric logit and probit regressions,
where the regressions are modeled by straight lines as well as Gaussian process based nonparametric functions. Both forward and inverse setups are undertaken,
which include, in addition, variable selection among two possible covariates. Among several insightful revelations, our results demonstrate 
that the inverse regression can outperform the forward counterpart when the regression considered is nonparametric. 

Thus, overall the premise for PBF investigation seems promising enough to pursue further research. In particular, we shall address PBF based variable selection in both forward
and inverse regression contexts in the so-called ``large $p$, small $n$" framework, where the number of variables considered increases with sample size with various
rates, crucially, at rates faster than the sample size. Various complex and high-dimensional real data based applications shall also be considered for model and
variable selection using forward and inverse PBF. More sophisticated computational methods combining advanced versions of TMCMC, bridge sampling and path sampling 
may need to be created for accurate estimations of PBF in such real situations. These ideas will be communicated elsewhere.

\newpage

\begin{appendix}

\section*{Appendix}

\section{Preliminaries for ensuring posterior consistency under general setup}
\label{sec:shalizi}

Following \ctn{Shalizi09} we consider a probability space $(\Omega,\mathcal F, P)$, 
and a sequence of random variables $y_1,y_2,\ldots$,   
taking values in some measurable space $(\Xi,\mathcal Y)$, whose
infinite-dimensional distribution is $P$. Let $\bY_n=\{y_1,\ldots,y_n\}$. The natural filtration of this process is
$\sigma(\bY_n)$, the smallest $\sigma$-field with respect to which $\bY_n$ is measurable. %where $\bY_n=(Y_1,Y_2,\ldots,Y_n)^T$.

We denote the distributions of processes adapted to $\sigma(\bY_n)$ 
by $F_{\theta}$, where $\theta$ is associated with a measurable
space $(\Theta,\mathcal T)$, and is generally infinite-dimensional. 
For the sake of convenience, we assume, as in \ctn{Shalizi09}, that $P$
and all the $F_{\theta}$ are dominated by a common reference measure, with respective
densities $f_{\theta_0}$ and $f_{\theta}$. The usual assumptions that $P\in\Theta$ or even $P$ lies in the support 
of the prior on $\Theta$, are not required for Shalizi's result, rendering it very general indeed.

\subsection{Assumptions and theorems of Shalizi}
\label{subsec:assumptions_shalizi}

\begin{itemize}
\item[(S1)] Consider the following likelihood ratio:
\begin{equation*}
R_n(\theta)=\frac{f_{\theta}(\bY_n)}{f_{\theta_0}(\bY_n)}.
%\label{eq:R_n}
\end{equation*}
Assume that $R_n(\theta)$ is $\sigma(\bY_n)\times \mathcal T$-measurable for all $n>0$.
\end{itemize}

\begin{itemize}
\item[(S2)] For every $\theta\in\Theta$, the KL-divergence rate
\begin{equation*}
h(\theta)=\underset{n\rightarrow\infty}{\lim}~\frac{1}{n}E\left(\log\frac{f_{\theta_0}(\bY_n)}{f_{\theta}(\bY_n)}\right).
%\label{eq:S3}
\end{equation*}
exists (possibly being infinite) and is $\mathcal T$-measurable.
\end{itemize}

\begin{itemize}
\item[(S3)] For each $\theta\in\Theta$, the generalized or relative asymptotic equipartition property holds, and so,
almost surely,
\begin{equation*}
\underset{n\rightarrow\infty}{\lim}~\frac{1}{n}\log R_n(\theta)=-h(\theta).
\end{equation*}
\end{itemize}

\begin{itemize}
\item[(S4)] 
Let $I=\left\{\theta:h(\theta)=\infty\right\}$. 
The prior $\pi$ satisfies $\pi(I)<1$.
\end{itemize}

%Following the notation of \ctn{Shalizi09}, for $A\subseteq\Theta$, let
%\begin{align}
%h\left(A\right)&=\underset{\theta\in A}{\mbox{ess~inf}}~h(\theta);\label{eq:h2}\\
%J(\theta)&=h(\theta)-h(\Theta);\label{eq:J}\\
%J(A)&=\underset{\theta\in A}{\mbox{ess~inf}}~J(\theta).\label{eq:J2}
%\end{align}
\begin{itemize}
\item[(S5)] There exists a sequence of sets $\mathcal G_n\rightarrow\Theta$ as $n\rightarrow\infty$ 
such that: %along with $\pi(\mathcal G_T)>0$
\begin{enumerate}
\item[(1)]
\begin{equation}
\pi\left(\mathcal G_n\right)\geq 1-\zeta\exp\left(-\gamma n\right),~\mbox{for some}~\zeta>0,~\gamma>2h(\Theta);
\label{eq:S5_1}
\end{equation}
\item[(2)]The convergence in (S3) is uniform in $\theta$ over $\mathcal G_n\setminus I$.
\item[(3)] $h\left(\mathcal G_n\right)\rightarrow h\left(\Theta\right)$, as $n\rightarrow\infty$.
\end{enumerate}
\end{itemize}
For each measurable $A\subseteq\Theta$, for every $\delta>0$, there exists a random natural number $\tau(A,\delta)$
such that
\begin{equation}
n^{-1}\log\int_{A}R_n(\theta)\pi(\theta)d\theta
\leq \delta+\underset{n\rightarrow\infty}{\lim\sup}~n^{-1}
\log\int_{A}R_n(\theta)\pi(\theta)d\theta,
\label{eq:limsup_2}
\end{equation}
for all $n>\tau(A,\delta)$, provided 
$\underset{n\rightarrow\infty}{\lim\sup}~n^{-1}\log\pi\left(\mathbb I_A R_n\right)<\infty$.
%$\mathbb I_A$ denotes the indicator function of the set $A$.
Regarding this, the following assumption has been made by Shalizi:
\begin{itemize}
\item[(S6)] The sets $\mathcal G_n$ of (S5) can be chosen such that for every $\delta>0$, the inequality
$n>\tau(\mathcal G_n,\delta)$ holds almost surely for all sufficiently large $n$.
\end{itemize}
\begin{itemize}
\item[(S7)] The sets $\mathcal G_n$ of (S5) and (S6) can be chosen such that for any set $A$ with $\pi(A)>0$, 
\begin{equation}
h\left(\mathcal G_n\cap A\right)\rightarrow h\left(A\right),
\label{eq:S7}
\end{equation}
as $n\rightarrow\infty$.
\end{itemize}
%Under the above assumptions, \ctn{Shalizi09} proved the following results.

%\begin{theorem}[\ctn{Shalizi09}]
%\label{theorem:shalizi1}
%Consider assumptions (S1)--(S7) and any set $A\in\mathcal T$ with $\pi(A)>0$ and $h(A)>h(\Theta)$. Then,
%\begin{equation*}
%\underset{n\rightarrow\infty}{\lim}~\pi(A|\bY_n)=0~\mbox{almost surely},
%%\label{eq:supp_post_conv1}
%\end{equation*}
%where $\pi(\cdot|\bY_n)$ denotes the posterior distribution of $\theta$ given $\bY_n$.
%\end{theorem}

\section{A result on sufficient condition for (S6) of Shalizi}
\label{sec:s6}
\begin{theorem}
	\label{theorem:shalizi_Rn}
	Consider the following assumptions:
	\begin{itemize}
	%	\item[(i)] $\Theta$ is finite-dimensional.
		\item[(i)] Let $\tilde\theta=\underset{\theta\in\Theta}{\arg\min}~h(\theta)$ be the unique minimizer of $h(\theta)$ on $\Theta$. 
	%Also, let $\bar\Theta$ be a compact subset of $\Theta$ such that
	%$\underset{\theta\in\bar\Theta}{\arg\min}~h(\theta)=\tilde\theta$. 
		\item[(ii)] Let $\tilde\theta^*_n=\underset{\theta\in\Theta}{\arg\max}~\frac{1}{n}\log R_n(\theta)$, and
	assume that $\tilde\theta^*_n\stackrel{a.s.}{\longrightarrow}\tilde\theta,~\mbox{as}~n\rightarrow\infty$.
\item[(iii)] $\frac{1}{n}\log R_n(\theta)$ is stochastically equicontinuous on %$\bar\Theta$
	%and that for all $\theta\in\bar\Theta$,
	compact subsets of $\Theta$. 
\item[(iv)] For all $\theta$ in such compact subsets,
	\begin{equation}
		\underset{n\rightarrow\infty}{\lim}~\frac{1}{n}\log R_n(\theta)=h(\theta),~\mbox{almost surely}.
		\label{eq:shalizi_ass1}
	\end{equation}
\item[(v)] The prior $\pi$ on $\Theta$ is proper.
	\end{itemize}
	%Then, provided that $\Theta$ is finite-dimensional and the prior $\pi$ on $\Theta$ is proper, 
	Then (\ref{eq:shalizi_Rn}) holds. 
	%it holds that

\end{theorem}
\begin{proof}
	Note that 
	\begin{align}
	\frac{1}{n}\log\int_{\mathcal G_n}R_n(\theta)\pi(\theta)d\theta
		&\leq \frac{1}{n}\log\left(\underset{\theta\in\mathcal G_n}{\sup}~R_n(\theta)\right)+\frac{1}{n}\log\pi(\mathcal G_n)\notag\\
		&=\underset{\theta\in\mathcal G_n}{\sup}~\frac{1}{n}\log R_n(\theta)+\frac{1}{n}\log\pi(\mathcal G_n)\notag\\
		&=\frac{1}{n}\log R_n(\tilde\theta^*_n)+\frac{1}{n}\log\pi(\mathcal G_n).
		\label{eq:shalizi0}
	\end{align}
	Since by condition (ii), $\tilde\theta^*_n\stackrel{a.s.}{\longrightarrow}\tilde\theta$ as $n\rightarrow\infty$, 
	for any $\epsilon>0$, there exists $n_0(\epsilon)\geq 1$ such that for $n\geq n_0(\epsilon)$,  
	\begin{equation}
		\tilde\theta^*_n\in (\tilde\theta-\epsilon,\tilde\theta+\epsilon),~\mbox{almost surely}.
		\label{eq:shalizi_ass2}
	\end{equation}
	%Note that (\ref{eq:shalizi_ass2}) can not be expected to hold in infinite-dimensional setups without problem-specific sparsity assumptions, which we do not
	%consider here for the present purpose.

	Conditions (iii) and (iv) validate the stochastic Ascoli lemma, and hence, for any compact subset $G$ of $\Theta$ that contains 
	$(\tilde\theta-\epsilon,\tilde\theta+\epsilon)$,
	\begin{equation*}
		\underset{n\rightarrow\infty}{\lim}\underset{\theta\in G}{\sup}~\left|\frac{1}{n}\log R_n(\theta)+h(\theta)\right|=0,~\mbox{almost surely}.
	\end{equation*}
	Hence, for any $\xi>0$, for all $\theta\in G$, almost surely,
	\begin{equation}
		\frac{1}{n}\log R_n(\theta)\leq -h(\theta)+\eta
		\leq -h(\Theta)+\eta,~\mbox{for sufficiently large}~n.
		\label{eq:shalizi_ass3}
	\end{equation}
	Since $G$ contains $(\tilde\theta-\epsilon,\tilde\theta+\epsilon)$, which, in turn contains $\tilde\theta^*_n$ for sufficiently large $n$, due to 
	(\ref{eq:shalizi_ass2}), it follows from (\ref{eq:shalizi_ass3}), that for any $\xi>0$,
	\begin{equation}
		\frac{1}{n}\log R_n(\tilde\theta^*_n)
		\leq -h(\Theta)+\eta,~\mbox{for sufficiently large}~n.
		\label{eq:shalizi_ass4}
	\end{equation}
	The proof follows by combining (\ref{eq:shalizi0}) and (\ref{eq:shalizi_ass4}), and noting that $\frac{1}{n}\log\pi(\mathcal G_n)<0$ for all $n\geq 1$,
	since $0<\pi(\mathcal G_n)<1$ for proper priors.

\end{proof}

\section{Proof of Theorem \ref{th:uniroot}}
\label{sec:proof1}
	Our proof uses concepts that are broadly similar to that of Theorem 10 of \ctn{Chandra20}.
	Here we shall provide the proof for $\frac{1}{n}\log R^{(1)}_n(\theta)$ since that for $\frac{1}{n}\log R^{(2)}_n(\theta)$ is exactly the same. For notational
	convenience, we denote $\frac{1}{n}\log R^{(1)}_n(\theta)$ by $\frac{1}{n}\log R_n(\theta)$, $h_1(\theta)$ by $h(\theta)$, $\tilde\theta_1$ by $\tilde\theta$
	and $\Theta_1$ by $\Theta$. 

	Since $h(\theta)$ is convex, $\tilde\theta$ must be an interior point of $\Theta$. Hence, there exists a compact set $G\subset\Theta$ such that
	$\tilde\theta$ is interior to $G$.
	From convergence (\ref{eq:pointconv1}) which is also uniform on compact sets, it follows that
	\begin{equation}
	\lim_{n\rightarrow\infty}\sup_{\theta\in G}~\left| \frac{1}{n}\log R_n(\theta)  + h(\theta)\right|=0.
	\label{eq:unifG}
	\end{equation}
	For any $\eta>0$, we define 
	\[
	N_\eta(\tilde\theta)=\{\theta:\|\tilde\theta-\theta\|<\eta\};~N'_\eta(\tilde\theta)=\{\theta:\|\tilde\theta-\theta\|=\eta\};~\overline{N}_\eta(\tilde\theta)
	=\{\theta:\|\tilde\theta-\theta\|\leq\eta\}.
	\]
	Note that for sufficiently small $\eta$, $\overline{N}_\eta(\tilde\theta)\subset G$. Let $H=\underset{\theta\in N'_\eta(\tilde\theta)}{\inf}~h(\theta)$. 
	Since $h(\theta)$ is minimum at $\theta=\tilde\theta$, $H>0$. Let us fix an $\varepsilon$ such that $0<\varepsilon<H$. 
	Then by \eqref{eq:unifG}, for large enough $n$ all $\theta\in N'_\eta(\tilde\theta)$,
	\begin{equation}
	\frac{1}{n}\log R_n(\theta)  <- h(\theta)+\varepsilon<-h(\tilde\theta)+\varepsilon.
		\label{eq:r1}
	\end{equation}
	Since by (\ref{eq:pointconv1}) $\frac{1}{n}\log R_n(\tilde\theta)>-h(\tilde\theta)-\varepsilon$ for sufficiently large $n$, it follows from this and (\ref{eq:r1})
	that
	\begin{equation}
		\frac{1}{n}\log R_n(\theta)  < \frac{1}{n}\log R_n(\tilde\theta)+2\varepsilon,
		\label{eq:r2}
	\end{equation}
	for sufficiently large $n$. Since $0<\varepsilon<H$ is arbitrary, it follows that for all $\theta\in N'_\eta(\tilde\theta)$, for large enough $n$,
	\begin{equation}
		\frac{1}{n}\log R_n(\theta)  < \frac{1}{n}\log R_n(\tilde\theta),
		\label{eq:r3}
	\end{equation}
	which shows that for large enough $n$, the maximum of $\frac{1}{n}\log R_n(\theta)$ is not attained at the boundary $N'_\eta(\tilde\theta)$. 
	Hence, the maximum must occur in the interior of $\overline{N}_\eta(\tilde\theta)$ when $n$ is sufficiently large. That the maximizer is unique
	is guaranteed by Theorem \ref{th:concave}.
	Hence, the result is proved.

\end{appendix}

\bibliography{irmcmc}

\end{document}